\documentclass[10pt,a4paper,american]{amsart}
\usepackage{babel}
\usepackage{amsmath,amsthm,bm}
\usepackage{amsbsy}
\usepackage{amstext}
\usepackage{amsfonts}
\usepackage{floatflt}
\usepackage[pdftex]{graphicx}
\usepackage{mathcomp}
\usepackage{mathtools}
\usepackage{color}
\usepackage{dsfont}
\usepackage{latexsym}
\usepackage{amssymb}
\usepackage{stmaryrd}
\usepackage{bm}
\usepackage{enumerate}
\usepackage{esint}	
\usepackage[colorlinks,pdfpagelabels,pdfstartview = FitH,bookmarksopen = true,bookmarksnumbered = true,linkcolor = blue,plainpages = false,hypertexnames = false,citecolor = red,pagebackref=false]{hyperref}
\usepackage{cite}

\allowdisplaybreaks
\makeatletter
\newsavebox{\@brx}
\newcommand{\llangle}[1][]{\savebox{\@brx}{\(\m@th{#1\langle}\)}%
  \mathopen{\copy\@brx\kern-0.5\wd\@brx\usebox{\@brx}}}
\newcommand{\rrangle}[1][]{\savebox{\@brx}{\(\m@th{#1\rangle}\)}%
  \mathclose{\copy\@brx\kern-0.5\wd\@brx\usebox{\@brx}}}
\makeatother

\newtheorem{theorem}{Theorem}
\newtheorem{lemma}[theorem]{Lemma}
\newtheorem{proposition}[theorem]{Proposition}
\newtheorem{corollary}[theorem]{Corollary}

\theoremstyle{definition}
\newtheorem{definition}[theorem]{Definition}

\theoremstyle{remark}
\newtheorem{remark}[theorem]{Remark}

\numberwithin{theorem}{section}
\numberwithin{equation}{section}

\newtheorem{assumption}[theorem]{Assumption}

\renewcommand{\div}{\operatorname{div}}
\DeclareMathOperator{\Div}{div}

\DeclareMathOperator*{\essliminf}{ess\,lim\,inf}
\DeclareMathOperator*{\esssup}{ess\,sup}

\newcommand{\N}{\ensuremath{\mathbb{N}}}
\newcommand{\R}{\ensuremath{\mathbb{R}}}
\newcommand{\mint}{- \mskip-19,5mu \int}

\newcommand{\spt}{\operatorname{spt}}

\newcommand{\ca}{\operatorname{cap}}

\newcommand{\dx}{\mathrm{d}x}
\newcommand{\dt}{\mathrm{d}t}

\makeatletter
\@namedef{subjclassname@2020}{%
  \textup{2020} Mathematics Subject Classification}
\makeatother

\newcommand{\wto}{\rightharpoondown}
\newcommand{\wsto}{\overset{\raisebox{-1ex}{\scriptsize $*$}}{\rightharpoondown}}

\def\Xint#1{\mathchoice
    {\XXint\displaystyle\textstyle{#1}}%
    {\XXint\textstyle\scriptstyle{#1}}%
    {\XXint\scriptstyle\scriptscriptstyle{#1}}%
    {\XXint\scriptscriptstyle\scriptscriptstyle{#1}}%
    \!\int}
\def\XXint#1#2#3{\setbox0=\hbox{$#1{#2#3}{\int}$}
    \vcenter{\hbox{$#2#3$}}\kern-0.5\wd0}
\def\bint{\Xint-}
\def\dashint{\Xint{\raise4pt\hbox to7pt{\hrulefill}}}

\def\Xiint#1{\mathchoice
    {\XXiint\displaystyle\textstyle{#1}}%
    {\XXiint\textstyle\scriptstyle{#1}}%
    {\XXiint\scriptstyle\scriptscriptstyle{#1}}%
    {\XXiint\scriptscriptstyle\scriptscriptstyle{#1}}%
    \!\iint}
\def\XXiint#1#2#3{\setbox0=\hbox{$#1{#2#3}{\iint}$}
    \vcenter{\hbox{$#2#3$}}\kern-0.5\wd0}
\def\biint{\Xiint{-\!-}}

\renewcommand{\epsilon}{\varepsilon}
\newcommand{\eps}{\varepsilon}
\renewcommand{\rho}{\varrho}

\renewcommand{\epsilon}{\varepsilon}
\renewcommand{\rho}{\varrho}
\renewcommand{\d}{\:\! \mathrm{d}}

\DeclareMathOperator{\loc}{loc}

\numberwithin{equation}{section}
\allowdisplaybreaks

\begin{document}
\renewcommand{\refname}{References} 
\renewcommand{\abstractname}{Abstract} 
\title[$p$-parabolic capacity]{On notions of $p$-parabolic capacity and applications}

\author[K.~Moring]{Kristian Moring}
\address{Kristian Moring\\
	Fachbereich Mathematik, Paris-Lodron-Universität Salzburg\\
	Hellbrunner Str.~34, 5020 Salzburg, Austria}
\email{kristian.moring@plus.ac.at}

\author[C.~Scheven]{Christoph Scheven}
\address{Christoph Scheven\\
	Fakult\"at f\"ur Mathematik, Universit\"at Duisburg-Essen\\
	Thea-Leymann-Str.~9, 45127 Essen, Germany}
\email{christoph.scheven@uni-due.de}


\begin{abstract}
We consider different notions of capacity related to the parabolic $p$-Laplace equation. Our focus is on a variational notion, which is consistent in the full range $1<p<\infty$. For such a notion we show some basic properties as well as its connection to other notions of capacity presented in the literature, and to a certain parabolic version of the Hausdorff measure. As applications, we use the introduced variational notion of capacity to study polar sets and removability results for supersolutions.
\end{abstract}

\makeatother

\maketitle

\section{Introduction}

In this paper, we consider notions of capacity and some applications for the parabolic $p$-Laplace equation
\begin{equation} \label{eq:parabolic-p-laplace}
\partial_t u - \Div ( |\nabla u|^{p-2} \nabla u ) = 0,
\end{equation}
where $1<p<\infty$. In potential theory, capacity is a central
concept. It is used for example in results related to boundary
regularity, removability and characterization of polar sets, see
e.g. the monographs \cite{HKM,watson-book} for an overview of the
theory.

Our aim in this paper is twofold. First, our objective is to present a
notion of capacity, which makes sense in the whole range $1<p<\infty$,
and also reflects the anisotropic behavior of
equation~\eqref{eq:parabolic-p-laplace}. Second, our goal is to show
that the notion of capacity satisfies some useful properties that
enable its use in applications. The beginning of our paper is devoted
to the properties of the variational capacity and its relation to the
notion of capacity introduced in~\cite{KKKP}. The latter part
demonstrates that the capacity we use is useful in potential theoretic
applications, and as examples we apply it to study polar sets and
removability results for supercaloric functions.

\subsection*{Notions of capacity used in the literature}
In the elliptic case, the notion of capacity is directly related to
the appropriate Sobolev space, and the properties are well known, see
e.g.~\cite{HKM,EG}. For
the heat equation and its generalizations ($p= 2$), different notions
for the capacity have been proposed, see
e.g.~\cite{lanconelli,watson,Pierre}. In view of function spaces, it
seems natural to consider functions that lie in the underlying
parabolic Sobolev space, and in addition, whose time derivative belong
to the corresponding dual space. This approach leads to the so-called
variational capacity, see~\cite{Pierre}. In the same article, the
author also considered a notion of capacity defined via a measure data
problem, as in Definition~\ref{def:nonlinear-capacity-compact}, and
showed that both concepts of capacity are equivalent in the case of
the heat equation. 

In the case of nonlinear parabolic equations of $p$-Laplace type, the authors
in~\cite{KKKP} used the latter approach and
defined the capacity of a set via measure data
problems, see Definition~\ref{def:nonlinear-capacity-compact}. Using
the function space setting, different variants of variational capacity have been introduced in~\cite{droniou} when $1<p<\infty$ (see also~\cite{Saraiva1,Saraiva2}), and in~\cite{AKP} when $2<p<\infty$. In both papers~\cite{AKP} and~\cite{droniou}, the function space is given by $\big\{ v \in L^p(0,T;V) : \partial_t v \in L^{p'}(0,T;V') \big\}$, in which $V = W^{1,p}_0(\Omega) \cap L^2(\Omega)$, where the latter space in $V$ can be omitted in case $\frac{2n}{n+2}<p<\infty$. However, there are also certain differences. On the one hand, the minimized quantity in~\cite{droniou} is given by the actual norm of the space, while in~\cite{AKP}, the norm is modified to respect the anisotropic behavior of equation~\eqref{eq:parabolic-p-laplace}. On the other hand, the authors in~\cite{AKP} use an intrinsic scaling in the definition of the capacity whenever $T < \infty$. Due to the intrinsic scaling and relying on the restriction $p>2$, the authors in~\cite{AKP} were able to show that their notion of variational capacity is equivalent with the notion of capacity introduced in~\cite{KKKP}, at least for compact sets.

\subsection*{The definition of (variational) capacity used in the
  present work}

We consider the space
$$
\mathcal{W}(\Omega_T) = \left\{ v \in L^p(0,T;W^{1,p}_0(\Omega)) \cap L^\infty(0,T;L^2(\Omega)) : \partial_t v \in L^{p'}(0,T;W^{-1,p'}(\Omega)) \right\}.
$$
In our notion of variational capacity, we denote the quantity to be minimized by
\begin{align*}
\|v\|_{\mathcal{W}(\Omega_T)} &= \iint_{\Omega_T} |\nabla v|^p \, \d x \d t + \esssup_{0<t<T} \int_\Omega v^2 \, \d x \\
&\quad + \sup_{\substack{ \phi \in C^\infty_0(\Omega_T) \\ \|\nabla \phi\|_{L^p(\Omega_T)} \leq 1}} \left| \iint_{\Omega_T} v \partial_t \phi \, \d x \d t \right|^{p'}.
\end{align*}
There is some freedom for choosing the precise definition of the variational capacity, provided that the base $\Omega$ of the reference set is bounded and satisfies rather mild regularity assumptions. As it was already noted in~\cite{droniou}, it is possible to use functions in $C_0^\infty(\Omega \times \R)$ or in $\mathcal{W}(\Omega_T)$ as a starting point for variational capacity.

In the present paper, we choose to work mainly with smooth functions, and define the capacity for a compact set $K \subset \Omega_T$ by
$$
\ca_{\mathrm{var}} (K, \Omega_T) = \inf \big\{ \|v\|_{\mathcal{W}(\Omega_T)} : v \in C_0^\infty(\Omega \times \R ),\, v \geq \chi_K \big\}.
$$
However, we note that in some cases notions with functions in
$\mathcal{W}(\Omega_T)$ seem to be more practical, and we show in
Appendix~\ref{appendix} that they give equivalent notions of capacity under mild regularity assumptions on $\Omega$.

Compared to the notion introduced in~\cite{AKP}, the advantage is that
no intrinsic scaling is used in our notion, such that the same notion
works in the full range $1<p<\infty$. Furthermore, the values of $T$
do not play any important role. On the other hand, the notion
in~\cite{KKKP} is based on a definition in which admissible objects
are Radon measures satisfying certain conditions, see
Definition~\ref{def:nonlinear-capacity-compact}. With this definition
an advantage is that the capacity of a compact set can be
represented via a capacitary potential, see Lemma~\ref{lem:cap-compact-representation}. However, it seems difficult to estimate capacities of some explicit sets using this definition, and to obtain some useful properties especially in the case when $p$ is close to $1$.

In Section~\ref{sec:prelim}, we gather some known results and show
some basic properties for the variational capacity defined
above. These include downward and upward monotone type convergences,
see Lemmas~\ref{lem:capvar-limit-compact}
and~\ref{lem:capvar-upwards-conv}, as well as a variant of countable
subadditivity in Lemma~\ref{lem:subadditivity}. In addition, we show
in Lemma~\ref{lem:capvar-Tinfty2} that the choice of $0<T\leq \infty$ in the reference set does not play an essential role in the notion of variational capacity.

\subsection*{Equivalence of different notions of capacities}

We are able to show that for compact sets, the notion of capacity in~\cite{KKKP} is equivalent with our variational capacity.

\begin{theorem} \label{thm:capvar-eq-cap}
Let $1<p<\infty$, $0<T\leq \infty$ and suppose that $\Omega \subset \R^n$ is a bounded open set such that $\R^n \setminus \Omega$ is uniformly $p$-fat according to Definition~\ref{def:p-fat}. Let $\ca$ and $\ca_{\mathrm{var}}$ be the notions of capacity according to Definitions~\ref{def:nonlinear-capacity-compact} and~\ref{def:varcap}, respectively.  If $K \subset \Omega_T$ is a compact set, then
$$
\ca_{\mathrm{var}} (K,\Omega_T) \approx \ca (K,\Omega_\infty)
$$
up to a positive constant depending only on $n,p$ and the parameter
$\alpha$ from Definition~\ref{def:p-fat} of $p$-fatness.
\end{theorem}

The proof is given in Section~\ref{subseq:capvar-cap-eq}.
The first step in the proof is to show that the variational capacity and the
so-called energy capacity (see Definition~\ref{def:encap}) are
equivalent for compact sets of specific type, and applying the known
equivalency of energy capacity and the notion in~\cite{KKKP}
from~\cite{AKP}. We prove the former result by showing that for each
admissible function for the energy capacity, there exists an
admissible function for the variational capacity, such that the
respective norms are ordered in a suitable direction, and vice
versa. These results are stated in Lemmas~\ref{lem:W-leq-en}
and~\ref{lem:en-leq-W}. This type of strategy was used in~\cite{AKP}
for the intrinsically scaled variational capacity ($p> 2$), and already in~\cite{Pierre} in the linear case ($p=2$).

Theorem~\ref{thm:capvar-eq-cap} also implies that the variational
capacity of a compact set can be expressed via a capacitary potential,
since this is known for the nonlinear parabolic capacity
from~\cite{KKKP}, see Lemma~\ref{lem:cap-compact-representation}. This will be a useful connection in the following applications.

In Section~\ref{sec:cap-sets}, we collect and prove some useful
estimates and characterizations for the capacity of specific
sets. Some of these results are also described in~\cite{AKP,droniou}
with their notions of capacity. Especially, we obtain a formula, up to
a constant, for the capacity of space-time cylinders. For a cylinder
of the form
$$
Q_{\rho,\tau} = Q_{\rho,\tau}(z_o) = B_\rho(x_o) \times (t_o - \tau,
t_o)\subset\Omega_T,
$$
we show that
\begin{equation} \label{eq:capvar-cylinder}
\ca_{\mathrm{var}} (Q_{\rho,\tau},\Omega_T) \approx \rho^n + \tau \rho^{n-p},
\end{equation}
provided $1<p<n$, cf. Lemma~\ref{lem:capvar-cylinder}.

\subsection*{Capacity and Hausdorff measure}

In the elliptic case, it is well known that the capacity of a set is
related to its Hausdorff measure, see e.g.~\cite{HKM,EG}. In the
parabolic setting, it is natural to consider a Hausdorff measure
defined via a parabolic metric, whose unit balls are parabolic cylinders. Property~\eqref{eq:capvar-cylinder} suggests to use
cylinders of the form $Q_{\rho,\rho^p}$, since their capacity is comparable to $\rho^n$. The corresponding parabolic metric is defined by
\begin{equation} 
d_p\big((x,t),(0,0)\big) = \max \big\{ |x|, |t|^\frac{1}{p} \big\},
\end{equation}
which has already been suggested in~\cite{AS}. Using this metric we
define the $s$-dimensional Hausdorff measure $\mathcal{P}^s$, see
Section~\ref{sec:cap-hausdorff}. For the connection between the
Hausdorff measures $\mathcal{P}^s$ and the variational capacity, we
obtain the following results, see
Propositions~\ref{prop:cap-hausdorff} and \ref{prop:cap-leq-Pn}.

\begin{theorem} \label{thm:hausdorff-capvar}
Let $1<p<\infty$, $0<T\leq \infty$ and suppose that $\Omega \subset \R^n$ is a bounded open set such that $\R^n \setminus \Omega$ is uniformly $p$-fat according to Definition~\ref{def:p-fat}. Suppose that $E \subset \Omega_T$ is a set. 
\begin{enumerate}
\item If $\ca_{\mathrm{var}} (E,\Omega_T) = 0$, then $\mathcal{P}^s
  (E) = 0$ for all $s > n$. \label{eq:hausdorff-capvar1}
\item If $\mathcal{P}^n(E) = 0$, then $\ca_{\mathrm{var}}(E, \Omega_T) = 0$. \label{eq:hausdoff-capvar2}
\end{enumerate}
\end{theorem}

Item~\eqref{eq:hausdorff-capvar1} in Theorem~\ref{thm:hausdorff-capvar} is proved with a similar strategy as in the elliptic case in~\cite{EG}. The proof applies a Poincar\'e type inequality, which is standard in the elliptic setting. In the parabolic case, we apply the elliptic Poincar\'e inequality on time slices, together with an application of a suitable gluing lemma, Lemma~\ref{lem:gluing}. The application of the gluing lemma is based on the fact that for any function $u \in \mathcal{W}(\Omega_T)$, we have the well known characterization $\partial_t u = \Div F$ for some $F \in L^{p'}(\Omega_T, \R^n)$.

On the other hand, we prove item~\eqref{eq:hausdoff-capvar2} in Theorem~\ref{thm:hausdorff-capvar} by showing that the capacity of a set can be bounded by the $n$-dimensional Hausdorff measure $\mathcal{P}^n$ of the set. In the proof, we apply a bound of the type~\eqref{eq:capvar-cylinder} with $\lesssim$ for a countable union of space time cylinders, which we prove in Lemma~\ref{lem:cap-union-cylinder-upperbound}.

\subsection*{Applications}

The concept of capacity arises naturally in the study of polar sets
and removability properties of supersolutions. For supersolutions, we
use two related notions, weak supersolutions and supercaloric
functions, see Section~\ref{subsec:supersol-supercal}. Here we mainly
focus on the properties of supercaloric functions, which are defined
as lower semicontinuous functions obeying a parabolic comparison principle, see Definition~\ref{d.supercal}.

We show that the polar set of a supercaloric function is a set of null capacity, provided that the function is integrable to a sufficiently large power. In the case $\tfrac{2n}{n+1} < p< \infty$, such an integrability condition characterizes the so called Barenblatt class, see~\cite{GKM,KuLiPa}.

\begin{theorem} \label{thm:supercal-polarset}
Let $1 < p < \infty$, $0<T\leq \infty$ and suppose that $\Omega \subset \R^n$ is a bounded open set such that $\R^n \setminus \Omega$ is uniformly $p$-fat according to Definition~\ref{def:p-fat}. Suppose that $u$ is a supercaloric function in $\Omega_T$ that satisfies $u \in L^s_{\loc}(\Omega_T)$ for sufficiently large $s$, specified in Assumption~\ref{as:barenblatt}. Then,
$$
\ca_{\mathrm{var}} (\{u=\infty\},\Omega_T) = 0.
$$
\end{theorem}
For the proof we refer to Section~\ref{sec:polarsets}. 
The main problem in the proof is to obtain a suitable estimate for
capacity of the superlevel sets of the supercaloric function, which is
shown in Lemma~\ref{lem:supercal-cap-levelset} and
Corollary~\ref{cor:signed-supercal-cap-levelset}. The proof of
Lemma~\ref{lem:supercal-cap-levelset} essentially relies on three main
properties. First, we use that the capacity of a compact set can be
represented via a capacitary potential by
Lemma~\ref{lem:cap-compact-representation} (see~\cite{KKKP}) and
Theorem~\ref{thm:capvar-eq-cap}. Second, we use the fact that a
sufficiently integrable supercaloric function has a locally finite
Riesz measure, which is guaranteed in the subcritical case
$1<p\le\frac{2n}{n+1}$ by Lemma~\ref{lem:supercal-L1}. Finally, we
show that a Poisson modification exists for a supercaloric function
satisfying Assumption~\ref{as:barenblatt}, i.e., that it can be
replaced by a locally bounded weak solution in any cylinder. This is
guaranteed by sup-estimates for weak solutions, see
Lemma~\ref{lem:blowup}. Then, the proof of
Lemma~\ref{lem:supercal-cap-levelset} can be concluded with the help of suitable estimates for Riesz measures of supersolutions given in~\cite{KKKP}.

In Lemma~\ref{lem:monster-polarset} we show that in case $p> 2$ the integrability assumption in Theorem~\ref{thm:supercal-polarset} is sharp. Whether the assumption in the case $1<p<2$ can be relaxed is an interesting open problem.

Finally, we show that relatively closed sets of null capacity are
removable, see Section~\ref{sec:removability}.

\begin{theorem} \label{thm:supercal-removability}
Let $1<p<\infty$, $0<T\leq \infty$ and suppose that $\Omega \subset \R^n$ is a bounded open set such that $\R^n \setminus \Omega$ is uniformly $p$-fat according to Definition~\ref{def:p-fat}. Let $E \subset \Omega_T$ be a relatively closed set and $u$ be a supercaloric function in $ \Omega_T \setminus E$. If $\ca_{\mathrm{var}} (E,\Omega_T) = 0$ and $\essliminf_{\Omega_T \setminus E \ni (y,s) \to (x,t)} u(y,s) > - \infty$ for every $(x,t) \in E$, then there exists a supercaloric function $v$ in $\Omega_T$ such that $v = u$ everywhere in $\Omega_T\setminus E$ and $v$ is given by
\begin{equation} \label{eq:supercal-u-essliminfextension}
v(x,t) = \essliminf_{\Omega_T \setminus E \ni (y,s) \to (x,t)} u(x,y).
\end{equation}
\end{theorem}

An analogous result with appropriate assumptions also holds for weak
supersolutions and weak solutions, see
Lemma~\ref{lem:supersol-removability} and
Corollary~\ref{cor:weak-sol-removability}. A crucial part of the proof
is to show the removability result for weak supersolutions in
Lemma~\ref{lem:supersol-removability}, which uses ideas in~\cite{AS}. We assume that the weak supersolution is
essentially bounded in a neighborhood of $E$, and show that for any
point in $E$, we can find a cylinder in
which~\eqref{eq:supercal-u-essliminfextension} is a weak supersolution
in the very cylinder, which is proved in
Lemma~\ref{lem:supersol-removability}. The proof relies on
Caccioppoli's inequality for weak supersolutions, as well as an
approximation result which allows to approximate functions
in $C_0^\infty(O)$ by functions in $C_0^\infty(O \setminus E)$ in a
suitable norm, for an open set $O \subset \Omega_T$ and certain sets
$E$ with vanishing capacity.
Since the space $\mathcal{W}(\Omega_T)$ is not closed under
truncations, we need to approximate the functions in a larger space
that contains functions with time derivatives in
$\mathcal{V}'(\Omega_T)+L^1(\Omega_T)$,
see Lemma~\ref{lem:S-closure}. This space is natural in the context of
bounded solutions and has proven to be very useful in several previous
works, see \cite{AS,droniou,Petitta}.

Finally, a comment on the case $1<p \leq \tfrac{2n}{n+2}$ is in order,
especially when the latter inequality is also strict. Our definition
of supercaloric functions exclude the possibility of attaining the
value $-\infty$ in the domain. In other words, supercaloric functions
are locally bounded from below by
Definition~\ref{d.supercal}. Therefore, our arguments do not apply to
solutions attaining $-\infty$ as in~\cite[Chapter 6, Eq. (1.10)]{DGV}.

\medskip

\noindent
{\bf Acknowledgments.} K.~Moring has been supported by the Magnus Ehrnrooth Foundation.

\section{Preliminaries} \label{sec:prelim}

Let $\Omega \subset \R^n$ be a bounded open set. For $T>0$, we denote by $\Omega_T := \Omega \times (0,T)$ a space-time cylinder in $\R^{n+1}$. We denote the lateral boundary of $\Omega_T$ by $S_T:= \partial \Omega \times [0,T)$, and the parabolic boundary by $\partial_p \Omega_T := \left( \Omega \times \{0\} \right) \cup S_T$. We consider the parabolic $p$-Laplace equation written as
\begin{equation}\label{p-lap}
\partial_t u -\Div \left(|\nabla u|^{p-2} \nabla u \right)=0\quad \text{ in } \Omega_T,
\end{equation}  
for $1<p<\infty$.

\subsection{Weak supersolutions and supercaloric functions} \label{subsec:supersol-supercal}

First we define the notion of a weak supersolution.

\begin{definition}
A function $u \in L^{p}_{\loc}(0,T;W^{1,p}_{\loc}(\Omega))$ is called a weak supersolution to~\eqref{p-lap} if 
\begin{equation} \label{eq:weak-super}
\iint_{\Omega_T} - u \partial_t \varphi + |\nabla u|^{p-2} \nabla u \cdot \nabla \varphi \, \d x \d t \geq 0
\end{equation}
for every $\varphi \in C_0^\infty(\Omega_T, \R_{\geq 0})$. If the integral is nonpositive for all such test functions we call $u$ a weak subsolution. A function is called a weak solution if it is both weak super- and subsolution.

For a noncylindrical open set $U\subset\R^{n+1}$, 
we say that $u$ is a weak super(sub)solution in $U$ if it is a weak super(sub)solution in every open subcylinder $V \times (t_1,t_2) \subset U$.
\end{definition}

By application of the Riesz representation theorem~\cite[Section 1.8, Theorem 1]{EG}, we have that for every weak supersolution $u$ in $\Omega_T$ there exists a Radon measure $\mu$ on $\Omega_T$ such that 
\begin{equation} \label{e.weak_measuresol}
\iint_{\Omega_T} - u \partial_t \varphi + |\nabla u|^{p-2} \nabla u \cdot \nabla \varphi \, \d x \d t = \int_{\Omega_T} \varphi \, \d \mu
\end{equation}
holds for every $\varphi \in C_0^\infty(\Omega_T)$. The measure $\mu$
is called the Riesz measure of the weak supersolution $u$ and denoted
by $\mu_u$. 

The following result states that under appropriate assumptions,
pointwise convergence of weak supersolutions implies convergence of the corresponding Riesz measures, see e.g.~\cite{KKKP,GKM,KKP,KL}.

\begin{lemma} \label{lem:bdd-weak-sol-convergence}
Let $1<p<\infty$ and $(u_i)_{i \in \N}$ be a sequence of uniformly
bounded weak supersolutions in $\Omega_T$ and suppose that $u_i\to u$ a.e. in $\Omega_T$ in the limit $i\to\infty$. Then $u$ is a weak
supersolution and the Riesz measures converge $\mu_{u_i}\wto \mu_u$
weakly in the space of Radon measures, as $i\to\infty$.  
\end{lemma}

We recall the following Caccioppoli inequalities for nonnegative weak
supersolutions. The first inequality is analogous to~\cite[Lemma
9]{KuLiPa},~\cite[Lemma 3.2]{GKM}. The second inequality follows by a
slight modification of the proof. 

\begin{lemma} \label{lem:caccioppoli}
Let $u$ be a nonnegative weak supersolution in $\Omega_T$ and $\eps \in (0,1)$. If $1<p<\infty$, there exists $c = c(p) >0$ such that
\begin{align} \label{eq:caccioppoli-sup}
\iint_{\Omega_T} &u^{- 1-\eps } |\nabla u|^p \varphi^p \, \d x \d t + \frac{1}{1-\eps }\esssup_{0<t<T} \int_{\Omega \times \{t\}} u^{1-\eps} \varphi^p \, \d x \nonumber \\
&\leq c \eps^{-p} \iint_{\Omega_T} u^{p-1-\eps} |\nabla \varphi|^p \, \d x\d t + \frac{2}{\eps(1-\eps)} \iint_{\Omega_T} u^{1-\eps} |\partial_t \varphi^p| \, \d x \d t 
\end{align}
holds true for all $\varphi \in C_0^\infty(\Omega_T, \R_{\geq 0})$.

If $1<p<2$, then there exists $c = c(p) > 0$ such that
\begin{align} \label{eq:caccioppoli-p'}
\iint_{\Omega_T} u^{-1-\eps} |\nabla u|^p \varphi^{p'} \, \d x\d t &\leq c \eps^{-p} \iint_{\Omega_T} u^{p-1-\eps} \varphi^{p(p'-2)} |\nabla \varphi|^p\, \d x \d t  \\
&\quad + \frac{2}{\eps(1-\eps)} \left| \iint_{\Omega_T} u^{1-\eps} \partial_t \varphi^{p'} \, \d x\d t  \right| \nonumber
\end{align}
holds true for all $\varphi \in C_0^\infty(\Omega_T, \R_{\geq 0})$.
\end{lemma}

\begin{proof}

The first inequality is proved by using formally a test function $u^{-\eps} \varphi^p$, and the second by using a test function $u^{- \eps} \varphi^{p'}$.

\end{proof}

A weak supersolution has a lower semicontinuous representative, provided that it is locally essentially bounded from below, see~\cite{K-lowersc,Naian}. 

\begin{lemma} \label{lem:supersol-lsc}
Let $1<p<\infty$ and $u$ be a weak supersolution to~\eqref{p-lap} in an open set $U \subset \R^{n+1}$ such that $u$ is locally essentially bounded from below in $U$. Then, there exists a lower semicontinuous
function $u_*$ such that $u_*(x,t) = u(x,t)$ for a.e. $(x,t) \in
U$. Moreover, 
$$
u_*(x,t) = \essliminf_{\substack{(y,s) \to (x,t) \\ s<t}} u(y,s),
$$
for every $(x,t) \in U$.
\end{lemma}

\begin{remark} \label{rem:supersol-bddbelow}
Observe that in case $p > \tfrac{2n}{n+2}$ a weak supersolution is
locally essentially bounded from below by~\cite[Corollary
3.7]{K-lowersc}. In case $1<p \leq \tfrac{2n}{n+2}$ a sufficient
condition can be expressed for example via the local integrability condition given in Assumption~\ref{as:barenblatt}, see e.g.~\cite[Theorem 2]{Choe}.
\end{remark}

Then, we define the notion of a supercaloric function.

\begin{definition} \label{d.supercal}
  Let $U\subset\R^{n+1}$ be an open set.
  A function $u \colon U \to (-\infty,\infty]$ is called a supercaloric function, if
\begin{itemize}
\item[(i)] $u$ is lower semicontinuous,
\item[(ii)] $u$ is finite in a dense subset of $U$,
\item[(iii)] $u$ satisfies the comparison principle in every subcylinder $Q_{t_1,t_2}=Q \times (t_1,t_2) \Subset U$: if $h\in C(\overline{Q}_{t_1,t_2})$ is a weak solution in $Q_{t_1,t_2}$ and $h \leq u$ on the parabolic boundary of $Q_{t_1,t_2}$, then $ h\leq u$ in $Q_{t_1,t_2}$.
\end{itemize}
\end{definition}

The next result states that in the fast diffusion case $1<p<2$, for a fixed time slice a nonnegative supercaloric function is either positive on a whole slice, or vanishes on the whole slice, provided that the domain is connected. The proof follows as in~\cite{MS} by an application of expansion of positivity, see~\cite[Chapter 4, Proposition 5.1]{DGV}.

\begin{lemma} \label{lem:supercal-slice-alt}
Let $1<p<2$ and assume that $u$ is a nonnegative supercaloric function in $\Omega_T$, where $\Omega \subset \R^n$ is open and connected. Then, for any $t \in (0,T)$ either $u$ is positive on the whole time slice $\Omega \times \{t\}$ or $u$ vanishes on the whole time slice.
\end{lemma}

For some results related to the boundary regularity, we need the concept of elliptic ($p$-)capacity, which we define as follows.

\begin{definition} \label{def:elliptic-cap}
For a compact set $K \subset \Omega$, we define the elliptic $p$-capacity as 
$$
\ca_\mathrm{e}(K,\Omega) = \inf \left\{ \int_\Omega  |\nabla u|^p \, \d x: u\in C_0^\infty(\Omega), \, u(x) \geq 1 \text{ for every } x\in K \right\}.
$$
For an open set $U \subset \Omega$ we define the elliptic  $p$-capacity by
$$
\ca_{\mathrm{e}}(U,\Omega) = \sup \left \{\ca_{\mathrm{e}}(K,\Omega): K \text{ compact subset of } \Omega, K \subset U  \right\},
$$
and for an arbitrary set $E \subset \Omega$ by
$$
\ca_{\mathrm{e}}(E,\Omega) = \inf \left \{\ca_{\mathrm{e}}(U,\Omega): U \text{ open subset of } \Omega, E \subset U  \right\}.
$$
\end{definition}

\begin{definition} \label{def:p-fat}
We say that a set $E \subset \R^n$ is uniformly $p$-fat, if there exists a constant $\alpha > 0$ such that
$$
\ca_\mathrm{e}\big(E \cap \overline{B_\rho(x)}, B_{2\rho}(x)\big) \geq \alpha \ca_\mathrm{e}\big(\overline{B_\rho(x)}, B_{2\rho}(x)\big) 
$$
holds true for every $x \in E$ and $\rho > 0$.
\end{definition}

Sobolev functions with zero boundary values satisfy Hardy's inequality, provided that the complement of the domain is uniformly $p$-fat, see e.g.~\cite{lewis,KM}.

\begin{lemma}\label{lem:hardy}
Let $1<p<\infty$ and suppose that $\Omega \subset \R^n$ is a bounded open set such that $\R^n \setminus \Omega$ is uniformly $p$-fat. If $u \in W^{1,p}_0(\Omega)$, then there exists $c = c(n,p,\alpha)> 0$ such that
$$
\int_{\Omega} \left( \frac{|u|}{\mathrm{dist}(x,\partial \Omega) } \right)^p \d x \leq c \int_{\Omega} |\nabla u|^p \, \d x.
$$
\end{lemma}

Next, we show that parallel sets inherit the $p$-fatness property.

\begin{lemma}\label{lem:parallel-fatness}
 Let $1<p<\infty$ and 
  assume that $\Omega\subset\R^n$ is a bounded open set such that
  $\R^n\setminus\Omega$ is uniformly $p$-fat for some parameter
  $\alpha>0$. For $\delta>0$, we consider the inner parallel set
  $\Omega_\delta:= \{x\in\Omega\colon
  \mathrm{dist}(x,\R^n\setminus\Omega)>\delta\}$. Then the complement
  $\R^n\setminus\Omega_\delta$ is uniformly $p$-fat for a parameter
  $\alpha_o=\alpha_o(n,p,\alpha)$.   
\end{lemma}

\begin{proof}
  We fix a point $x\in
  \R^n\setminus\Omega_\delta$ and $\rho > 0$. By definition of $\Omega_\delta$ we can find a point $x_o\in \R^n\setminus\Omega$ with
  $|x-x_o|\le\delta$. We distinguish between the two cases
  $\rho\le2|x-x_o|$ and $\rho>2|x-x_o|$.

  In the first case, we define
  $z:=x+\frac\rho2\frac{x_o-x}{|x_o-x|}$ and observe that for
      $y\in\overline {B_{\rho/2}(z)}$, we
      have $$\mathrm{dist}(y,\R^n\setminus\Omega)\le|y-x_o|\le
      |y-z|+|z-x_o|\le \frac\rho2+|x-x_o|-\frac\rho2\le\delta,$$
    which means $y\in\R^n\setminus\Omega_\delta$. Thus
  \begin{equation*}
    \overline {B_{\rho/2}(z)}\subset (\R^n\setminus\Omega_\delta)\cap \overline{B_\rho(x)}.
  \end{equation*}
  Using first this and then the inclusion $B_{2\rho}(x)\subset
  B_{3\rho}(z)$, we estimate
  \begin{align*}
    \ca_\mathrm{e}\big((\R^n\setminus\Omega_\delta) \cap \overline{B_\rho(x)}, B_{2\rho}(x)\big)
    &\geq
    \ca_\mathrm{e}\big(\overline {B_{\rho/2}(z)}, B_{2\rho}(x)\big)\\
    &\geq
      \ca_\mathrm{e}\big(\overline {B_{\rho/2}(z)}, B_{3\rho}(z)\big)
    =
      c(n,p)\rho^{n-p}\\
    &=
      c(n,p) \ca_\mathrm{e}\big(\overline{B_\rho(z)}, B_{2\rho}(z)\big) .
  \end{align*}
  The last two identites are well-known facts for the elliptic
  capacity, see e.g. \cite[Eqn. (2.13)]{HKM}. This yields the desired
  estimate in the first case. 

  In the second case $\rho>2|x-x_o|$, we observe
  \begin{align*}
    \R^n\setminus\Omega\subset \R^n\setminus\Omega_\delta
    \qquad\mbox{and}\qquad
    \overline{B_{\rho/2}(x_o)}\subset \overline{B_\rho(x)},
  \end{align*}
  so that
  \begin{align*}
    \ca_\mathrm{e}\big((\R^n\setminus\Omega_\delta) \cap \overline{B_\rho(x)}, B_{2\rho}(x)\big)
    &\geq
    \ca_\mathrm{e}\big((\R^n\setminus\Omega)\cap \overline {B_{\rho/2}(x_o)}, B_{2\rho}(x)\big)\\
    &\geq
      c(p)\ca_\mathrm{e}\big((\R^n\setminus\Omega)\cap \overline
      {B_{\rho/2}(x_o)}, B_{\rho}(x_o)\big).
  \end{align*}
  The last step follows by using a cutoff function $\zeta\in C^\infty_0(B_{2\rho}(x))$ with
  $\spt\zeta\subset B_\rho(x_o)\subset B_{2\rho}(x)$, 
  $\zeta\equiv1$ in $B_{\rho/2}(x_o)$ and $|\nabla\zeta|\le\frac3\rho$
  and applying Poincar\'e's
  inequality. We estimate the right hand side further by using the
  uniform $p$-fatness of $\R^n\setminus\Omega$ and arrive at 
  \begin{align*}
    \ca_\mathrm{e}\big((\R^n\setminus\Omega_\delta) \cap \overline{B_\rho(x)}, B_{2\rho}(x)\big)
    &\geq
    c(p)\alpha \ca_\mathrm{e}\big(\overline{B_{\rho/2}(x_o)},
      B_{\rho}(x_o)\big)
      =
      c(n,p)\alpha\rho^{n-p}\\
      &=
      c(n,p)\alpha \ca_{\mathrm{e}}\big(\overline{B_{\rho}(x)},
      B_{2\rho}(x)\big).
  \end{align*}
  This yields the claim in the remaining case. 
\end{proof}

Then we recall the following result on the obstacle problem, see~\cite{GKM,KKP,kokusi}.

\begin{proposition} \label{prop:obstacle}
Let $1<p<\infty$ and $\Omega \subset \R^n$ be a bounded open set such that $\R^{n} \setminus \Omega$ is uniformly $p$-fat. In addition, suppose that $\psi \in C(\overline{\Omega_T})$. Then there exists a supercaloric function $u \in C(\overline{\Omega_T})$ to~\eqref{p-lap} with the following properties:
	\begin{itemize}
		\item[(i)] $u\geq \psi$ in $\Omega_T$ and $u=\psi$ on $\partial_p \Omega_T$,
		\item[(ii)] $u$ is a weak solution in the set $\{u>\psi\}$,
		\item[(iii)] $u$ is the smallest supercaloric function above $\psi$ in $\Omega_T$, i.e., if $v$ is a supercaloric function satisfying $v \geq \psi$ in $\Omega_T$, then $v \geq u$ in $\Omega_T$.
	\end{itemize} 
\end{proposition}

The following lemma states that pointwise values of a supercaloric function are recovered by the essliminf-regularization. Furthermore, sufficient conditions to guarantee that a supercaloric function is a weak supersolution are given.

\begin{lemma} \label{lem:supercal-prop}
Let $1<p<\infty$ and $u$ be a supercaloric function in an open set $U \subset \R^{n+1}$. Then $u_* = u$ everywhere in $U$, where $u_*$ is defined as in Lemma~\ref{lem:supersol-lsc}. Furthermore, if $u \in L^\infty_{\loc}(U)$ or $u \in L^p(t_1,t_2;W^{1,p}(V))$ whenever $V_{t_1,t_2} \Subset U$, then $u$ is a weak supersolution in $U$.
\end{lemma}

\begin{proof}
The fact that $u_* = u$ is proved in~\cite[Theorem 5.1]{KL} in case $p
\geq 2$, but the proof works also in the case $1<p < 2$ by applying
Proposition~\ref{prop:obstacle} and~\cite[Theorem 2.4]{BBGP}. The fact that $u$ is a weak
supersolution provided that $u \in L^\infty_{\loc}(U)$ follows
from~\cite{GKM,KL}. The claim with $u \in L^p(t_1,t_2;W^{1,p}(V))$ for
arbitrary $V_{t_1,t_2} \Subset U$ follows by considering truncations
$\min \{u,k\}$ for $k = 1,2,...$ and using the previous result
together with the dominated convergence theorem.
\end{proof}

Conversely, weak supersolutions have supercaloric representatives, provided they are
essentially bounded from below.
\begin{lemma}\label{lem:supersol-supercal}
Let $1<p<\infty$ and $u$ be a weak supersolution in an open set $U \subset \R^{n+1}$, such that $u$ is locally essentially bounded from below. Then $u_*$ defined in Lemma~\ref{lem:supersol-lsc} is a supercaloric function in $U$.
\end{lemma}

\begin{proof}
Clearly $u_*$ is lower semicontinuous and finite in a dense subset of $U$ since $u$ is. For comparison, let $Q \Subset U$ be a cylinder and let $h \in C(\overline{Q})$ be a weak solution in $Q$ such that $h \leq u_*$ on $\partial_p Q$. Denote $\tilde u = \min \left\{u_*, \sup_{Q} h\right\}$ which is a bounded, lower semicontinuous weak supersolution by~\cite[Lemma 3.2]{KKP} and satisfies $h \leq \tilde u$ on $\partial_p Q$. By using~\cite[Lemma 3.5]{KKP} it follows that $u_* \geq \tilde u \geq h$ a.e. in $Q$. By continuity of $h$ and the fact that $(u_*)_* = u_*$, it follows that $u_* \geq h$ everywhere in $Q$.
\end{proof}

\subsection{Balayage}

The notion of balayage is a central concept in classical potential
theory. We use the following definition, which is adapted to the
nonlinear parabolic setting.

\begin{definition}
Let $\psi: \Omega_T \to [-\infty,\infty]$ be a function, and denote
$$
\mathcal U_\psi := \{ v \text{ is a supercaloric function in } \Omega_T : v(x,t) \geq \psi(x,t)\, \text{ for every } (x,t) \in \Omega_T \}.
$$
We define the r\'eduite of $\psi$ as 
$$
R_\psi(x,t) = \inf \{ v(x,t): v \in \mathcal U_\psi \},
$$
and the balayage of $\psi$ as its lower semicontinuous regularization
$$
\widehat{R}_\psi(x,t) := \lim_{\rho \to 0} \left( \inf_{B_\rho(x) \times (t-\rho^p,t+\rho^p)} R_\psi \right), 
$$
in case $\mathcal U_\psi \neq \varnothing$. If $\mathcal U_\psi = \varnothing$ we interpret $R_\psi = \widehat R_\psi \equiv \infty$.

If $\psi=\chi_K$ is a characteristic function of a compact set $K
\subset \Omega_T$, we use the abbreviation $\widehat R_K := \widehat R_{\chi_K}$.
\end{definition}

We summarize some elementary properties of the balayage in the
following lemma. 
\begin{lemma} \label{lem:balayage-psi-prop}
Let $1<p<\infty	$.
 If $\psi$ is a bounded function, then $\widehat R_\psi$ is a
  bounded supercaloric function in $\Omega_T$ that satisfies
  $\widehat R_\psi = (R_\psi)_*$ everywhere in $\Omega_T$ and
  $\widehat R_\psi = R_\psi$ a.e. in $\Omega_T$.  

  If
  $\psi \in C(\overline{\Omega_T})$, then
  $\widehat R_\psi \in C(\overline{\Omega_T})$, $\widehat R_\psi$ is a
  weak solution in the set $\{u>\psi\}$ and $\widehat R_\psi$ is the
  smallest supercaloric function above $\psi$ in $\Omega_T$.
\end{lemma}

\begin{proof}

Since constants are supercaloric functions, we have $\widehat{R}_\psi
\leq \sup_{\Omega_T} \psi < \infty$. Furthermore, $-\infty <
\inf_{\Omega_T} \psi \leq \widehat{\psi} \leq \widehat{R}_\psi$ in
$\Omega_T$, which shows boundedness of
$\widehat{R}_\psi$. $\widehat{R}_\psi$ is clearly a supercaloric
function, and the further two properties follow from arguments
in~\cite[Lemma 2.7, Theorem 2.8]{LP} in connection with Lemmas~\ref{lem:supersol-lsc} and~\ref{lem:supercal-prop}. The properties for $\psi \in  C(\overline{\Omega_T})$ follow from Proposition~\ref{prop:obstacle} since $\widehat R_\psi$ is the smallest supercaloric function above $\psi$ and thus coincides with the solution given by Proposition~\ref{prop:obstacle}.

\end{proof}

The next lemma is concerned with the boundary behaviour of the
balayage in the case that $\psi$ is compactly supported. 

\begin{lemma} \label{lem:balayage-K-properties}
  Let $1<p<\infty$ and $\Omega \subset \R^n$ be a bounded open set. Suppose that $\psi \colon \Omega_T \to  [0, \infty)$ is a bounded function with compact support in $\Omega_T$.
  Then, there exists $\delta
  > 0$ such that $\widehat R_\psi (\cdot,t) \equiv 0$ for every $t \in
  (0,\delta]$. Furthermore, $\widehat R_\psi$ is bounded in $\Omega_T$ and a continuous weak solution in $\Omega_T \setminus \spt(\psi)$ with $\widehat R_\psi  \in L^p(0,T;W^{1,p}_0(\Omega))$. Moreover, $\widehat R_\psi$ is continuous up to $\partial_p \Omega_T$ if $\R^n \setminus \Omega$ is uniformly $p$-fat.
\end{lemma}

\begin{proof}

The first property is clear by definition and since a nonnegative supercaloric function can be  redefined by zero in the past. 

$\widehat R_\psi$ is a weak solution in $\Omega_T \setminus
\spt(\psi)$, since it coincides with its Poisson modification in any
cylinder $Q \Subset \Omega_T \setminus \spt (\psi)$, and continuity
follows from boundedness of the Poisson modification (which follows
from boundedness of $\widehat R_\psi$). Finally, we have $\widehat R_\psi  \in
L^p(0,T;W^{1,p}_0(\Omega))$ since approximants in the Poisson
modification are zero on $\partial \Omega$ in the Sobolev sense (and
uniformly bounded) and take the boundary values continuously. For this
argument, we also make use of the comparison principle on
noncylindrical sets, see~\cite[Theorem 2.4]{BBGP}.

For continuity up to the boundary of (bounded) weak solutions when $\R^n \setminus \Omega$ is uniformly $p$-fat, we refer to~\cite{GL1,GL2,GLL,kilpelainen-lindqvist}.

\end{proof}

\subsection{Parabolic function spaces}
\subsubsection{Definitions}

We abbreviate
$$\mathcal V(\Omega_T) = L^p(0,T;W^{1,p}_0(\Omega))\quad \text{ and }\quad\mathcal V'(\Omega_T) = \big( L^{p}(0,T;W^{1,p}_0(\Omega)) \big)'.$$
For the definition of variational parabolic capacity, we use the
function space 
\begin{equation}
\mathcal W(\Omega_T) = \left\{ v \in \mathcal V(\Omega_T) \cap
  L^\infty(0,T;L^2(\Omega)) : \partial_t v \in \mathcal V'(\Omega_T)
\right\},\label{def-W}
\end{equation}
 together with
\begin{align}\label{W-norm}
\|v\|_{\mathcal{W}(\Omega_T)} &= \|v\|_{\mathcal{V}(\Omega_T)}^p + \|\partial_t v\|_{\mathcal{V}'(\Omega_T)}^{p'} + \|v\|_{L^\infty(0,T;L^2(\Omega))}^2  \\\nonumber
&=  \iint_{\Omega_T} |\nabla v|^{p} \, \d x \d t  + \sup_{\substack{ \phi \in C^\infty_0(\Omega_T) \\ \|\phi\|_{\mathcal{V}(\Omega_T)} \leq 1}} \left| \iint_{\Omega_T} v \partial_t \phi \, \d x \d t \right|^{p'} \\\nonumber
&\quad + \sup_{0<t<T} \int_\Omega v(x,t)^2 \, \d x.
\end{align}

\begin{remark} \label{rem:C-L2}
  In~\cite{droniou}, the authors consider the space
  $$
  \widehat{\mathcal{W}}(\Omega_T) = \big\{ v \in L^p(0,T;V):
  \partial_t v \in L^{p'}(0,T;V') \big\}
  $$
  for $V = W^{1,p}_0(\Omega) \cap L^2(\Omega)$. Observe that
  $\mathcal{W}(\Omega_T)$ is continuously embedded in
  $\widehat{\mathcal{W}}(\Omega_T)$ and, furthermore,
  $\widehat{\mathcal{W}}(\Omega_T)$ is continuously embedded in
    $L^\infty(0,T;L^2(\Omega))$.

\end{remark}

In the case $p\ge\frac{2n}{n+2}$, the condition $v\in
L^\infty(0, \infty;L^2(\Omega))$ in the definition of
$\mathcal{W}(\Omega_{\infty})$ could be omitted. This is a consequence of the
following result.  
\begin{lemma}\label{lem:parabolic-embedding}
  Let $\Omega\subset\R^n$ be a bounded domain, $T\in(0,\infty]$ and
  $p\in(1,\infty)$ an exponent with $p\ge\frac{2n}{n+2}$. Then, for every $v\in\mathcal{V}(\Omega_T)$ with
  $\partial_tv\in\mathcal{V}'(\Omega_T)$, we have the estimate
  \begin{equation*}
    \|v\|_{L^\infty(0,T;L^2(\Omega))}^2
    \le
    2\|\partial_tv\|_{\mathcal{V}'(\Omega_T)}\|v\|_{\mathcal{V}(\Omega_T)}
    +
    \frac{c}{T^{\frac{2}{p}}}\|v\|^2_{\mathcal{V}(\Omega_T)}
  \end{equation*}
  with a constant $c=c(n,p,|\Omega|)$. In the case $T=\infty$, the
  last term can be omitted. 
\end{lemma}
\begin{proof}
  From \cite[Chapter III.1, Proposition 1.2]{Showalter} we know that
  $u\in C([0,T);L^2(\Omega))$ and that for a.e. $t,\tau\in(0,T)$ we
  have
  \begin{align}\label{bound-L2-norm}
    \|v(t)\|_{L^2(\Omega)}^2
    &\le
      \|v(\tau)\|_{L^2(\Omega)}^2
      +
      \bigg|\int_\tau^t\frac{d}{ds}\|v(s)\|_{L^2(\Omega)}^2\,\d s\bigg|\\\nonumber
    &=
      \|v(\tau)\|_{L^2(\Omega)}^2
      +
      2\bigg|\int_\tau^t\langle\partial_tv,v\rangle_{\mathcal
      V',\mathcal V}\,\d s\bigg|\\\nonumber
    &\le
      c\|v(\tau)\|_{W^{1,p}_0(\Omega)}^2
      +
      2\|\partial_tv\|_{\mathcal{V}'(\Omega_T)}\|v\|_{\mathcal{V}(\Omega_T)}.
  \end{align}
  In the last line we used H\"older's inequality and
  Sobolev's embedding, which is possible
  in the stated form since $p\ge\frac{2n}{n+2}$.
  The notation $\langle\cdot,\cdot\rangle_{\mathcal{V}',\mathcal{V}}$
  stands for the duality pairing between $\mathcal{V}'(\Omega_T)$ and
  $\mathcal{V}(\Omega_T)$. 
  The constant in the preceding estimate depends on $n,p,$ and $|\Omega|$. 
  In the case $T<\infty$, we can choose $\tau\in(0,T)$ so that 
  \begin{align*}
    \|v(\tau)\|_{W^{1,p}_0(\Omega)}^p
    \le
    \bint_0^T\|v(s)\|_{W^{1,p}_0(\Omega)}^p\, \d s
    =
    \frac1T\|v\|_{\mathcal{V}(\Omega_T)}^p.
  \end{align*}
  With this choice of $\tau$, estimate~\eqref{bound-L2-norm} implies
  the claim in the case $T<\infty$ after taking the
  supremum over $t\in(0,T)$.
  In the case $T=\infty$, for any given $\varepsilon>0$ we
  can find a time $\tau\in(0,\infty)$ with
  \begin{align*}
    \|v(\tau)\|_{W^{1,p}_0(\Omega)}<\eps,
  \end{align*}
  since otherwise, $\|v\|_{\mathcal{V}(\Omega_T)}=\infty$. Using this
  estimate 
  in~\eqref{bound-L2-norm}, taking the supremum over $t\in(0,\infty)$ and
  letting $\varepsilon\downarrow0$, we obtain the assertion also in
  this case. 
\end{proof}

\subsubsection{Constructions in parabolic function spaces}
The following lemma deals with a cutoff construction in time for a
function in $\mathcal{W}(\Omega_T)$, which provides us with
a function in $\mathcal{W}(\Omega_\infty)$ that vanishes at times $t>T$. 

\begin{lemma} \label{lem:WT-Winfty-supercritical}
  Let $\frac{2n}{n+2}<p<\infty$, $0<T<\infty$
  and $v\in\mathcal{W}(\Omega_T)$. For
   any cutoff function $\xi \in C_0^\infty((-\infty,T),[0,1])$ with
   $\xi =1$ in $[0,\frac T2]$ and $|\xi'| \leq \frac3T$, we have 
   \begin{align*}
     \|\xi v\|_{\mathcal{W}(\Omega_\infty)}
     &\le
     \|v\|_{\mathcal{W}(\Omega_T)} +
     \Big[\big(1+cT^{-\frac1p}\big)^{p'-1}-1\Big]\|\partial_t
     v\|_{\mathcal{V}'(\Omega_T)}^{p'}\\
     &\qquad+
     cT^{-\frac1p}\big(1+cT^{-\frac1p}\big)^{p'-1}\|v\|_{L^\infty(0,T;L^2(\Omega))}^{p'},
   \end{align*}
   with a constant $c$ depending only on $n,p$ and $|\Omega|$. 
 \end{lemma}

 \begin{proof}
   Let $\xi$ be a cutoff function as in the statement of the lemma.
   For every $\phi\in C^\infty_0(\Omega_\infty)$ we have 
   \begin{align*}
     \bigg|\iint_{\Omega_T}\xi v\partial_t\phi\,\dx\dt\bigg|
     &\le
     \bigg|\iint_{\Omega_T}v\partial_t(\xi\phi)\,\dx\dt\bigg|
     +
     \bigg|\iint_{\Omega_T} v\phi \xi'\,\dx\dt\bigg|\\
     &\leq \|\partial_t v\|_{\mathcal{V}'(\Omega_{T})} \|\phi \|_{\mathcal{V}(\Omega_\infty)} + \frac{3}{T} \|v \phi \|_{L^1(\Omega_{T})}.
\end{align*}
The last term can be estimated by 
\begin{align*}
\|v \phi \|_{L^1(\Omega_{T})} &\leq \|v \|_{L^\infty(0,T;L^2(\Omega))} \|\phi \|_{L^1(0,T;L^2(\Omega))} \\
&\leq \|v \|_{L^\infty(0,T;L^2(\Omega))}c{T}^{1- \frac{1}{p}} \|\phi\|_{\mathcal{V}(\Omega_\infty)},
\end{align*}
with $c=c(n,p,|\Omega|)$. 
In the last step we used the embedding $W^{1,p}_0(\Omega)\subset
L^2(\Omega)$ and H\"older's inequality in the time integral. 
The two preceding estimates imply 
\begin{align*}
  &\big\|\partial_t(\xi v)\big\|_{\mathcal{V}'(\Omega_\infty)}^{p'}\\
  &\qquad\leq
  \Big(\| \partial_t v\|_{\mathcal{V}'(\Omega_{T})} +
  cT^{-\frac{1}{p}} \|v\|_{L^\infty(0,T;L^2(\Omega))}\Big)^{p'}\\
  &\qquad=
    \big(1+cT^{-\frac1p}\big)^{p'}\bigg(\frac{1}{1+cT^{-\frac1p}}\| \partial_t v\|_{\mathcal{V}'(\Omega_{T})} +
    \frac{cT^{-\frac{1}{p}}}{1+cT^{-\frac1p}}
    \|v\|_{L^\infty(0,T;L^2(\Omega))}\Big)^{p'}\\
  &\qquad\le
    \big(1+cT^{-\frac1p}\big)^{p'-1}\| \partial_t v\|_{\mathcal{V}'(\Omega_{T})}^{p'}
    +
    cT^{-\frac1p}\big(1+cT^{-\frac1p}\big)^{p'-1}\|v\|_{L^\infty(0,T;L^2(\Omega))}^{p'},
\end{align*}
where we applied Jensen's inequality in the last step. Finally, it is
straightforward to show 
\begin{align*}
\|\xi v\|_{\mathcal{V}(\Omega_\infty)} \le  \|v\|_{\mathcal{V}(\Omega_T)}
  \qquad\mbox{and}\qquad
\|\xi v\|_{L^\infty(0,\infty;L^2(\Omega))} \le  \|v\|_{L^\infty(0,T;L^2(\Omega))}.
\end{align*}
Combining the preceding estimates, we deduce the asserted bound. 
\end{proof}

The next result will enable us to extend a given function in
$\mathcal{W}(\Omega_{T_o})$ to a larger time interval. 

\begin{lemma} \label{lem:WT-Winfty}
  For times $0<T_o<T\le\infty$, let $v_1\in
  \mathcal{W}(\Omega_{T_o})$ and  $v_2\in
  \mathcal{W}(\Omega\times(T_o,T))$ be two functions with
  \begin{equation}\label{final=initial}
    v_1(\cdot,T_o)
    :=
    \lim_{t\uparrow T_o}v_1(\cdot,t)
    =
    \lim_{t\downarrow T_o}v_2(\cdot,t)
    =:v_2(\cdot,T_o),
  \end{equation}
  where the limits are to be understood with respect to the $L^2(\Omega)$-norm. 
  Then the function
  \begin{equation*}
    \hat v(\cdot,t):=\left\{
      \begin{array}{ll}
        v_1(\cdot,t)&\mbox{if }t\in(0,T_o],\\[0.8ex]
        v_2(\cdot,t)&\mbox{if }t\in(T_o,T),
      \end{array}
    \right.
  \end{equation*}
  satisfies $\hat v\in \mathcal{W}(\Omega_T)$ with
  \begin{equation} \label{eq:vhat-v1-v2-estimate}
    \|\hat v\|_{\mathcal{W}(\Omega_T)}
    \le
    \|v_1\|_{\mathcal{W}(\Omega_{T_o})}+\|v_2\|_{\mathcal{W}(\Omega\times(T_o,T))}.
  \end{equation}
In particular, if $v_2\in
  \mathcal{W}(\Omega\times(T_o,T))$ is a solution to
  \begin{align} \label{eq:initial-problem}
\left\{
\begin{array}{rl}
\partial_t v_2 - \Delta_p  v_2 = 0 \quad &\text{ in } \Omega \times (T_o,T), \\[0.5ex]
v_2 = v_1 \quad &\text{ on } \Omega \times \{T_o\},
\end{array}
\right. 
\end{align}
then 
\begin{equation} \label{eq:vhat-v1-estimate}
    \|\hat v\|_{\mathcal{W}(\Omega_T)}
    \le
    3\|v_1\|_{\mathcal{W}(\Omega_{T_o})}.
\end{equation}
\end{lemma}

\begin{proof}
  Clearly, we have $\hat v\in L^p(0,T;W^{1,p}_0(\Omega))\cap
  L^\infty(0,T;L^2(\Omega))$ with 
  \begin{equation}\label{bound-W1}
    \|\hat v\|_{\mathcal{V}(\Omega_T)}^p
    =
    \|v_1\|_{\mathcal{V}(\Omega_{T_o})}^p
    +
    \|v_2\|_{\mathcal{V}(\Omega\times(T_o,T))}^p
  \end{equation}
  and 
  \begin{equation}\label{bound-W2}
    \|\hat v\|_{L^\infty(0,T;L^2(\Omega))}^2
    \le
    \|v_1\|_{L^\infty(0,T_o;L^2(\Omega))}^2
    +
    \|v_2\|_{L^\infty(T_o,T;L^2(\Omega))}^2.
  \end{equation}
  For the analysis of the time derivative, we introduce piecewise
  affine cutoff functions $\xi_h(t)$ which converge monotonically to
  $\chi_{(0,T)\setminus\{T_o\}}$ as $h\to0$. 
  For an arbitrary $\phi \in C_0^\infty(\Omega_T)$, we compute
\begin{align*}\label{dual-norm-extension}
  \left|  \iint_{\Omega_T} \hat v \partial_t \phi \, \d x \d t
  \right|
  &=
    \lim_{h \to 0}\left|  \iint_{\Omega_T} \xi_h \hat v \partial_t \phi \, \d x \d t \right| \\\nonumber
&= \lim_{h \to 0}\bigg|  \iint_{\Omega_{T_o}}  v_1 \partial_t (\xi_h \phi) \, \d x \d t -  \iint_{\Omega_{T_o}}  v_1 \phi \xi_h'  \, \d x \d t \\\nonumber
&\quad\qquad + \iint_{\Omega \times (T_o,T)}  v_2 \partial_t (\xi_h \phi) \, \d x \d t -  \iint_{\Omega \times (T_o,T)}  v_2 \phi \xi_h'  \, \d x \d t \bigg| \\\nonumber
&= \lim_{h \to 0}  \bigg| \iint_{\Omega_{T_o}}  v_1 \partial_t (\xi_h \phi) \, \d x \d t + \iint_{\Omega \times (T_o,T)} v_2 \partial_t (\xi_h \phi) \, \d x \d t \bigg| \\\nonumber
&\leq \left( \| \partial_t v_1 \|_{\mathcal{V}'(\Omega_{T_o})}^{p'} + \| \partial_t v_2 \|_{\mathcal{V}'(\Omega \times (T_o,T))}^{p'} \right)^{\frac{1}{p'}} \| \phi \|_{\mathcal{V}(\Omega_T)}.
\end{align*}
In the second last step we used assumption~\eqref{final=initial} and
in the final step H\"older's inequality for sums. 
The preceding estimate implies $\partial_t\hat v\in \mathcal{V}'(\Omega_T)$ and
\begin{equation}
\label{bound-W3}
\| \partial_t \hat v \|_{\mathcal{V}'(\Omega_T)}^{p'} \leq \|
\partial_t v_1 \|_{\mathcal{V}'(\Omega_{T_o})}^{p'} + \| \partial_t
v_2 \|_{\mathcal{V}'(\Omega \times (T_o,T))}^{p'}.
\end{equation}
Altogether, we have shown $\hat v\in \mathcal{W}(\Omega_T)$ and the
combination of~\eqref{bound-W1}, \eqref{bound-W2} and \eqref{bound-W3}
yields the claimed estimate~\eqref{eq:vhat-v1-v2-estimate}. 

In case $v_2$ is a solution to~\eqref{eq:initial-problem}, we test the
weak formulation of equation~\eqref{eq:initial-problem}$_1$ with the test
function $( (v_2)_\eps \chi_{h,\tau})_\eps$, where $(\cdot)_\eps$
denotes a standard mollification in time direction with a small parameter $\eps > 0$, and $\chi_{h,\tau}(t)$ is a piecewise affine approximation of $\chi_{(T_o,\tau)}(t)$ for $\tau \in (T_o+2\eps, T)$. With standard computations, after passing to the limit this yields
\begin{equation*}
\frac12 \int_\Omega v_2(x,\tau)^2 \, \d x + \int_{T_o}^{\tau} \int_\Omega |\nabla v_2|^p \, \d x \d t  = \frac12 \int_{\Omega} v_1(x,T_o)^2 \, \d x
\end{equation*}
for every $\tau \in (T_o,T)$, which implies
\begin{equation} \label{eq:v2energy-v1}
\max \left\{ \tfrac12 \|v_2\|_{L^\infty(T_o,T;L^2(\Omega))}^2, \|v_2\|_{\mathcal{V}(\Omega \times (T_o,T))}^p \right\} \leq \tfrac12 \|v_1\|_{L^\infty(0,T_o;L^2(\Omega))}^2.
\end{equation}
From the weak formulation of equation~\eqref{eq:initial-problem}$_1$ we further deduce that 
$$
\| \partial_t v_2 \|_{\mathcal{V}'(\Omega \times (T_o,T))}^{p'} \leq \|v_2\|_{\mathcal{V}(\Omega \times (T_o,T))}^{p},
$$
which together with~\eqref{eq:v2energy-v1} and~\eqref{eq:vhat-v1-v2-estimate} implies estimate~\eqref{eq:vhat-v1-estimate}.
\end{proof}

\subsubsection{An approximation result}

The following lemma can be seen as an approximation result for functions
$v\in\mathcal{W}(\Omega_T)$ by smooth functions. For the applications
to variational capacity that we have in mind,
it will be crucial that an obstacle condition of the form
$v\ge\chi_K$ is preserved under the approximation.  

\begin{lemma} \label{lem:W-approx}
  Let $1<p < \infty$ and $\Omega\subset\R^n$ be an open set whose complement
  $\R^n\setminus\Omega$ is uniformly $p$-fat and $0<T\leq \infty$.
  Moreover, assume that
  $K\subset\Omega_T$ is a compact set consisting of a finite union of space-time
  cylinders, whose bases are balls. Let $v\in \mathcal{W}(\Omega_T)$ be given with
  $v\ge\chi_K$ a.e. in $\Omega_T$. If $1<p \leq \tfrac{2n}{n+2}$ and $T = \infty$ suppose in addition that $v \in L^\infty(\Omega_\infty)$. Then, there exists a function $w\in
  C^\infty_0(\Omega\times\R)$ that satisfies $w\ge\chi_K$ as well and, moreover,
  \begin{equation*}
    \|w\|_{\mathcal{W}(\Omega_T)}\le c\|v\|_{\mathcal{W}(\Omega_T)}
  \end{equation*}
  with a constant $c=c(n,p,\alpha) > 0$. 
\end{lemma}

\begin{proof}
First suppose that $T < \infty$. We extend $v$ by using
reflection to $\Omega \times [-T,2T]$ by defining $\tilde v (\cdot, t) =
v(\cdot,-t)$ for $t\in [-T,0)$ and $\tilde v(\cdot,t) = v(\cdot,2T-t)$
for $t\in (T,2T]$.
Because of the embedding $\mathcal{W}(\Omega_T)\subset
C([0,T];L^2(\Omega))$, we can use Lemma~\ref{lem:WT-Winfty}
to show $\tilde v\in \mathcal{W}(\Omega\times(-T,2T))$ with
\begin{equation}
  \label{dual-norm-reflection}
  \|\tilde v\|_{\mathcal{W}(\Omega\times(-T,2T))}
  \le
  3\|v\|_{\mathcal{W}(\Omega_T)}.
\end{equation}
We consider a cutoff function in space $\zeta_\varepsilon\in
  C^\infty_0(\Omega,[0,1])$ with $\zeta_\varepsilon\equiv1$ on 
  $\Omega_{2\varepsilon}:=\{x\in\Omega\colon
    \mathrm{dist}(x,\partial\Omega)>2\varepsilon\}$,
  $\zeta\equiv0$ on $\Omega\setminus\Omega_{\eps}$ 
  and $|\nabla\zeta_\eps|\le\frac{c_o}{\varepsilon}$.
  By choosing $\varepsilon>0$ small enough, we can assume
  $K\subset\Omega_{2\varepsilon}\times(0,T)$. For standard mollifiers
  $\varphi^x_\delta \in C^\infty_0(\R^{n})$ and $\varphi^t_\delta \in
  C^\infty_0(\R)$ and a positive parameter $\delta<\min\{\eps,T\}$ we let $\eta_\delta = \varphi^x_\delta  \varphi^t_\delta \in
  C^\infty_0(\R^{n+1})$ and define
  \begin{equation*}
    w:= 5(\zeta_\varepsilon \tilde v)\ast\eta_\delta
    \qquad\mbox{in }\Omega\times\R,
  \end{equation*}
  where we extended $\tilde v$ by zero outside of $\Omega\times [-T,2T]$ for the definition. 
  Because of $\delta<\eps$, the support of $w$ is
  compactly contained in $\Omega\times\R$. Since $\R^n\setminus\Omega$
  is uniformly $p$-fat and $v\in L^p(0,T;W^{1,p}_0(\Omega))$, Young's
  theorem for convolutions and Hardy's inequality from Lemma~\ref{lem:hardy} imply
  \begin{align}\label{V-est-w}
    \iint_{\Omega_T}|\nabla w|^p\dx\dt
    &\le
      5^p\iint_{\Omega\times(-\delta,T+\delta)}\big|\nabla(\zeta_\eps\tilde
      v)\big|^p\,\dx\dt\\\nonumber
    &\le
    c\iint_{\Omega_T}|\nabla v|^p+|\nabla\zeta_\varepsilon|^p|v|^p\,\dx\dt\\\nonumber
    &\le
    c\iint_{\Omega_T}|\nabla v|^p+\Big(\frac{|v|}{\mathrm{dist}(x,\partial\Omega)}\Big)^p\dx\dt
    \le
    c\iint_{\Omega_T}|\nabla v|^p\dx\dt,
  \end{align}
  for $c= (n,p,\alpha)>0$, and we have
  \begin{equation}
  \label{CL2-est-w}
    \| w \|_{L^\infty(0,T;L^2(\Omega))} \leq 5\|v\|_{L^\infty(0,T;L^2(\Omega))}.
  \end{equation}
  Furthermore, for any $\phi \in C_0^\infty(\Omega_T)$ we can estimate
  \begin{align*}
    &\left| \iint_{\Omega_T} w \partial_t \phi \, \d x\d t \right|\\
    &\qquad=
      5 \left| \iint_{\Omega\times(-\delta,T+\delta)} \tilde v \partial_t \left[ \zeta_\eps (\phi * \eta_\delta) \right]  \d x\d t \right| \\
  &\qquad\leq 5 \|\partial_t \tilde v
    \|_{\mathcal{V}'(\Omega\times(-\delta,T+\delta))}
    \|\zeta_\eps (\phi * \eta_\delta)\|_{\mathcal{V}(\Omega\times(-\delta,T+\delta))} \\
  &\qquad\leq c \|\partial_t v\|_{\mathcal{V}'(\Omega_T)} \big(\|\nabla(\phi\ast\eta_\delta)\|_{L^p(\Omega\times(-\delta,T+\delta))}+\|\nabla\zeta_\eps(\phi\ast\eta_\delta)\|_{L^p(\Omega\times(-\delta,T+\delta))}\big).
  \end{align*}
  For the last estimate we used~\eqref{dual-norm-reflection}.
  Using the properties of $\zeta_\eps$, Young's inequality for convolutions and finally Hardy's inequality for
  $\phi\in C^\infty_0(\Omega_T)$, we estimate the last term by
  \begin{align*}
    \|\nabla\zeta_\eps(\phi\ast\eta_\delta)\|_{L^p(\Omega\times(-\delta,T+\delta))}^p
    &\le
    \frac{c}{\eps^p}\iint_{(\Omega_{\eps}\setminus\Omega_{2\eps})\times(-\delta,T+\delta)}|\phi\ast\eta_\delta|^p\dx\dt\\
    &\le
    \frac{c}{\eps^p}\iint_{(\Omega\setminus\Omega_{3\eps})\times(0,T)}|\phi|^p\dx\dt\\
    &\le
      c\iint_{\Omega_T}\Big(\frac{|\phi|}{\mathrm{dist}(x,\partial\Omega)}\Big)^p\dx\dt
    \le c\iint_{\Omega_T}|\nabla\phi|^p\dx\dt,
  \end{align*}
  where $c=c(n,p,\alpha)>0$. Inserting this above, we deduce
  \begin{align*}
    \left| \iint_{\Omega_T} w \partial_t \phi \, \d x\d t \right|
    \le
     c \|\partial_t v\|_{\mathcal{V}'(\Omega_T)}\|\nabla\phi\|_{L^p(\Omega_T)}    
  \end{align*}
  for every $\phi\in C^\infty_0(\Omega_T)$.
  Thus $\|\partial_t w\|_{\mathcal{V}'(\Omega_T)} \leq c \|\partial_t
  v\|_{\mathcal{V}'(\Omega_T)}$, and together with~\eqref{V-est-w} and
  \eqref{CL2-est-w} we obtain
  \begin{equation*}
    \|w\|_{\mathcal{W}(\Omega_T)}\le c(n,p,\alpha) \|v\|_{\mathcal{W}(\Omega_T)}.
  \end{equation*}
  Moreover, since $\zeta_\eps \tilde v\ge\chi_K$ and
  $K$ is a finite union of space-time cylinders, we have
  \begin{equation*}
    w\ge 5\chi_K\ast\eta_\delta\ge 1 \qquad\mbox{ on $K$},
  \end{equation*}
  provided we choose the mollification parameter $\delta>0$ small
  enough. This completes the proof in the case $T<\infty$.  

If $T = \infty$ and $p > \tfrac{2n}{n+2}$, we use a piecewise affine
cutoff function in time $\xi$ with $\xi \equiv 1$ in $(-\infty,
T_o/2)$ and $\xi \equiv 0$ in
$[T_o,\infty)$. Lemma~\ref{lem:WT-Winfty-supercritical} guarantees
that we can choose $T_o \in (0,\infty)$ large enough, possibly
depending on $v$, to achieve
\begin{equation*}
  \|\xi v\|_{\mathcal{W}(\Omega_\infty)}
  \le
  2 \|v\|_{\mathcal{W}(\Omega_{T_o})}
  \le
  2 \|v\|_{\mathcal{W}(\Omega_\infty)}.
\end{equation*}
Repeating the argument above with $\xi v$ instead of $v$, we deduce
that $w = 5(\zeta_\varepsilon \xi v)\ast\eta_\delta$ satisfies the
desired properties. 

Then we consider the case $1<p\leq \tfrac{2n}{n+2}$ and $T = \infty$. Recall from Remark~\ref{rem:C-L2} that we may assume $v \in C([0,\infty);L^2(\Omega))$. Suppose $T_o \in (0,\infty)$ is large enough such that $K \subset \Omega_{T_o}$, and let $\tilde v \in \mathcal{V}(\Omega \times (T_o,\infty))$ be a solution to
\begin{align*}
\left\{
\begin{array}{rl}
\partial_t \tilde v - \Delta_p  \tilde v = 0 \quad &\text{ in } \Omega \times (T_o,\infty), \\
\tilde v = v\quad &\text{ on } \Omega \times \{T_o\}.
\end{array}
\right. 
\end{align*}
Then we define $\hat v = v$ for $t \leq T_o$ and $\hat v = \tilde v$
for $t > T_o$. According to Lemma~\ref{lem:WT-Winfty}, we have $\hat
v\in\mathcal{W}(\Omega_\infty)$ with 
$$
\| \hat v \|_{\mathcal{W}(\Omega_\infty)} \leq 3 \| v \|_{\mathcal{W}(\Omega_{T_o})} \leq 3 \| v \|_{\mathcal{W}(\Omega_{\infty})}.
$$
Observe that by the assumption $v \in L^\infty(\Omega_\infty)$ together with~\cite[Chapter VII, Proposition 2.1]{Di} there exists $T_1 \in [T_o, \infty)$ such that $\tilde v (\cdot,t) \equiv 0$ for all $t \in [T_1,\infty)$. Thus, we can choose $w = 5 (\zeta_\eps \hat v) * \eta_\delta \in C_0^\infty(\Omega \times \R^n)$ as in the case $T < \infty$, and conclude the result similarly.
  
\end{proof}

\subsubsection{A parabolic function space consisting of bounded functions}

When dealing with bounded solutions to parabolic problems, it is
natural to work with the function space 
$$
\widetilde{\mathcal{W}}(\Omega_T) = \left\{ u \in \mathcal{V}(\Omega_T) \cap L^\infty(\Omega_T): \partial_t u\in \mathcal{V}'(\Omega_T) + L^1(\Omega_T) \right\}.
$$
We recall a useful result from~\cite[Sect. 5, Lemma 1]{Petitta}, see
also~\cite[Lemma 2.17]{droniou}, which relates the function space
$\widetilde{\mathcal{W}}(\Omega_T)$ to the function space
$\mathcal{W}(\Omega_T)$.

\begin{lemma} \label{lem:petitta}
If $v \in \widetilde{\mathcal{W}}(\Omega_T)$, then there exists $u \in \mathcal{W}(\Omega_T)$ such that $|v| \leq u$ a.e. in $\Omega_T$ and
\begin{align*}
\|u\|_{\mathcal{W}(\Omega_T)} \leq c \bigg( &\|v\|_{\mathcal{V}(\Omega_T)}^p + \|[\partial_t v]_a\|_{\mathcal{V}'(\Omega_T)}^{p'} \\
&\quad+ \|v\|_{L^\infty(\Omega_T)} \|[\partial_t v]_b\|_{L^1(\Omega_T)} + \|v\|_{L^\infty(0,T;L^2(\Omega))}^2 \bigg),
\end{align*} 
where $c = c(p) > 0$, and $[\partial_t v]_a \in \mathcal{V}'(\Omega_T)$, $[\partial_t v]_b \in L^1(\Omega_T)$ is any decomposition of $\partial_t v$, that is $\partial_t v = [\partial_t v]_a + [\partial_t v]_b$.
\end{lemma}

\section{Notions of capacity for the parabolic $p$-Laplace equation}

From now on we assume that $\Omega \subset \R^n$ is a bounded open set such that $\R^n \setminus \Omega$ is uniformly $p$-fat according to Definition~\ref{def:p-fat}. We begin by defining the nonlinear parabolic capacity introduced in our setting in~\cite{KKKP}.

\begin{definition} \label{def:nonlinear-capacity-compact}
We define a nonlinear parabolic capacity for a compact set $K \subset \Omega_\infty$ as
$$
\ca (K) = \sup \left\{ \mu(\Omega_\infty) : 0\leq u_\mu \leq 1,\, \spt(\mu) \subset K \right\},
$$
where $\mu$ is a nonnegative Radon measure and $u_\mu$ is a weak solution to the measure data problem
$$
\partial_t u_\mu - \Div \left(|\nabla u_\mu|^{p-2} \nabla u_\mu \right) = \mu\quad \text{ in } \Omega_\infty
$$
in the sense of~\eqref{e.weak_measuresol}, which satisfies $u_\mu \in
L^p(0,T; W^{1,p}_0(\Omega))$ for every $T > 0$ and $u_\mu$ attains
initial values 0 on $\Omega \times \{0\}$ in the $L^2$-sense.

\end{definition}

We recall useful results from~\cite[Theorem 5.7 \& Lemma 5.8]{KKKP},
which hold true also in the case $1 < p \leq \frac{2n}{n+2}$
by~\cite[Chapter VII, Proposition 2.1]{Di}, the results on obstacle
problems in Proposition~\ref{prop:obstacle} and
Lemma~\ref{lem:balayage-K-properties} together with the comparison principle~\cite[Theorem 2.4]{BBGP}.
 
\begin{lemma} \label{lem:cap-compact-representation}
Let $1<p<\infty$ and $K\subset \Omega_\infty$ be a compact set. Then
\begin{equation} \label{eq:cap-potential-representation}
 \ca (K,\Omega_\infty) = \mu_{\widehat{R}_K} (K) = \mu_{\widehat{R}_K} (\Omega_\infty),
\end{equation}
where $\mu_{\widehat{R}_K}$ is the Riesz measure of $\widehat{R}_K$.
\end{lemma}

\begin{lemma} \label{lem:cap-decreasing-limit}
Let $1<p<\infty$ and $K_i \subset \Omega_\infty$ be compact sets such that $K_i \supset K_{i+1}$ for every $i \in \N$ and denote $K = \bigcap_{i=1}^\infty K_i$. Then
$$
\lim_{i \to \infty} \ca (K_i, \Omega_\infty) = \ca (K ,\Omega_\infty). 
$$
\end{lemma}

\begin{remark}
In principle, one could define the capacity in
Definition~\ref{def:nonlinear-capacity-compact} for arbitrary
sets. However, in view of~\cite{KKKP}, to obtain results such as
countable subadditivity and upwards monotone convergence, it seems to
be required that the class of admissible measures in the definition
are closed under taking restrictions to the subsets considered. This
appears to be a delicate issue especially when $p$ is close to $1$.
In the arguments that follow,
we therefore only use the properties of $\ca$ that are stated in
Lemmas~\ref{lem:cap-compact-representation}
and~\ref{lem:cap-decreasing-limit}. This means that one could
alternatively take~\eqref{eq:cap-potential-representation} as
definition of the capacity of a compact set.

Furthermore, whether $T$ in the reference set is finite or not does not play an important role in this case. However, for the sake of consistency with~\cite{KKKP}, we set $T = \infty$ in Definition~\ref{def:nonlinear-capacity-compact}.
\end{remark}

Then we define a variational capacity related to the parabolic $p$-Laplace equation, cf.~\cite{AKP,droniou}.

\begin{definition} \label{def:varcap}
For a compact set $K \subset \Omega_T$, $0<T\le\infty$, we define the variational capacity as
\begin{align*}
\ca_{\text{var}} (K,\Omega_T)= \inf \left\{ \|v\|_{\mathcal W(\Omega_T)}: v \in C_0^\infty(\Omega \times \R), v \geq \chi_K  \right\},
\end{align*}
with
\begin{equation*}
  \|v\|_{\mathcal{W}(\Omega_T)} = \|v\|_{\mathcal{V}(\Omega_T)}^p + \|\partial_t v\|_{\mathcal{V}'(\Omega_T)}^{p'} + \|v\|_{L^\infty(0,T;L^2(\Omega))}^2, 
\end{equation*}
cf. \eqref{W-norm}.
For an open set $U \subset \Omega_T$ we define the variational capacity by
\begin{equation} \label{eq:capvar-openset}
\ca_{\mathrm{var}}(U,\Omega_T) = \sup \left \{\ca_{\mathrm{var}}(K,\Omega_T): K \text{ compact subset of } \Omega_T, K \subset U  \right\},
\end{equation}
and for an arbitrary set $E \subset \Omega_T$ by
\begin{equation} \label{eq:capvar-anyset}
\ca_{\mathrm{var}}(E,\Omega_T) = \inf \left \{\ca_{\mathrm{var}}(U,\Omega_T): U \text{ open subset of } \Omega_T, E \subset U  \right\}.
\end{equation}
\end{definition}

We note that $\|\cdot\|_{\mathcal{W}(\Omega_T)}$ is not a norm because it is
not homogeneous. In this point our definition of parabolic capacity differs from the one
in \cite{droniou}. However, our definition is analogous to the
standard definition in the elliptic case, see Definition~\ref{def:elliptic-cap}.

\begin{remark} \label{rem:W-to-Linfty-L2}
Observe that in case $\frac{2n}{n+2}\le p < \infty$,
Lemma~\ref{lem:parabolic-embedding} and Young's inequality imply
\begin{equation}\label{eq:para-embed}
  \|v\|_{L^\infty(0,T;L^2(\Omega))}^2
  \le
    \frac{2}{p}\|v\|_{\mathcal{V}(\Omega_T)}^p
    +
    \frac{2}{p'}\|\partial_tv\|_{\mathcal{V}'(\Omega_T)}^{p'}
    +
    \frac{c}{T^{\frac{2}{p}}}\|v\|^2_{\mathcal{V}(\Omega_T)}
  \end{equation}
  for every $v\in\mathcal{W}(\Omega_T)$, with a constant
  $c=c(n,p,|\Omega|)$. In the case $T=\infty$, the last term can be
  interpreted as zero and therefore, 
  in the case $p\ge\frac{2n}{n+2}$ and $T=\infty$
  we may neglect the term $\|v\|_{L^\infty(0,\infty;L^2(\Omega))}^2$
  in Definition~\ref{def:varcap} of the variational capacity.
  However, for finite times $T<\infty$
  it seems to be unavoidable to get some quadratic term on the
  right hand side of~\eqref{eq:para-embed}. In order to deal with the
  resulting inhomogeneous form of this estimate, in \cite{AKP} the
  authors introduced a variational parabolic capacity on $\Omega_T$
  that relies on an intrinsic scaling in the time variable. However, their
  approach seems to be limited to exponents $p>2$. In the present
  work, we choose to avoid the use of an intrinsic scaling and instead include
  the term $\|v\|_{L^\infty(0,T;L^2(\Omega))}^2$ in the definition of
  the variational capacity. Anyway, this term is needed in the subcritical
  case $1<p<\frac{2n}{n+2}$, because for such exponents, it can not be
  controlled by the norms $\|v\|_{\mathcal{V}(\Omega_T)}$ and
  $\|\partial_t v\|_{\mathcal{V}'(\Omega_T)}$ due to the lack of a suitable
  Sobolev embedding.
\end{remark}

Finally, as an auxiliary notion we define the energy capacity related to equation~\eqref{p-lap}, see~\cite{AKP}.

\begin{definition} \label{def:encap}
For a compact set $K \subset \Omega_T$, $0<T\le\infty$, we define the energy capacity as
\begin{align*}
\ca_{\text{en}} &(K,\Omega_T) \\
&= \inf \left\{ \|v\|_{\text{en},\Omega_T} : v \in \mathcal{V}(\Omega_T),\, v \text{ is a supercaloric function in } \Omega_T,\, v \geq \chi_K  \right\},
\end{align*}
where
$$
\|v\|_{\text{en},\Omega_T} := \sup_{0<t<T} \frac12 \int_{\Omega} v^2(x,t) \, \d x + \iint_{\Omega_T} |\nabla v|^p \, \d x \d t.
$$
\end{definition}

\subsection{Some properties of variational capacity}

First we collect immediate properties for the variational capacity.

\begin{lemma} \label{lem:capvar-prop}
For variational capacity the following properties hold true.
\begin{enumerate}[(i)]
\item If $K \subset \Omega_T$ is a compact set, then $\ca_{\mathrm{var}} (K,\Omega_T) < \infty$.
\item If $E_1 \subset E_2 \subset \Omega_T$ are arbitrary sets, then $\ca_{\mathrm{var}} (E_1,\Omega_T) \leq \ca_{\mathrm{var}} (E_2,\Omega_T)$.\item If $0<T' \leq T \leq \infty$ and $E \subset \Omega_{T'}$ is a set, then $\ca_{\mathrm{var}} (E,\Omega_{T'}) \leq \ca_{\mathrm{var}} (E,\Omega_{T})$. 
\end{enumerate}
\end{lemma}

For the elliptic capacity, the definition immediately implies that for open sets $\Omega' \subset \Omega \subset \R^n$ and any set $E \subset \Omega'$, one has $\ca_{\mathrm{e}}(E, \Omega) \leq \ca_{\mathrm{e}}(E, \Omega')$. In the parabolic case, the analogous result is not that clear due to the dual norm of the time derivative. However, we are able to prove a weaker form of the inequality, which will be useful in the proof of Theorem~\ref{thm:supercal-polarset}.

\begin{lemma} \label{lem:capvar-omega-omega'}
Let $1<p<\infty$, $0 < T \leq \infty$ and $E \subset \Omega_T$ be a set such that $\mathrm{dist} (E, S_T) > 0$. Then there exists an open set $\Omega' \Subset \Omega$ with $E \subset \Omega'_T$ and
$$
\ca_{\mathrm{var}} (E,\Omega_T) \leq c \ca_{\mathrm{var}} (E, \Omega'_T)
$$
for a constant  $c = c(n,p,\alpha)  > 0$. More precisely, $\Omega'$
can be chosen as $\Omega' = \{x\in \Omega : \mathrm{dist}
(x,\R^n \setminus \Omega) > \delta\}$ for any sufficiently small $\delta>0$.
\end{lemma}

\begin{proof}

Let $\delta > 0$ be such that $\overline{E} \subset \Omega_{2 \delta}
\times [0,T]$, where $\Omega_{\delta} = \{x\in \Omega : \mathrm{dist}
(x,\R^n \setminus \Omega) > \delta\}$. Let $U$ be any open set with
$E\subset U\subset\Omega_{2\delta} \times (0,T)$ and let $K \subset U$ be compact. Let us denote $\Omega' := \Omega_{\delta}$. For every $\eps > 0$ there exists $v \in C_0^\infty(\Omega' \times \R)$ with $v \geq \chi_{K}$ such that
\begin{equation} \label{eq:cap-omega'}
\|v\|_{\mathcal{W}(\Omega'_T)} \leq \ca_{\mathrm{var}} (K, \Omega'_T) + \eps.
\end{equation}
We choose $\zeta_\delta \in C_0^\infty(\Omega_{\delta},[0,1])$ such
that $\zeta_{\delta} \equiv 1$ in $\Omega_{2 \delta}$ and $|\nabla
\zeta_{\delta}| \leq \tfrac{c}{\delta}$. Observe that $\|\zeta_\delta
v\|_{L^\infty(0,T;L^2(\Omega))} \leq
\|v\|_{L^\infty(0,T;L^2(\Omega'))}$. By Hardy's inequality,
Lemma~\ref{lem:hardy}, it follows that  
\begin{align*}
\int_0^T \int_{\Omega} |v|^p|\nabla \zeta_{\delta}|^p \, \d x\d t &\leq c\int_0^T \int_{\Omega_\delta \setminus \Omega_{2\delta}} \frac{|v|^p}{\delta^p} \, \d x\d t \\
&\leq c \int_0^T \int_{\Omega} \frac{|v|^p}{\mathrm{dist}(x, \partial \Omega)^p} \, \d x\d t \\
&\leq c \int_0^T \int_{\Omega} |\nabla v|^p \, \d x\d t
\end{align*}
for $c = c(n,p,\alpha) > 0$, which implies that $\|\zeta_\delta v\|_{\mathcal{V}(\Omega_T)} \leq c(n,p,\alpha) \|v\|_{\mathcal{V}(\Omega'_T)}$. Finally, for $\phi \in C_0^\infty(\Omega_T)$ we have
\begin{align*}
\left| \iint_{\Omega_T} \zeta_\delta v \partial_t \phi \, \d x \d t \right| &= \left| \iint_{\Omega'_T} v \partial_t(\zeta_\delta  \phi) \, \d x \d t \right| \\
&\leq \|\partial_t v\|_{\mathcal{V}'(\Omega'_T)} \|\zeta_\delta \phi \|_{\mathcal{V}(\Omega_T)} \\
&\leq c(n,p,\alpha) \|\partial_t v\|_{\mathcal{V}'(\Omega'_T)} \|\phi \|_{\mathcal{V}(\Omega_T)}
\end{align*}
by using again Hardy's inequality on the last line. Thus, we obtain
$\|\partial_t (\zeta_\delta v)\|_{\mathcal{W}(\Omega_T)} \leq c
\|\partial_t v\|_{\mathcal{V}'(\Omega'_T)}$ and further
$\|\zeta_\delta v \|_{\mathcal{W}(\Omega_T)} \leq c \| v
\|_{\mathcal{W}(\Omega'_T)}$. By using~\eqref{eq:cap-omega'} and
observing that $\zeta_\delta v$ is admissible for the variational capacity in $\Omega_T$, we obtain
$$
\ca_{\mathrm{var}}(K, \Omega_T) \leq \|\zeta_\delta v \|_{\mathcal{W}(\Omega_T)} \leq c \| v \|_{\mathcal{W}(\Omega'_T)} \leq c ( \ca_{\mathrm{var}} (K, \Omega'_T) + \eps ).
$$
Since $\eps  >0$ was arbitrary, we conclude the result.

\end{proof}

The variational capacity satisfies the decreasing monotone convergence property for compact sets, cf. Lemma~\ref{lem:cap-decreasing-limit}.

\begin{lemma} \label{lem:capvar-limit-compact}
Let $1<p<\infty$, $0< T\leq \infty$ and $(K_i)_{i\in \N}$ be a sequence of compact sets contained in $\Omega_T$ and satisfying $K_i \supset K_{i+1}$ for every $i \in \N$. Then
$$
\lim_{i \to \infty} \ca_{\mathrm{var}} (K_i, \Omega_T) =  \ca_{\mathrm{var}} (\cap_{i=1}^\infty K_i, \Omega_T).
$$
\end{lemma}

\begin{proof}

Observe that "$\geq$" is clear by monotonicity. For the reverse inequality, denote $K := \cap_{i=1}^\infty K_i$ and suppose that the right hand side is finite. For every $\eps > 0$ there exists $v \in C_0^\infty(\Omega \times \R)$ with $v \geq \chi_K$ such that
$$
\|v\|_{\mathcal{W}(\Omega_T)} \leq \ca_{\mathrm{var}} (K, \Omega_T) + \eps.
$$
Furthermore, by smoothness of $v$ for every $\gamma >0$ there exists $i_o = i_o(\gamma) \in \N$ such that
$$
v_\gamma = (1-\gamma)^{-1} v \geq \chi_{K_i}
$$
for every $i \in \N_{\geq i_o}$. Observe that $\ca_{\mathrm{var}} (K_i, \Omega_T) \leq \|v_\gamma\|_{\mathcal{W}(\Omega_T)} \leq (1-\gamma)^{-s} \|v\|_{\mathcal{W}(\Omega_T)}$, where $s = \max\{p,p'\}$. Thus,
$$
\ca_{\mathrm{var}} (K_i, \Omega_T) \leq (1-\gamma)^{-s} \left( \ca_{\mathrm{var}} (K, \Omega_T) + \eps \right)
$$
for every $i \in \N_{\geq i_o}$. By letting $i \to \infty$, $\gamma \to 0$ and $\eps \to 0$, respectively, the claim follows.
\end{proof}

Then we have the following subadditivity property for a certain power of the variational capacity. 

\begin{lemma} \label{lem:subadditivity}

Let $1<p<\infty$ and $0<T\leq \infty$. Then, $\ca_{\mathrm{var}}^s :=
\left( \ca_{\mathrm{var}}\right)^s$ satisfies the countable subadditivity property, where $s = (\max\{p,p'\})^{-1}$. That is, if $E_i$ is a subset of $\Omega_T$ for every $i \in \N$, then
$$
\ca_{\mathrm{var}}^s \left(\bigcup_{i=1}^\infty E_i , \Omega_T\right) \leq \sum_{i =1}^\infty \ca_{\mathrm{var}}^s ( E_i , \Omega_T).
$$
\end{lemma}

\begin{proof}

We show the claim in the case $1<p<2$, in which $s=\frac{1}{p'}$. For compact sets $K_1,K_2 \subset \Omega_T$ and any given $\eps > 0$ we find $u_i \in C_0^\infty(\Omega \times \R)$ such that $u_i \geq \chi_{K_i}$ and $\|u_i\|_{\mathcal{W}(\Omega_T)} < \ca_{\mathrm{var}}(K_i,\Omega_T) + (\eps/2)^{p'}$ for $i\in \{1,2\}$. Clearly the sum $u_1 + u_2$ is admissible in the definition of variational capacity of the union $K_1 \cup K_2$. Thus
\begin{align*}
&\ca_{\mathrm{var}}^\frac{1}{p'}(K_1 \cup K_2,\Omega_T) \leq \|u_1 + u_2\|_{\mathcal{W}(\Omega_T)}^\frac{1}{p'} \\
&\quad\leq \left[ \left(\sum_{i=1}^2 \|u_i\|_{\mathcal V (\Omega_T)}\right)^p + \left(\sum_{i=1}^2 \|u_i\|_{L^\infty(0,T;L^2(\Omega))}\right)^2 +\left(\sum_{i=1}^2 \|\partial_t u_i\|_{\mathcal V'(\Omega_T)}\right)^{p'} \right]^\frac{1}{p'} \\
&\quad \leq \left[ \left(\sum_{i=1}^2 \|u_i\|_{\mathcal V(\Omega_T)}^{p-1}\right)^{p'} + \left(\sum_{i=1}^2 \|u_i\|_{L^\infty(0,T;L^2(\Omega))}^\frac{2}{p'}\right)^{p'} +\left(\sum_{i=1}^2 \|\partial_t u_i\|_{\mathcal V'(\Omega_T)}\right)^{p'} \right]^\frac{1}{p'} \\
&\quad \leq \sum_{i=1}^2 \left[ \|u_i\|_{\mathcal V(\Omega_T)}^{p} + \|u_i\|_{L^\infty(0,T;L^2(\Omega))}^{2} +\|\partial_t u_i\|_{\mathcal V'(\Omega_T)}^{p'} \right]^\frac{1}{p'} \\
&\quad \leq \sum_{i=1}^2 \left[ \ca_{\mathrm{var}}(K_i,\Omega_T) + (\eps/2)^{p'}\right]^\frac{1}{p'} \\
&\quad\leq \sum_{i=1}^2 \ca_{\mathrm{var}}^\frac{1}{p'}(K_i,\Omega_T) + \eps,
\end{align*}
where we used the facts $0< p -1 < 1$ and $ 2/p' < 1$, together with
the triangle inequality for the $p'$-norm in $\R^3$. Since $\eps > 0$ was arbitrary, we have
$$
\ca_{\mathrm{var}}^\frac{1}{p'}(K_1 \cup K_2,\Omega_T) \leq \ca_{\mathrm{var}}^\frac{1}{p'}(K_1,\Omega_T) + \ca_{\mathrm{var}}^\frac{1}{p'}(K_2,\Omega_T),
$$
which implies finite subadditivity for compact sets. The proof in case $p > 2$ is analogous. By given definition, this implies also finite subadditivity for open sets as in~\cite[Proposition 2.13]{droniou}.

To conclude the proof, w.l.o.g. we assume that $\sum_{i = 1}^\infty \ca_{\mathrm{var}}^s (E_i, \Omega_T) < \infty$. Let $U_i \supset E_i$ be an open set for each $i \in \N$ such that $\ca_{\mathrm{var}}^s (U_i, \Omega_T) \leq \ca_{\mathrm{var}}^s (E_i, \Omega_T) + \eps/2^i$. Thus $\sum_{i = 1}^\infty \ca_{\mathrm{var}}^s (U_i, \Omega_T) \leq \sum_{i = 1}^\infty \ca_{\mathrm{var}}^s (E_i, \Omega_T) + \eps$. Let $U = \bigcup_{i=1}^\infty U_i$ and $K \subset U$ be a compact set. Since $\{U_i\}_{i\in \N}$ is an open cover of $K$, there exists a finite subcover $\{U_i\}_{i\in \{1,2,...,k\}}$. Now finite subadditivity for open sets imply
\begin{align*}
&\ca_{\mathrm{var}}^s(K,\Omega_T) \\
&\phantom{+}\leq \ca_{\mathrm{var}}^s\left(\bigcup_{i=1}^k U_i,\Omega_T\right) \leq \sum_{i = 1}^{k} \ca_{\mathrm{var}}^s\left(U_i,\Omega_T\right)\leq \sum_{i = 1}^\infty \ca_{\mathrm{var}}^s (E_i, \Omega_T) + \eps.
\end{align*}
By taking the supremum over $K \subset U$ and using the fact that $\bigcup_{i=1}^\infty E_i \subset U$ the claim follows.
\end{proof}

The following lemma states the upwards monotone convergence, up to a constant, for the variational capacity. For the proof we refer to Lemma~\ref{lem:cap*-increasinglim} and Remark~\ref{rem:cap*-increasinglim}.

\begin{lemma} \label{lem:capvar-upwards-conv}
Let $1<p<\infty$ and $0<T\leq \infty$ and $E_i \subset \Omega_T$ be a set satisfying $E_i \subset E_{i+1}$ for every $i \in \N$. Then
$$
\ca_{\mathrm{var}} \left( \bigcup_{i=1}^\infty E_i, \Omega_T \right) \approx \lim_{i \to \infty} \ca_{\mathrm{var}} \left( E_i, \Omega_T \right)
$$
up to a constant depending only on $n,p$ and $\alpha$.
\end{lemma}

The next two lemmas deal with the dependence of the variational
capacity on the final time $T\in(0,\infty]$ of the reference set.

\begin{lemma} \label{lem:capvar-Tinfty}
Let $\frac{2n}{n+2} < p < \infty$ and $K \subset \Omega_\infty$ be a compact set. Then
$$
\lim_{T\to \infty} \ca_{\mathrm{var}} (K,\Omega_T) = \ca_{\mathrm{var}} (K,\Omega_\infty).
$$
\end{lemma}

\begin{proof}

Since $K$ is compact there exists $T_o \in (1,\infty)$ such that $K
\subset \Omega_T$ for every $T \geq T_o/2$. The variational capacity
is well defined for all such $T$. Observe that $\ca_{\mathrm{var}}
(K,\Omega_T)$ is a nonnegative and increasing function in $T$ by Lemma~\ref{lem:capvar-prop}\,(iii), which implies that the limit exists and clearly 
$$
\lambda:= \lim_{T\to \infty} \ca_{\mathrm{var}} (K,\Omega_T) \leq \ca_{\mathrm{var}} (K,\Omega_\infty).
$$
Let $(T_i)_{i \in \N}$ be a sequence such that $T_i \geq T_o$ for
every $i \in \N$ and $T_i \to \infty$ in the limit $i \to \infty$. We
take a sequence $(\eps_i )_{i \in \N}$ with $\eps_i > 0$ and $\eps_i\to 0$, as $i \to \infty$. By definition, for every $i \in \N$ there exists $v_i \in C_0^\infty(\Omega \times \R)$ with $v_i \geq \chi_K$ such that 
$$
\|v_i\|_{\mathcal{W}(\Omega_{T_i})} \leq \ca_{\mathrm{var}} (K,\Omega_{T_i}) + \eps_i \leq \lambda + \eps_i.
$$
We use a cutoff function $\xi_i \in C_0^\infty(-\infty,T_i,[0,1])$ with $\xi_i =1$ in $(0,T_i/2)$ and $|\xi_i'| \leq 3/T_i$. Observe that $\xi_i v_i$ is admissible for $\ca_{\mathrm{var}}(K,\Omega_\infty)$ for every $i\in \N$, such that Lemma~\ref{lem:WT-Winfty-supercritical} implies
\begin{align*}
\ca_{\mathrm{var}} (K,\Omega_\infty) &\leq
\|v_i \xi_i\|_{\mathcal{W}(\Omega_\infty)} \\
&\leq \|v_i\|_{\mathcal{W}(\Omega_{T_i})} +
     \Big[\big(1+cT_i^{-\frac1p}\big)^{p'-1}-1\Big]\|\partial_t
     v_i\|_{\mathcal{V}'(\Omega_{T_i})}^{p'}\\
     &\qquad+
     cT_i^{-\frac1p}\big(1+cT_i^{-\frac1p}\big)^{p'-1}\|v_i\|_{L^\infty(0,T_i;L^2(\Omega))}^{p'}\\
     &\leq \big(1+cT_i^{-\frac1p}\big)^{p'-1} (\lambda + \eps_i) +
     cT_i^{-\frac1p}\big(1+cT_i^{-\frac1p}\big)^{p'-1} (\lambda + \eps_i)^\frac{p'}{2}.
\end{align*}
By passing to the limit $i \to \infty$, the claim follows.

\end{proof}

\begin{lemma} \label{lem:capvar-Tinfty2}
Let $1< p  < \infty$, $0<T< \infty$ and $E \subset \Omega_T$ be a set. Then, there exists a constant $c = c(n,p,\alpha) \geq 1$, such that
$$
\ca_{\mathrm{var}} (E, \Omega_T) \leq \ca_{\mathrm{var}}(E,\Omega_\infty) \leq c \ca_{\mathrm{var}}(E,\Omega_T).
$$
\end{lemma}

\begin{proof}
The first inequality is clear by Lemma~\ref{lem:capvar-prop}\,(iii). For the second inequality, we first consider the case where $K$ is a finite union of space time cylinders, whose bases are balls. By definition, for every $\eps > 0$ there exists $v \in C_0^\infty(\Omega \times \R)$ with $v \geq \chi_K$ such that 
$$
\| v \|_{\mathcal{W}(\Omega_T)} \leq \ca_{\mathrm{var}} (K, \Omega_T) + \eps.
$$
Let $\tilde v \in \mathcal{V}(\Omega \times (T,\infty))$ be a solution to
\begin{align*}
\begin{aligned}
\left\{
\begin{array}{rl}
\partial_t \tilde v - \Delta_p  \tilde v = 0 \quad &\text{ in } \Omega \times (T,\infty), \\
\tilde v = v \quad &\text{ on } \Omega \times \{T\}.
\end{array}
\right.
\end{aligned}
\end{align*}
We define $\hat v = v$ if $t \leq T$ and $\hat v = \tilde v$ if $t >
T$. According to Lemma~\ref{lem:WT-Winfty} we have $\hat
v\in\mathcal{W}(\Omega_\infty)$ with 
$$
\| \hat v \|_{\mathcal{W}(\Omega_\infty)} \leq 3 \| v \|_{\mathcal{W}(\Omega_T)}.
$$
By observing that $\hat v \in L^\infty(\Omega_\infty)$ and using Lemma~\ref{lem:W-approx}, we conclude that
$$
\ca_{\mathrm{var}}(K,\Omega_\infty) \leq c(n,p,\alpha) \ca_{\mathrm{var}}(K,\Omega_T).
$$
To deduce the inequality for arbitrary compact set $K$, we take a shrinking sequence of compact sets, $(K_i)_{i \in \N}$, for which $K_i \supset K_{i+1}$ and $K_i$ is a finite union of space time cylinders whose bases are balls for each $i \in \N$ such that $\bigcap_{i=1}^\infty K_i = K$. The proof in this case is concluded by an application of Lemma~\ref{lem:capvar-limit-compact}. The claim for arbitrary set $E$ follows from the definition.

\end{proof}

\subsection{Connection of nonlinear parabolic capacity and variational capacity} \label{subseq:capvar-cap-eq}

First, we recall that the nonlinear parabolic capacity from
Definition~\ref{def:nonlinear-capacity-compact} 
and the energy capacity from Definition~\ref{def:encap} are equivalent
for finite unions of compact space time cylinders. The proof
from~\cite[Theorem 4.1]{AKP} can be immediately extended to the whole range $1<p<\infty$
by using Lemmas~\ref{lem:cap-compact-representation} and~\ref{lem:supercal-prop}, and observing that $\widehat{R}_K$ is bounded in $\Omega_T$ and has zero boundary values on $\partial_p \Omega_T$ by Lemma~\ref{lem:balayage-K-properties}.

\begin{lemma} \label{lem:cap-eq-encap}
Let $1 < p < \infty$ and $0<T\leq \infty$ 
and $K \subset \Omega_T$ be a finite family of compact space-time
cylinders whose bases are balls. Then 
$$
c(p) \ca (K, \Omega_\infty) \leq \ca_{\mathrm{en}} (K, \Omega_T) \leq 2 \ca (K, \Omega_\infty). 
$$
\end{lemma}

Then we state and prove the result~\cite[Theorem 4.2]{AKP} in the
whole range $1<p<\infty$. The core of the proof consists of the following two lemmas.  Observe that we assume $T < \infty$ in order to apply a backwards in time problem in the proof.

\begin{lemma} \label{lem:W-leq-en}
Let $1<p<\infty$ and $0<T<\infty$. For every nonnegative supercaloric function $u \in \mathcal{V}(\Omega_T) \cap L^\infty(\Omega_T)$ there exists $v \in \mathcal{W}(\Omega_T)$ and $c = c(p) > 0$ such that $v \geq u$ a.e. in $\Omega_T$ and 
$$
\|v\|_{\mathcal{W}(\Omega_T)} \leq c \|u\|_{\mathrm{en}, \Omega_T}.
$$
\end{lemma}

\begin{proof}

We proceed as in~\cite[Theorem 4.2]{AKP}. Let $\tau \in (0,T)$ be a Lebesgue instant of $u$, and let $w \in \mathcal{V}(\Omega \times (\tau,\infty))$ be a solution to
\begin{align}\label{extension}
\begin{aligned}
\left\{
\begin{array}{rl}
\partial_t w - \Delta_p w = 0 \quad &\text{ in } \Omega  \times (\tau, \infty),  \\
w = u\quad &\text{ on } \Omega \times \{\tau\}.
\end{array}
\right.
\end{aligned}
\end{align}
Observe that $w \in L^\infty(\Omega \times (\tau,\infty))$ and $w$
takes the initial values $u$ in the $L^2$-sense. We define
$u^\tau\in\mathcal{V}(\Omega_\infty)$ by letting $u^\tau = u$ if $t
\leq \tau$ and $u^\tau = w$ if $t > \tau$. The equation~\eqref{extension} implies
$$
\| u^\tau \|_{\mathrm{en},\Omega_\infty} \leq 3 \| u \|_{\mathrm{en},\Omega_T}.
$$
Note that by the comparison principle, $\tau \mapsto u^\tau$ is an
increasing mapping. By taking an increasing sequence of Lebesgue instants $\{\tau_i\}$ with $\tau_i \xrightarrow{i\to \infty} T$, we have that $\bar u := \lim_{i\to \infty}u_i := \lim_{i\to \infty} u^{\tau_i}$ is a bounded supercaloric function, since $ u_i \leq \|u\|_{L^\infty(\Omega_T)}$ in $\Omega_T$ uniformly in $i \in \N$. Furthermore, $\bar u = u$ in $\Omega_T$ and 
$$
\| \bar u \|_{\mathrm{en},\Omega_\infty} \leq 2 \| u \|_{\mathrm{en},\Omega_T}.
$$
We denote $\bar u$ by $u$ and take a Lebesgue instant $\tau \in
(T,\infty)$ of $u$. Then we choose $v$ as a weak solution of the backwards in time equation 
\begin{align*}
\begin{aligned}
\left\{
\begin{array}{rl}
-\partial_t v - \Delta_p v = -2 \Delta_p u \quad &\text{ in } \Omega_\tau, \\
v = u\quad &\text{ on } \Omega \times \{\tau\}, \\
v = 0 \quad &\text{ on } \partial \Omega \times (0,\tau).
\end{array}
\right.
\end{aligned}
\end{align*}
Observe that since $u \in \mathcal{V}(\Omega_\tau)$, it follows that
$-2 \Delta_p u \in \mathcal{V}'(\Omega_\tau)$. In case $\frac{2n}{n+2}
< p <\infty$ this already implies that the solution $v \in
\mathcal{W}(\Omega_\tau)$ exists by~\cite[Chapter III.4, Proposition
4.1]{Showalter}. If $1<p\leq \frac{2n}{n+2}$, we observe that
$\mathcal{V}'(\Omega_\tau) \subset L^{p'}(0,\tau; (W^{1,p}_0(\Omega)
\cap L^2(\Omega))')$ which implies that a solution with $v \in
L^{p}(0,\tau; W^{1,p}_0(\Omega) \cap L^2(\Omega))$ exists as well by
the previous result. This further implies that $v \in
C([0,\tau];L^2(\Omega))$ by~\cite[Chapter III.1, Proposition 1.2]{Showalter}.

We test the weak formulation of the problem above by $( v_\eps
\chi_{h}(t))_\eps$, where the subscript $\eps$ denotes the standard
mollification in time and $\chi_h\in C^\infty_0(0,\tau)$, $h>0$, is an
approximation of $\chi_{(t_1,\tau)}$ for some $t_1\in(0,\tau)$. In this way, 
we obtain
\begin{align*}
&\frac12 \iint_{\Omega_\tau} v^2 \chi_h' \, \d x \d t +
                 \iint_{\Omega_\tau} \chi_h |\nabla v|^p \, \d x \d t
  \\
  &\qquad= 2 \iint_{\Omega_\tau}\chi_h |\nabla u|^{p-2} \nabla u \cdot \nabla v \, \d x \d t \\
&\qquad\leq \frac12 \iint_{\Omega_\tau}\chi_h |\nabla v|^p \, \d x \d t 
+ c(p) \iint_{\Omega_\tau}\chi_h |\nabla u|^p \, \d x \d t
\end{align*}
after passing to the limit $\eps \to 0$ and using Young's inequality. By passing to the limit $h \to 0$, we obtain
\begin{align*}
\int_{\Omega} &v^2(x,t_1) \, \d x + \iint_{\Omega \times (t_1,\tau)} |\nabla v|^p \, \d x \d t\\
&\leq c(p) \iint_{\Omega_\tau} |\nabla u|^p \, \d x \d t + \int_{\Omega} u^2(x,\tau) \, \d x \leq c(p) \|u\|_{\mathrm{en},\Omega_\tau},
\end{align*}
for a.e. $t_1 \in (0,\tau)$, which implies 
$$
\|v\|_{\mathrm{en},\Omega_\tau} \leq c(p) \|u\|_{\mathrm{en},\Omega_\tau} \leq c(p) \|u\|_{\mathrm{en},\Omega_T}.
$$
Furthermore,
$$
\|\partial_t v\|_{\mathcal{V}'(\Omega_T)}^{p'} \leq c(p) \left(\|v\|_{\mathcal{V}(\Omega_T)}^{p} + \|u\|_{\mathcal{V}(\Omega_T)}^{p}\right).
$$
By combining the estimates, it follows that
$$
\|v\|_{\mathcal{W}(\Omega_T)} \leq c(p)  \|u\|_{\mathrm{en},\Omega_T}.
$$
Furthermore, since $u$ is a weak supersolution, we have
$-2\Delta_pu\ge -\partial_tu-\Delta_pu$ in the distributional
sense. Therefore, the comparison principle implies 
$v \geq u$ a.e. in $\Omega_T$.

\end{proof}

\begin{lemma} \label{lem:en-leq-W}
Let $1<p<\infty$, $0<T\leq \infty$ and $v \in C_0^\infty(\Omega \times \R,\R_{\ge0})$. Then there exist a nonnegative supercaloric function $u \in \mathcal{V}(\Omega_T) \cap C(\overline{\Omega_T})$ with $u \geq v$ in $\Omega_T$ and a constant $c = c(p) > 0$ such that
$$
\|u\|_{\mathrm{en}, \Omega_T} \leq c \|v\|_{\mathcal{W}(\Omega_T)}
$$ 
holds true.
\end{lemma}

\begin{proof}
Since $v \in C_0^\infty(\Omega \times \R,\R_{\geq 0})$, there exists a
continuous solution (attaining the boundary values given by $v$
continuously) to the obstacle problem with $v$ acting as an obstacle
by Proposition~\ref{prop:obstacle}. Let us denote this solution by
$u$. 
Observe that by definition $u\ge v$ and
$u$ is a weak supersolution in $\Omega_T$ by Lemma~\ref{lem:supercal-prop}, such that it satisfies~\eqref{e.weak_measuresol} with a Riesz measure $\mu_u$.

Let us use a test function $\varphi =
(((u-v-\delta)_\eps)_+\chi_{h,\tau})_\eps$ for $\delta, h > 0$, $\eps \in
(0,h/2)$ and $\tau \in (h,T-h)$ in~\eqref{e.weak_measuresol}, where
$(\cdot)_\eps$ denotes the standard mollification in time direction
and $\chi_{h,\tau}\in C^\infty_0(0,T)$ are approximations of
$\chi_{(0,\tau)}$. This yields
\begin{align*}
\iint_{\Omega_T} &\partial_t u_\eps ((u-v-\delta)_\eps)_+ \chi_{h,\tau} \, \d x \d t \\
&+ \iint_{\Omega_T} \chi_{h,\tau} (|\nabla u|^{p-2} \nabla u)_\eps \cdot \nabla ((u-v-\delta)_\eps)_+  \, \d x \d t = \iint_{\Omega_T} \varphi \, \d \mu_u. 
\end{align*}
For the first term we have
\begin{align*}
\iint_{\Omega_T} \partial_t u_\eps ((u-v-\delta)_\eps)_+ \chi_{h,\tau} \, \d x \d t &= \tfrac12 \iint_{\Omega_T}\chi_{h,\tau}  \partial_t ((u-v-\delta)_\eps)_+^2  \, \d x \d t \\
&\quad + \iint_{\Omega_T}\chi_{h,\tau}  \partial_t v_\eps ((u-v-\delta)_\eps)_+  \, \d x \d t \\
&= - \tfrac12 \iint_{\Omega_T}\chi_{h,\tau}'  ((u-v-\delta)_\eps)_+^2  \, \d x \d t \\
&\quad + \iint_{\Omega_T}\chi_{h,\tau}  \partial_t v_\eps ((u-v-\delta)_\eps)_+  \, \d x \d t.
\end{align*}
The last term can be estimated as
\begin{align*}
\iint_{\Omega_T}&\chi_{h,\tau}  \partial_t v_\eps ((u-v-\delta)_\eps)_+  \, \d x \d t \\ &\geq - \|\partial_t v\|_{\mathcal{V}'(\Omega_T)} \|u-v\|_{\mathcal{V}(\Omega_T)} \\
&\geq - c(\sigma,p) \|\partial_t v\|_{\mathcal{V}'(\Omega_T)}^{p'} - \sigma \|u\|_{\mathcal{V}(\Omega_T)}^p - \sigma \|v\|_{\mathcal{V}(\Omega_T)}^p
\end{align*}
for every $\sigma > 0$. Now we obtain
\begin{align*}
\liminf_{\eps,\delta \to 0} &\iint_{\Omega_T} \partial_t u_\eps ((u-v-\delta)_\eps)_+ \chi_{h,\tau} \, \d x \d t \\
&\ge - \tfrac12 \iint_{\Omega_T}\chi_{h,\tau}'  (u-v)^2  \, \d x \d t- c(\sigma,p) \|\partial_t v\|_{\mathcal{V}'(\Omega_T)}^{p'}\\
&\quad - \sigma \|u\|_{\mathcal{V}(\Omega_T)}^p - \sigma \|v\|_{\mathcal{V}(\Omega_T)}^p.
\end{align*}
By Young's inequality it follows that 
\begin{align*}
\lim_{\eps, \delta \to 0} &\iint_{\Omega_T} \chi_{h,\tau} (|\nabla u|^{p-2} \nabla u)_\eps \cdot \nabla ((u-v-\delta)_\eps)_+  \, \d x \d t \\
&= \iint_{\Omega_T} \chi_{h,\tau}( |\nabla u|^{p}  - |\nabla u|^{p-2} \nabla u \cdot \nabla v ) \, \d x \d t \\
&\geq \tfrac{1}{p} \iint_{\Omega_T} \chi_{h,\tau} |\nabla u|^{p} \, \d x \d t -  \tfrac{1}{p} \iint_{\Omega_T} \chi_{h,\tau} |\nabla v|^{p} \, \d x \d t.
\end{align*}
Moreover, by continuity of $u$ and $v$, we have $\varphi \xrightarrow{\eps,\delta \to 0} 0$ uniformly on $\{ u = v \}$, and $\spt( \mu_u) \subset \{u=v\}$ by Proposition~\ref{prop:obstacle} (ii). Thus, by combining the estimates above we obtain
\begin{align} \label{eq:en-W-prelim-est}
- \tfrac12 \iint_{\Omega_T}&\chi_{h,\tau}'  (u-v)^2  \, \d x \d t + \tfrac{1}{p} \iint_{\Omega_T} \chi_{h,\tau} |\nabla u|^{p} \, \d x \d t \nonumber \\
&\leq c(\sigma,p) \|\partial_t v\|_{\mathcal{V}'(\Omega_T)}^{p'} + \sigma \|u\|_{\mathcal{V}(\Omega_T)}^p + (\sigma + \tfrac{1}{p}) \|v\|_{\mathcal{V}(\Omega_T)}^p.
\end{align}
Observe that 
\begin{align*}
\lim_{h \to 0} - \tfrac12 \iint_{\Omega_T}\chi_{h,\tau}'  (u-v)^2  \, \d x \d t = \tfrac12 \int_{\Omega\times \{\tau\}} (u-v)^2  \, \d x,
\end{align*}
because $u=v$ on $\Omega\times\{0\}$. 
Since $(u-v)^2 \geq \frac12 u^2 - v^2$, after passing to the limit $h \to 0$ in~\eqref{eq:en-W-prelim-est} we obtain
\begin{align*}
\tfrac14 \int_{\Omega \times \{\tau\}}& u^2  \, \d x + \tfrac{1}{p} \iint_{\Omega_\tau} |\nabla u|^{p} \, \d x \d t \nonumber \\
&\leq c(\sigma,p) \|\partial_t v\|_{\mathcal{V}'(\Omega_T)}^{p'} + \sigma \|u\|_{\mathcal{V}(\Omega_T)}^p + (\sigma + \tfrac{1}{p}) \|v\|_{\mathcal{V}(\Omega_T)}^p + \tfrac12 \int_{\Omega \times \{\tau\}}& v^2  \, \d x \\
&\leq c(\sigma, p) \|v\|_{\mathcal{W}(\Omega_T)} + \sigma \|u\|_{\mathcal{V}(\Omega_T)}^p.
\end{align*}
By considering separately the terms on the left hand side, we can let
$\tau \to T$ in the second term and take the supremum over $\tau \in (0,T)$ in the first, which after choosing $\sigma = \frac{1}{4p}$ implies 
\begin{equation*}
\|u\|_{\mathrm{en},\Omega_T} = \tfrac12 \sup_{t\in (0,T)}\int_{\Omega \times \{t\}} u^2  \, \d x + \iint_{\Omega_T} |\nabla u|^{p} \, \d x \d t \leq c(p)  \|v\|_{\mathcal{W}(\Omega_T)},
\end{equation*}
concluding the claim.

\end{proof}

The combination of the two preceding lemmas yields the following result.
\begin{lemma} \label{lem:varcap-eq-encap}
Let $0<T<\infty$ and $K \subset \Omega_T$ be a compact set consisting
of a finite union of space time cylinders, whose bases are balls. Then,
$$
c(p) \ca_{\mathrm{en}} (K,\Omega_T) \leq \ca_{\mathrm{var}} (K,\Omega_T) \leq c(n,p,\alpha) \ca_{\mathrm{en}} (K,\Omega_T).
$$
\end{lemma}

\begin{proof}

First we show the second inequality. Suppose that the right hand side is finite since otherwise the claim is clear. For any $\eps > 0$ we can find a supercaloric function $u \in \mathcal{V}(\Omega_T)$ with $u \geq \chi_K$ such that 
$$
\|u\|_{\mathrm{en},\Omega_T} \leq \ca_{\mathrm{en}} (K,\Omega_T) + \eps.
$$  
By truncation, we may assume that $u$ is bounded, such that
Lemma~\ref{lem:W-leq-en} implies that there exists $v \in
\mathcal{W}(\Omega_T)$ satisfying $v \geq u$ a.e. in $\Omega_T$ and 
$$
\|v\|_{\mathcal{W}(\Omega_T)} \leq c(p) \|u\|_{\mathrm{en},\Omega_T}.
$$ 
Since $\eps > 0$ was arbitrary, the second inequality follows by Lemma~\ref{lem:W-approx}.

Then let us consider the first inequality and suppose again that the right hand side is finite. For every $\eps > 0$ there exists $v\in C_0^\infty(\Omega \times \R)$ with $v \geq \chi_K$ such that 
$$
\|v\|_{\mathcal{W}(\Omega_T)}\leq \ca_{\mathrm{var}} (K,\Omega_T) + \eps.
$$
In this case, Lemma~\ref{lem:en-leq-W} concludes the claim.
\end{proof}

Now we are in a position to give the

\subsubsection*{Proof of Theorem~\ref{thm:capvar-eq-cap}. }
For a given compact set $K\subset\Omega_T$, we
take a sequence $\{K_i\}_{i \in \N}$ of decreasing sets in $\Omega_T$ such that $K_i$ is a finite union of space time cylinders whose bases are balls for every $i \in \N$ and
$$
\bigcap_{i=1}^\infty K_i = K.
$$
In the case $0<T<\infty$, we infer from Lemmas~\ref{lem:cap-eq-encap} and~\ref{lem:varcap-eq-encap} that 
$$
\ca_{\mathrm{var}} (K_i,\Omega_T) \approx \ca_{\mathrm{en}} (K_i,\Omega_T) \approx \ca (K_i,\Omega_\infty)
$$
for every $i\in\N$. 
By Lemmas~\ref{lem:cap-decreasing-limit}
and~\ref{lem:capvar-limit-compact} the claim in the case $T < \infty$
follows by letting $i \to \infty$. The case $T = \infty$ follows now
from Lemma~\ref{lem:capvar-Tinfty} by passing to the limit $T \to
\infty$ in case $\frac{2n}{n+2} < p <\infty$. In the case $1 < p \leq
\frac{2n}{n+2}$, the claim follows by using
Lemma~\ref{lem:capvar-Tinfty2}. 
\hfill \qed

\section{Capacities of explicit sets} \label{sec:cap-sets}

Recall the definition of elliptic capacity from Definition~\ref{def:elliptic-cap}. First we state two results~\cite[Theorem 2.16 \& Theorem 2.15]{droniou} related to specific sets of null capacity. Since our setting and statements are slightly different compared to~\cite{droniou}, we also give proofs of the parts that are not completely identical.

\begin{lemma}
Let $1 < p < \infty$, $E \subset \Omega$ be a set and $0\leq t_1 < t_2 \leq T \leq \infty$. Then 
$$
\ca_{\mathrm{e}}(E,\Omega) = 0 \quad \text{ if and only if } \quad \ca_{\mathrm{var}}(E\times (t_1,t_2),\Omega_T) = 0.
$$
\end{lemma}

\begin{proof}
By Lemma~\ref{lem:capvar-Tinfty2}, it is sufficient to consider the case $T < \infty$.
Let $\ca_{\mathrm{e}}(E,\Omega) = 0$. Then, for every $\eps > 0$ there
exists an open set $U \supset E$ with $U \subset \Omega$ such that
$\ca_{\mathrm{e}}(U,\Omega) < \eps/2$, which implies that for every
compact set $K \subset U$ there exists $v \in C(\Omega, [0,1]) \cap
W^{1,p}_0(\Omega)$ with $v =1$ in a neighborhood of $K$ such that 
$$
\|v\|_{W^{1,p}_0(\Omega)} < \eps.
$$
Observe that when $p \geq 2$ we have $\|v\|_{L^2(\Omega)} \leq
c|\Omega|^{(p-2)/2p} \|v\|_{W^{1,p}_0(\Omega)} < c|\Omega|^{(p-2)/2p}
\eps$ by Poincar\'e's inequality and when $p < 2$ we have
$\|v\|_{L^2(\Omega)}^2 \leq \|v\|_{L^p(\Omega)}^p\leq
c\|v\|_{W^{1,p}_0(\Omega)}^{p} < c\eps^{p}$ since $0 \leq v \leq
1$. Thus $\|v\|_{L^2(\Omega)}^2 + \|v\|_{W^{1,p}_0(\Omega)}^p <
c(p,\Omega) (\eps^2 + \eps^p)$. Let $\widehat K$ be an arbitrary
compact set contained in $U \times (t_1,t_2)$. Then there exists a
compact set $K \subset U$ and $t_1 < \tau_1 < \tau_2 < t_2$ such that $\widehat K \subset K \times [\tau_1, \tau_2]$. By using $u(x,t) = v(x)$ in the definition of $\ca_{\mathrm{var}}'(K \times [\tau_1,\tau_2],\Omega_T)$ and Lemma~\ref{lem:cap-cap'-equiv} we have 
\begin{align*}
\ca_{\mathrm{var}}(\widehat K,\Omega_T) &\leq c \ca_{\mathrm{var}}'(\widehat K,\Omega_T) \leq c \ca_{\mathrm{var}}'(K \times [\tau_1,\tau_2],\Omega_T) \\
&\leq c \|u\|_{\mathcal{W}(\Omega_T)} < c (\eps^2 + \eps^p),
\end{align*}
where $c = c(n,p,\alpha,\Omega)$. Since $\widehat{K} \subset U \times (t_1,t_2)$ was arbitrary, we obtain that
$$
\ca_{\mathrm{var}}(E \times (t_1,t_2),\Omega_T) \leq \ca_{\mathrm{var}}(U \times (t_1,t_2),\Omega_T) \leq c (\eps^2 + \eps^p).
$$
Since $\eps > 0$ was arbitrary, the proof of the "only if" direction is completed.

Then suppose that $\ca_{\mathrm{var}}(E\times (t_1,t_2),\Omega_T) =
0$. From definition it follows that there exists an open set $U
\supset E \times (t_1,t_2)$ such that
$\ca_{\mathrm{var}}(U,\Omega_T) < \eps$.
Fix $t_1 < \tau_1 < \tau_2
< t_2$. For every $x \in E$ there exists an open set $U_x \subset \Omega$ such that $\{x\} \times (\tau_1,\tau_2) \subset U_x \times (\tau_1,\tau_2) \subset U$. Let $\widehat{U} = \bigcup_{x \in E} U_x$ (that is an open set), which implies that $E \subset \widehat{U} \subset \Omega$ and $\widehat{U} \times (\tau_1, \tau_2) \subset U$. Thus 
$$
\ca_{\mathrm{var}}(\widehat{U} \times (\tau_1,\tau_2),\Omega_T) \leq \ca_{\mathrm{var}}(U,\Omega_T) < \eps.
$$
In particular, for every compact subcylinder
$K\times[\sigma_1,\sigma_2]\subset \widehat{U} \times
(\tau_1,\tau_2)$, we can find $v \in C^\infty_0(\Omega\times\R)$ with $v \geq \chi_{K\times [\sigma_1,\sigma_2]}$ in $\Omega_T$ such that $\|v\|_{\mathcal{W}(\Omega_T)} < \eps$. By defining
$$
u = \frac{1}{\sigma_2-\sigma_1} \int_{\sigma_1}^{\sigma_2} v \, \d t,
$$
we have that $u \in C^\infty_0(\Omega)$ and $u \geq \chi_{K}$ in
$\Omega$. This implies that 
$$
  \ca_{\mathrm{e}}(K,\Omega)
\le
\|u\|_{W^{1,p}_0(\Omega)}^p \leq \frac{T^{p-1}}{(\sigma_2-\sigma_1)^p} \|v\|_{\mathcal{W}(\Omega_T)} < \frac{T^{p-1}}{(\sigma_2-\sigma_1)^p} \eps.
$$
Since the compact subset $K\subset\widehat U$ and
$[\sigma_1,\sigma_2]\subset(\tau_1,\tau_2)$ were arbitrary and $E \subset \widehat{U}$, this implies that 
$$
\ca_{\mathrm{e}}(E,\Omega) \leq \ca_{\mathrm{e}} (\widehat{U},\Omega) \le \frac{T^{p-1}}{(\tau_2-\tau_1)^p} \eps.
$$
Since $\eps > 0$ was arbitrary the claim follows.

\end{proof} 

For a set $E \subset \R^{n+1}$, we denote $\pi_t(E) = \left\{x : (x,t) \in E \right\}$.

\begin{lemma} \label{lem:cap-slice-null}
Let $E \subset \Omega$ be a set and $t_o \in (0,T)$. Then
$$
\ca_{\mathrm{var}} (E \times \{t_o\} , \Omega_T) = 0 \quad \text{ if and only if }\quad |E|_n = 0,
$$ 
in which $|\cdot|_n$ denotes the $n$-dimensional Lebesgue measure.
\end{lemma}

\begin{proof}

Suppose that $\ca_{\mathrm{var}}(E \times \{t_o\}, \Omega_T) = 0 $. Then, for every $\eps > 0$ there exists an open set $U \supset E \times \{t_o\}$ with $U \subset \Omega_T$ such that $\ca_{\mathrm{var}}(U , \Omega_T) < \eps/2$. Thus,
$$
\ca_{\mathrm{var}}(K \times \{t_o\}, \Omega_T) \leq \ca_{\mathrm{var}}(U , \Omega_T) < \eps/2
$$
for every compact $K \subset \pi_{t_o}(U)$. By definition of variational capacity, there exists $v \in C_0^\infty(\Omega \times \R)$ with $v \geq \chi_{K \times \{t_o\}}$ such that $\|v\|_{\mathcal{W}(\Omega_T)} < \eps$ . This implies 
$$
|K|_n \leq \int_{K \times\{t_o\}} v^2 \, \d x \leq \|v\|_{L^\infty(0,T;L^2(\Omega))}^2 \leq \|v\|_{\mathcal{W}(\Omega_T)} < \eps.
$$
Since $\pi_{t_o} (U)$ is open in $\R^n$, by inner regularity of the Lebesgue measure we have that 
$$
|\pi_{t_o}(U)|_n < \eps.
$$
Since $\pi_{t_o} (U)\supset E$ and $\eps > 0$ was arbitrary, it
follows that $|E|_n = 0$. This proves the ``only if'' direction. 

The proof of the other direction follows the lines of~\cite[Theorem
2.15]{droniou} by using $\ca_{\mathrm{var}}^*$ instead of
$\ca_{\mathrm{var}}$, see Appendix~\ref{appendix}.

\end{proof}

We show that the $(n+1)$-dimensional Lebesgue measure of a set is bounded by the (variational) capacity of the set. The proof is analogous to~\cite[Lemma 6.1]{AS}. 

\begin{lemma} \label{lem:meas-leq-cap}
Let $E \subset \Omega_T$ be a set and $0<T \leq \infty$. Then there exists $c = c(n,p) >0$ such that
$$
|E| \leq c \ca_{\mathrm{var}} (E, \Omega_T)^\frac{n+p}{n}.
$$
\end{lemma}

\begin{proof}
Let $K \subset \Omega_T$ be an arbitrary compact set and let $\{K_i\}$
be a decreasing sequence of compact sets, such that $K_i$ for every $i
\in \N$ is a finite union of space time cylinders whose bases are balls, and $K = \cap_{i=1}^\infty K_i$. Fix $\eps > 0$ and let $v_i \in \mathcal{V}(\Omega_T)$ be a (bounded) supercaloric function with $v_i \geq \chi_{K_i}$ and
$$
\|v_i\|_{\mathrm{en},\Omega_T} \leq \ca_{\mathrm{en}} (K_i, \Omega_T) + \eps.
$$
Since we can assume that $v_i \leq 1$ in $\Omega_T$, by Sobolev inequality we have
\begin{align*}
|K_i| &\leq \iint_{\Omega_T} v_i^\frac{p(n+2)}{n} \, \d x\d t \leq c \iint_{\Omega_T} |\nabla v_i|^p \, \d x\d t \left( \sup_{0<t<T} \int_\Omega v_i^2 \, \d  x \right)^\frac{p}{n} \\
&\leq c \|v_i\|_{\mathrm{en},\Omega_T}^\frac{n+p}{n} \leq c \left( \ca_{\mathrm{en}} (K_i, \Omega_T) + \eps \right)^\frac{n+p}{n}
\end{align*}
for $c = c(n,p) > 0$. 
By using Lemmas~\ref{lem:varcap-eq-encap} and~\ref{lem:capvar-limit-compact}, and passing to the limit $i \to \infty$, we have 
$$
|K| \leq c(n,p)  \ca_{\mathrm{var}} (K, \Omega_T)^\frac{n+p}{n}.
$$
The claim follows by definition of variational capacity and properties of the Lebesgue (outer) measure.
\end{proof}

We recall the following lemmas from~\cite[Theorem 5.2 \& Lemma 5.3]{AKP}. The first follows directly from the definitions.

\begin{lemma} \label{lem:capvar-geq-ellipticint}
Suppose that $1<p<\infty$ and let $K \subset \Omega_T$ be a compact set. Then
$$
\int_0^T \ca_{\mathrm{e}}(\pi_t(K),\Omega) \, \d t \leq \ca_{\mathrm{var}}(K,\Omega_T).
$$
\end{lemma}

\begin{lemma} \label{lem:capvar-lowerbound}
Let $1<p<n$ and $Q_{\rho,\tau} = B_\rho(x_o) \times (t_o - \tau, t_o) \Subset \Omega_T$. Then, there exists $c= c(n,p) > 0$ such that
$$
\ca_{\mathrm{var}}\left(\overline{Q}_{\rho,\tau},\Omega_T\right) \geq c \tau \rho^{n-p}.
$$
If $\Omega = B_{2\rho}(x_o)$, the estimate holds for every $1<p<\infty$.
\end{lemma}

\begin{proof}

Observe that in case $1<p<n$, we have that
$$
\ca_{\mathrm{e}} \left(\overline{B_\rho(x_o)}, \Omega \right) \geq \ca_{\mathrm{e}} \left(\overline{B_\rho(x_o)}, \R^n \right) = c(n,p) \rho^{n-p}.
$$
Furthermore, for every $1<p<\infty$ we have
$$
\ca_{\mathrm{e}} \left(\overline{B_\rho(x_o)}, B_{2\rho}(x_o) \right) = c(n,p) \rho^{n-p}.
$$
By Lemma~\ref{lem:capvar-geq-ellipticint} the claims follow.

\end{proof}

The next result states a subadditivity type inequality for space time cylinders. In this form, it will also be useful in Section~\ref{sec:cap-hausdorff}.

\begin{lemma} \label{lem:cap-union-cylinder-upperbound}
Let $1<p<\infty$ and $Q'_i := Q'_{\rho_i,\tau_i}(z_i) = \overline{B_{\rho_i}(x_i)} \times \left[ [t_i - \tau_i, t_i] \cap (0,T) \right]$ such that $B_{2\rho_i}(x_i)\Subset \Omega_T$ for every $i \in \N$. Then, there exists $c= c(n,p,\alpha) > 0$ such that
$$
\ca_{\mathrm{var}}\left(\bigcup_{i=1}^\infty Q'_i, \Omega_T\right) \leq c \sum_{i=1}^\infty (\rho_i^n + \tau_i \rho_i^{n-p}).
$$
\end{lemma}

\begin{proof}

Fix $k \in \N$. We consider radii $\rho_i+\eps$ for some small $\eps> 0$, but for simplicity we denote it by $\rho_i$. Furthermore, suppose that $[t_i - \tau_i, t_i] \subset (0,T)$, such that if $t_i-\tau_i = 0$ consider $\tau_i-\delta$ instead of $\tau_i$ and if $t_i = T$ consider $t_i - \delta$ instead of $t_i$ for small $\delta >0$. Let $w_i \in C_0^\infty(B_{2\rho_i},[0,1])$ such that $w_i = 1$ in $\overline{B_{\rho_i}}$ and $|\nabla w_i| \leq \frac{2}{\rho_i}$. Furthermore, let $C^\infty(\R,[0,1]) \ni\xi_i(t) \approx \chi_{[t_i-\tau_i-\eps,t_i+\eps]}$ with $\|\xi'\|_{L^1(\R)} \leq 2$. Denote $v_i = \xi_i w_i$. Observe that $v_i \leq 1$. Then, it follows that 
\begin{align*}
\iint_{\Omega_T} |\nabla v_i|^p \, \d x \d t = \iint_{\Omega_T} \xi_i^p|\nabla w_i|^p \, \d x \d t &\leq (\tau_i + 2\eps) \left(\tfrac{2}{\rho_i} \right)^p |B_{2\rho_i}| \\
&= c(n,p) (\tau_i + 2\eps)\rho_i^{n-p}
\end{align*}
and 
$$
\int_\Omega v_i^2(x,t) \, \d x \leq 
\int_\Omega w_i^2(x,t) \, \d x \leq |B_{2\rho_i}| \leq  c(n) \rho_i^n
$$
for every $t\in(0,T)$. 
Now we have $\partial_t v_i = \xi_i' w_i
\in L^1(\Omega_T)$, and 
$$
\iint_{\Omega_T} |\partial_t v_i| \, \d x \d t \leq c(n) \rho_i^n.
$$
For a fixed $k\in\N$, denote $g = \max\{v_1,v_2,...,v_k\}$ and observe that 
$$
\iint_{\Omega_T} |\nabla g|^p \, \d x \d t \leq \sum_{i=1}^k
\iint_{\Omega_T} |\nabla v_i|^p \, \d x \d t
\le
c(n,p)\sum_{i=1}^k (\tau_i + 2\eps)\rho_i^{n-p},
$$
$$
\int_{\Omega} g^2(x,t) \, \d x \leq \sum_{i=1}^k \int_{\Omega}
v_i^2(x,t) \, \d x
\le c(n)\sum_{i=1}^k\rho_i^n
$$
for every $t\in(0,T)$, and
$$
\iint_{\Omega_T} |\partial_t g| \, \d x \d t \leq \sum_{i=1}^k
\iint_{\Omega_T} |\partial_t v_i| \, \d x \d t
\le c(n)\sum_{i=1}^k\rho_i^n.
$$
By Lemma~\ref{lem:petitta}, using the trivial decomposition
$\partial_t g=0+\partial_tg\in \mathcal{V}'(\Omega_T)+L^1(\Omega_T)$,  there exists $\mathcal{W}(\Omega_T) \ni u \geq g$ a.e. in $\Omega_T$ such that
\begin{align*}
  \|u\|_{\mathcal{W}(\Omega_T)}
  &\leq
  c(p) \bigg( \|g\|_{\mathcal{V}(\Omega_T)}^p +
             \|g\|_{L^\infty(\Omega_T)} \|\partial_t g\|_{L^1(\Omega_T)} + \|g\|_{L^\infty(0,T;L^2(\Omega))}^2 \bigg)\\
  &\leq
  c(n,p) \sum_{i=1}^k \left( \rho_i^n + (\tau_i + 2\eps)\rho_i^{n-p} \right).
\end{align*}
Observe that $g \geq 1$ in a neighborhood of $\bigcup_{i=1}^k \overline{Q_i'}$, which implies that $g$ is admissible for $\ca_{\mathrm{var}}'$. By Lemma~\ref{lem:cap-cap'-equiv} and passing to the limit in $\eps \to 0$, we obtain
$$
\ca_{\mathrm{var}}\left(\bigcup_{i=1}^k Q_i', \Omega_T\right) \leq c \sum_{i=1}^k (\rho_i^n + \tau_i \rho_i^{n-p}),
$$
where $c=c(n,p,\alpha)$. 
Letting $\delta \to 0$ and $k \to \infty$ by using Lemma~\ref{lem:capvar-upwards-conv}, the claim follows.

\end{proof}

The following lemma gives an estimate for the variational capacity of a
graph contained in a space-time cylinder, which was proved
in~\cite{AKP} in case $p>2$ and $T= \infty$. For convenience we
include the proof of the lemma in our case.

\begin{lemma} \label{lem:capvar-H-est}
Let $1<p<\infty$ and $Q^+_{\rho,\tau} = B_\rho(x_o) \times (t_o,t_o +
\tau)$ be a cylinder such that $Q_{2\rho,\tau}^+ \Subset \Omega_T$ and let
$$
\mathcal H = \left\{ (x,h(x)) : x \in \overline{B_\rho(x_o)} \right\},
$$
in which $h \in C(\R^n,[t_o,t_o+\tau])$ satisfies $h(x) = t_o$ for
every $x \in \partial B_\rho(x_o)$.
Then
$$
c_1 \left( \int_0^T \ca_{\mathrm{e}} (\pi_t(\mathcal{H}),\Omega) \, \d t + \rho^n \right) \leq \ca_{\mathrm{var}} (\mathcal{H},\Omega_T) \leq c_2 \left( \rho^n + \tau \rho^{n-p} \right)
$$
holds true for constants $c_1 = c_1(n,p)>0$ and $c_2 = c_2(n,p,\alpha) > 0$.
\end{lemma}
\begin{proof}
Lemma~\ref{lem:capvar-prop}\,(ii) together with
Lemma~\ref{lem:cap-union-cylinder-upperbound} implies the upper bound,
since $\mathcal{H}\subset\overline{Q_{\rho,\tau}^+}$. 

By definition of $\ca_{\mathrm{en}}$, for every $\eps > 0$ there exists a bounded supercaloric function $v \in \mathcal{V}(\Omega_T)$ such that $v \geq \chi_{\mathcal{H}}$ and 
$$
\|v\|_{\mathrm{en},\Omega_T} \leq \ca_{\mathrm{en}} (\mathcal{H}, \Omega_T) + \eps.
$$
Define 
$$
\widetilde{\mathcal{H}} := \left\{ (x,t) : x \in \overline{B_\rho(x_o)} ,\,t\in [t_o, h(x)] \right\}.
$$

By~\cite[Lemma 2.9]{BBGP} the function
\[  \tilde v := \left\{
\begin{array}{ll}
      \min\{v,1 \} & (x,t) \in \Omega_T \setminus \widetilde{\mathcal{H}} \\
      1 & (x,t) \in  \widetilde{\mathcal{H}} \\
\end{array} 
\right. \]
is a supercaloric function since it is lower semicontinuous.

Suppose that $|\pi_{t_o} (\mathcal{H})| \geq \frac{1}{2}
|B_\rho(x_o)|$. Observe that $v \geq 1$ on
$\pi_{t_o}(\mathcal{H})$. By Lemma~\ref{lem:supercal-prop}, the
bounded 
supercaloric function $v$ is a weak supersolution. For any $s \in
(t_o,T)$ we may use $(v_\eps \chi_{t_o,s})_\eps$ as a test function in
the weak formulation of $v$, where the subscript $\eps$ denotes the 
standard mollification in time. In this way, we obtain
\begin{align*}
\frac14 |B_\rho(x_o)| &\leq \frac12 \int_{\Omega} v^2(x,t_o) \, \d x \leq \frac12 \int_{\Omega} v^2(x,s) \, \d x + \iint_{\Omega \times (t_o,s)} | \nabla v|^p \, \d x \d t \\
&\leq \|v\|_{\mathrm{en},\Omega_{T}} \leq \ca_{\mathrm{en}} (\mathcal{H}, \Omega_T) + \eps,
\end{align*}
where we also used the assumptions $v \geq 1$ on $\pi_{t_o}(\mathcal{H})$ and $|\pi_{t_o} (\mathcal{H})| \geq \frac{1}{2} |B_\rho(x_o)|$.

Then suppose that $|\pi_{t_o} (\mathcal{H})| < \frac{1}{2}
|B_\rho(x_o)|$. By continuity of $h$ there exists $\delta > 0$ such
that $|\pi_{t_o+\delta}(\widetilde{\mathcal{H}})| \geq \frac14
|B_\rho(x_o)|$. By diminishing $\delta>0$ further if necessary, we can ensure that
there exists $s \in (t_o+\delta, T)$ with $\pi_{s}
(\widetilde{\mathcal{H}}) = \varnothing$. Arguing similarly as above,
we use the weak
formulation for $\tilde v$ to conclude 
\begin{align*}
\frac18 |B_\rho(x_o)| &\leq \frac12 \int_{\Omega} \tilde v^2(x,t_o+\delta) \, \d x \leq \frac12 \int_{\Omega} \tilde v^2(x,s) \, \d x + \iint_{\Omega \times (t_o+\delta,s)} | \nabla \tilde v|^p \, \d x \d t \\
&\leq \| v\|_{\mathrm{en},\Omega_{T}} \leq \ca_{\mathrm{en}} (\mathcal{H}, \Omega_T) + \eps.
\end{align*}
Since $\eps > 0$ was arbitrary, it follows that 
$$
\rho^n \leq c(n) \ca_{\mathrm{en}} (\mathcal{H}, \Omega_T).
$$
Observe that for arbitrary compact set $K \subset \Omega_T$ we have $\ca_{\mathrm{en}} (K, \Omega_T) \leq c(n,p) \ca_{\mathrm{var}} (K, \Omega_T)$ by Lemmas~\ref{lem:capvar-limit-compact} and~\ref{lem:varcap-eq-encap}. By using Lemma~\ref{lem:capvar-geq-ellipticint} the claim follows.
\end{proof}

By the results above we obtain the following estimate for the variational capacity of a space-time cylinder.

\begin{lemma} \label{lem:capvar-cylinder}
Let $1<p<n$ and $Q_{\rho,\tau} = B_\rho(x_o) \times (t_o-\tau,t_o)$ be
a cylinder such that $Q_{2\rho,\tau} \Subset \Omega_T$. Then 
$$
\ca_{\mathrm{var}} (\overline{Q}_{\rho,\tau},\Omega_T) \approx \rho^n + \tau \rho^{n-p}
$$
up to a constant depending only on $n,p$ and $\alpha$.

If $\Omega = B_{2\rho}(x_o)$ the result holds for all $1<p<\infty$.
\end{lemma}

\begin{remark}
The result holds also for open cylinders $Q_{\rho,\tau}$ by Lemma~\ref{lem:capvar-upwards-conv}. 
\end{remark}

\begin{proof}
By Lemmas~\ref{lem:capvar-lowerbound} and~\ref{lem:cap-union-cylinder-upperbound} we obtain
$$
c^{-1} \tau \rho^{n-p} \leq \ca_{\mathrm{var}} (\overline{Q}_{\rho,\tau},\Omega_T) \leq c (\rho^n + \tau \rho^{n-p} ).
$$
By choosing $h \in C(\R^n)$ such that $h(x) = t_o$ for every $x \in \overline{B_\rho(x_o)}$ in Lemma~\ref{lem:capvar-H-est}, we obtain that 
$$
c^{-1} \rho^n \leq \ca_{\mathrm{var}}(\overline{B_\rho(x_o)}\times \{t_o\}, \Omega_T) \leq \ca_{\mathrm{var}}(\overline{Q}_{\rho,\tau}, \Omega_T),
$$
where we also used Lemma~\ref{lem:capvar-prop} (ii). Thus the claim follows.
\end{proof}

\section{Capacity and Hausdorff measure} \label{sec:cap-hausdorff}

We use parabolic cylinders of the type
\begin{equation*}
  Q_\rho(z_o):= B_\rho(x_o)\times\Lambda_\rho(t_o):= B_\rho(x_o)\times (t_o-\rho^p,t_o+\rho^p),
\end{equation*}
for $z_o=(x_o,t_o)$ and $\rho>0$.

We consider a metric defined by
$$
d_p ((x,t);(0,0)) = \max \left\{ |x|, |t|^\frac{1}{p} \right\}
$$
for $(x,t) \in \R^{n+1}$ and a parabolic diameter of a set  $E \subset \R^{n+1}$ by
$$
d_p(E) = \sup \left\{ d_p((x,t);(y,s)) : (x,t),(y,s) \in E \right\}.
$$
Moreover, for a set $E \subset \Omega_T$ we define the $s$-dimensional Hausdorff $\delta$-content with respect to $d_p$ as
$$
\mathcal{P}_\delta^s(E) = \inf \left\{ \sum_{i=1}^\infty d_p(A_i)^s: E \subset \bigcup_{i=1}^\infty A_i, A_i \subset \Omega_T, d_p(A_i) < \delta \right\}.
$$
The $s$-dimensional Hausdorff measure with respect to $d_p$ is obtained by
$$
\mathcal{P}^s (E) = \lim_{\delta \to 0} \mathcal{P}_\delta^s(E).
$$

In the following, we denote
$$
(u)_{x_o,\rho}(t) = (u)_{B_\rho(x_o)}(t) = \bint_{B_\rho(x_o)\times \{t\}} u \, \d x,
$$
and
$$
(u)_{z_o,\rho} = (u)_{Q_\rho(z_o)} = \biint_{Q_\rho(z_o)} u \, \d x\d t.
$$

First we prove a suitable gluing type lemma.

\begin{lemma}\label{lem:gluing}
  Let $u\in L^1(Q_\rho(z_o))$ be a function that satisfies $\partial_tu= \Div F$ in the
  distributional sense in $Q_\rho(z_o)$, for a vector field $F\in
  L^1(Q_\rho(z_o),\R^n)$. Then there exists a radius $\hat\rho\in
  (\frac{\rho}{2},\rho)$ such that for a.e.
  $t_1,t_2\in\Lambda_\rho(t_o)$ there holds 
\begin{align*}
 \big|(u)_{x_o,\hat\rho}(t_2) - (u)_{x_o,\hat\rho}(t_1)\big| 
	&\le
	2^{n+2}\rho^{p-1}
	\biint_{Q_{\rho}(z_o)} |F|\, \dx\dt.
\end{align*}
\end{lemma}

\begin{proof}
  Since the center of the cylinder is fixed throughout the proof, we
  omit it in the notation and write $Q_\rho:=Q_\rho(z_o)$,
  $B_\rho:=B_\rho(x_o)$ and $\Lambda_\rho:=\Lambda_\rho(t_o)$. 
We fix $t_1,t_2\in\Lambda_\rho$ with $t_1<t_2$ and consider a radius
$r \in (\frac{\rho}{2}, \rho)$.
For $\delta>0$ and $0<\varepsilon\ll 1$, we define $\xi_\epsilon\in W^{1,\infty}_0(\Lambda_\rho)$ by 
$$
	\xi_\epsilon(t)
	:=
	\left\{
	\begin{array}{cl}
	0 ,& \mbox{for $t_o-\rho^p\le t\le t_1-\epsilon$,}\\[3pt]
	\frac{t-t_1+\epsilon}{\epsilon}, & \mbox{for $t_1-\epsilon< t< t_1$,}\\[3pt]
	1 ,& \mbox{for $t_1\le t\le t_2$,}\\[3pt]
	\frac{t_2+\epsilon-t}{\epsilon}, & \mbox{for $t_2< t< t_2+\epsilon$,}\\[3pt]
	0 ,& \mbox{for $t_2+\epsilon\le t\le t_o+\rho^p$,}
	\end{array}
	\right.
$$
and a function $\Psi_\delta(x):=\psi_\delta(|x-x_o|)\in
W^{1,\infty}_0(B_{r+\delta})$ by letting
$$
	\psi_\delta(s)
	:=
	\left\{
	\begin{array}{cl}
	1 ,& \mbox{for $0\le s\le r$,}\\[3pt]
	\frac{r+\delta-s}{\delta}, & \mbox{for $r< s< r+\delta$,}\\[3pt]
	0 ,& \mbox{for $r+\delta\le s\le \rho$.}
	\end{array}
	\right.
$$
Testing the equation $\partial_tu=\div F$ with
the test function $\varphi(x,t):= \Psi_\delta(x)\xi_\epsilon(t)\in
W^{1,\infty}_0(Q_\rho)$, which can be justified by a standard
approximation argument, we deduce
\begin{equation*}
  \iint_{Q_\rho}u\xi_\eps'(t)\Psi_\delta(x)\,\dx\dt
  =
  \iint_{Q_\rho}F\cdot\nabla\Psi_\delta(x)\xi_\eps(t)\,\dx\dt.
\end{equation*}
We let $\varepsilon,\delta \downarrow 0$ and obtain
\begin{align*}
	\bigg|\int_{B_{r}} \big[u(t_2)-u(t_1)\big]\,\d x\bigg| 
	&= 
	\bigg|\int_{t_1}^{t_2}\int_{\partial B_{r}}
	F\cdot \frac{x-x_o}{|x-x_o|}\,\d \mathcal{H}^{n-1}(x) \d t\bigg| \\
        &\le
          \int_{t_1}^{t_2}\int_{\partial B_{r}}|F|\, \d\mathcal{H}^{n-1}(x)\d t
\end{align*}
for a.e. $t_1,t_2\in\Lambda_\rho$ and
a.e. $r\in(\frac{\rho}{2},\rho)$. 
For the mean integral of the right hand side over
$r\in(\frac\rho2,\rho)$, we have
\begin{align*}
  \mint_{\frac\rho2}^\rho \int_{t_1}^{t_2}\int_{\partial
  B_{r}}|F|\d\mathcal{H}^{n-1}(x)\, \d t\d r
  &\le 
  \frac{2}{\rho}\iint_{Q_\rho}|F|\, \dx\dt.
\end{align*}
Therefore, we can choose $\hat\rho\in(\frac\rho2,\rho)$ such that
\begin{align*}
  \bigg|\int_{B_{\hat\rho}} \big[u(t_2)-u(t_1)\big]\, \d x\bigg| 
  &\le 
  \frac{2}{\rho}\iint_{Q_\rho}|F|\,\dx\dt.
\end{align*}
After diving both sides by $|B_{\hat\rho}|$ and using $|B_{\hat\rho}|\ge
2^{-n}|B_{\rho}|=2^{-n-1}\rho^{-p}|Q_\rho|$, we arrive at the claim.
\end{proof}

\begin{corollary}\label{cor:poincare}
  Let $u\in L^p(\Lambda_\rho(t_o);W^{1,p}_0(B_\rho(x_o)))$ be a
  function with $\partial_tu=\div F$ in the distributional sense in
  $Q_\rho(z_o)$, for a vector field $F\in
  L^1(Q_\rho(z_o),\R^n)$. Then we have the Poincar\'e type
  inequality
  \begin{equation*}
    \biint_{Q_\rho(z_o)}|u-(u)_{z_o,\rho}|^p\, \dx\dt
    \le
    c\rho^p\biint_{Q_\rho(z_o)}|\nabla u|^p\, \dx\dt
    +
    c\bigg(\rho^{p-1}\biint_{Q_\rho(z_o)}|F|\, \dx\dt\bigg)^{p}
  \end{equation*}
  with a constant $c=c(n,p)$. 
\end{corollary}

\begin{proof}
  As in the previous proof, we omit the center of the cylinder in the
  notation. 
  With the radius $\hat\rho\in[\frac\rho2,\rho]$ provided by
  Lemma~\ref{lem:gluing}, we estimate
  \begin{align*}
    &\biint_{Q_\rho}|u-(u)_{\rho}|^p\, \dx\dt
    \le
    c(p)\biint_{Q_\rho}|u-(u)_{\hat\rho}|^p\, \dx\dt\\
    &\qquad\le
      c(p)\biint_{Q_\rho}|u-(u)_{\hat\rho}(t)|^p\, \dx\dt
      +
      c(p)\mint_{\Lambda_\rho}|(u)_{\hat\rho}(t)-(u)_{\hat\rho}|^p\, \dt\\
    &\qquad\le
      c(n,p)\biint_{Q_\rho}|u-(u)_{\rho}(t)|^p\, \dx\dt
      +
      c(p)\mint_{\Lambda_\rho}\mint_{\Lambda_\rho}|(u)_{\hat\rho}(t)-(u)_{\hat\rho}(\tau)|^p\, \dt\d \tau.
  \end{align*}
  We estimate the second last integral by slicewise applications of
  Poincar\'e's inequality and the last one by
  Lemma~\ref{lem:gluing}. This yields the estimate
  \begin{align*}
    &\biint_{Q_\rho}|u-(u)_{\rho}|^p\, \dx\dt
      \le
      c\rho^p\biint_{Q_\rho}|\nabla u|^p\, \dx\dt
      +
      c\bigg(\rho^{p-1}\biint_{Q_\rho}|F|\, \dx\dt\bigg)^p,
  \end{align*}
  with a constant $c$ depending on $n$ and $p$. 
\end{proof}

Next, we establish a connection between the capacity and parabolic
Hausdorff measures, as in Theorem~\ref{thm:hausdorff-capvar}~\eqref{eq:hausdorff-capvar1}. We follow the strategy from \cite[Section 4.7, Theorem 4]{EG}.
\begin{proposition}\label{prop:cap-hausdorff}
  Let $E\subset\Omega_T$ be a set with
  $\ca_{\mathrm{var}}(E,\Omega_T)=0$. Then we have $\mathcal{P}^s(E)=0$ for every
  $s>n$. 
\end{proposition}

\begin{proof}
  In view of Lemma~\ref{lem:cap*-cap'-eq} and Lemma~\ref{lem:cap-cap'-equiv} we have
  $\ca^\ast_{\mathrm{var}}(E,\Omega_T)=0$. Therefore, we can find
  nonnegative functions $v_i\in\mathcal{W}(\Omega_T)$ with $v_i\ge1$
  a.e. on a neighborhood $U_i$ of $E$ and
  \begin{equation}\label{choice-vi}
    \iint_{\Omega_T}|\nabla v_i|^p\, \dx\dt+\|\partial_tv_i\|_{\mathcal{V}'(\Omega_T)}^{p'}+\sup_{0<t<T}\int_\Omega
    v_i^2\, \dx\le \frac1{4^{qi}}
  \end{equation}
  for each $i\in\N$, where $q = \max\{p,p'\}$. We let
  \begin{equation*}
    u:=\sum_{i\in\N}v_i.
  \end{equation*}
  The triangle inequalities for the norms in $L^p(\Omega_T)$,
  $\mathcal{V}'(\Omega_T)$ and $L^\infty(0,T;L^2(\Omega))$ together
  with the choice of $v_i$ according to~\eqref{choice-vi} imply
  $u\in\mathcal{W}(\Omega_T)$. By definition, for every $m\in\N$ we
  have $u\ge m$ on the neighborhood $\cap_{i=1}^mU_i$ of $E$.
  Therefore, we have 
  \begin{equation}\label{means-blow-up}
    \lim_{r\downarrow0}\,(u)_{z_o,r}=\infty
    \qquad\mbox{for every $z_o\in E$}. 
  \end{equation}
  Because of a well-known characterization of $W^{-1,p'}(\Omega)$, we
  can choose $F\in L^{p'}(\Omega_T,\R^n)$ with $\partial_tu=\div F$ and
  $\|F\|_{L^{p'}(\Omega_T)}\le
  c(n)\|\partial_tu\|_{\mathcal{V}'(\Omega_T)}$, see e.g. \cite[Prop.~9.20]{Brezis}.
  We claim that
  \begin{equation}
    \label{claim:density}
    \limsup_{r\downarrow0} \frac1{r^s}\iint_{Q_r(z_o)}\big[|\nabla u|^p+|F|^{p'}\big]\, \dx\dt=\infty
  \end{equation}
  for every $z_o\in E$. 
  Indeed, otherwise we would have
  \begin{equation}
    \label{not-claim}
    \frac1{r^s}\iint_{Q_r(z_o)}\big[|\nabla u|^p+|F|^{p'}\big]\dx\dt\le M
  \end{equation}
  for some $z_o\in E$, all $r\in(0,1]$ and a constant $M>0$. Then, the Poincar\'e
  inequality from Corollary~\ref{cor:poincare} and H\"older's inequality yield the estimate
  \begin{align*}
    \biint_{Q_r(z_o)}|u-(u)_{z_o,r}|^p\, \dx\dt
    &\le
    cr^p\biint_{Q_r(z_o)}|\nabla u|^p \, \dx\dt
    +
    c\bigg(r^{p}\biint_{Q_r(z_o)}|F|^{p'}\, \dx\dt\bigg)^{p-1}\\
    &\le
      cr^{s-n}M+c(r^{s-n}M)^{p-1}
    \le c_Mr^\theta  
  \end{align*}
  with a constant $c_M=c_M(n,p,M)$ and
  $\theta:=(s-n)\min\{1,p-1\}>0$. Therefore, we have
  \begin{align*}
    \big|(u)_{z_o,r/2}-(u)_{z_o,r}\big|
    &\le
    \biint_{Q_{r/2}(z_o)}|u-(u)_{z_o,r}|\,\dx\dt\\
    &\le
    \bigg(2^{n+p}\biint_{Q_{r}(z_o)}|u-(u)_{z_o,r}|^p\,\dx\dt\bigg)^{\frac1p}
    \le
    c_Mr^{\theta/p}  
  \end{align*}
  for every $r\in(0,1]$, and consequently,
  \begin{align*}
    \big|(u)_{z_o,2^{-k}}-(u)_{z_o,2^{-\ell}}\big|
    \le
    \sum_{j=\ell}^{k-1}\big|(u)_{z_o,2^{-j-1}}-(u)_{z_o,2^{-j}}\big|
    \le
    c_M\sum_{j=\ell}^{k-1}\bigg(\frac{1}{2^{\theta/p}}\bigg)^j
  \end{align*}
  for $k>\ell$ in $\N$. Since the right hand side vanishes in the
  limit $k,\ell\to\infty$, this
  implies that $((u)_{z_o,2^{-k}})_{k\in\N}$ is a Cauchy sequence, which
  contradicts~\eqref{means-blow-up}. Therefore, we have
  established~\eqref{claim:density}, which implies in particular
  \begin{equation*}
    E\subset \bigg\{z_o\in\Omega_T\colon \limsup_{r\downarrow0}\frac1{r^s}\iint_{Q_r(z_o)}|\nabla u|^p+|F|^{p'}\, \dx\dt>\lambda\bigg\}
  \end{equation*}
  for every $\lambda>0$. By a standard result on the densities of
  measures, this implies
  \begin{equation*}
    \mathcal{P}^s(E)\le \frac{1}{\lambda}\iint_{\Omega_T}|\nabla u|^p+|F|^{p'}\, \dx\dt,
  \end{equation*}
  see e.g. \cite[2.10.19(3)]{Federer}. 
  Since $|Du|^p+|F|^{p'}\in L^1(\Omega_T)$ and $\lambda>0$ was
  arbitrary, this implies the claim $\mathcal{P}^s(E)=0$.
\end{proof}

Then we prove the direction~\eqref{eq:hausdoff-capvar2} in Theorem~\ref{thm:hausdorff-capvar}. Recall that we denote the lateral boundary of $\Omega_T$ by $S_T:= \partial \Omega \times [0,T)$.

\begin{proposition} \label{prop:cap-leq-Pn}
  Let $E \subset \Omega_T$ be a set.
  Then $\ca_{\mathrm{var}}(E,\Omega_T) \leq c(n,p,\alpha) \mathcal{P}^n(E)$.
\end{proposition}

\begin{proof}
First, we prove the statement for a set $E\Subset\Omega_T$.
Let $\eps > 0$ and fix $0< \delta < \frac12
\mathrm{dist}(E,S_T)$. Then there exists a covering $\bigcup_{i=1}^\infty A_i \supset E$ with $d_p(A_i) < \delta$ and 
$$
\sum_{i=1}^\infty d_p(A_i)^n < \mathcal{P}^n_\delta(E) + \eps.
$$
There exists a parabolic cylinder $Q_i \supset A_i$ with radius $d_p(A_i)$ for every $i \in \N$ such that $\bigcup_{i=1}^\infty Q_i \subset \Omega_T$. Then by Lemma~\ref{lem:cap-union-cylinder-upperbound} we have
\begin{align*}
\ca_{\mathrm{var}} (E, \Omega_T) &\leq \ca_{\mathrm{var}} \left( \bigcup_{i=1}^\infty A_i, \Omega_T\right) \leq c \ca_{\mathrm{var}} \left( \bigcup_{i=1}^\infty Q_i, \Omega_T\right) \\
&\leq c\sum_{i=1}^\infty d_p(A_i)^n \leq c \left( \mathcal{P}^n_\delta(E) + \eps\right) \leq c \left( \mathcal{P}^n(E) + \eps\right).
\end{align*}
By letting $\eps \to 0$ the claim follows
in case $E\Subset\Omega_T$.

If $E\subset\Omega_T$ is an arbitrary subset, we consider an
increasing sequence of compact sets $K_i\Subset\Omega_T$, $i\in\N$,
with $\bigcup_{i\in\N}K_i=\Omega_T$. Using
Lemma~\ref{lem:capvar-upwards-conv} and then the already proven
inequality for $E\cap K_i\Subset\Omega_T$, we obtain
\begin{align*}
  \ca_{\mathrm{var}} (E, \Omega_T)
  \le
  c\lim_{i\to\infty}\ca_{\mathrm{var}} (E\cap K_i, \Omega_T)
  \le
  c\lim_{i\to\infty}\mathcal{P}^n(E\cap K_i)
  \le
  c\mathcal{P}^n(E).
\end{align*}
This is the asserted inequality for arbitrary sets.
\end{proof}

\section{Polar sets of supercaloric functions} \label{sec:polarsets}

Observe that for a supercaloric function $u$ in $\Omega_T$ the gradient defined by
\begin{equation} \label{eq:gradient}
Du = \lim_{k \to \infty} \nabla (\min\{u,k\})
\end{equation}
is a well defined measurable function by Lemma~\ref{lem:supercal-prop}. If $|Dv| \in L^1_{\loc}(\Omega_T)$ then it is the standard gradient in the weak sense.

First we recall a result from~\cite{GKM,KuLiPa} in case $\tfrac{2n}{n+1}<p<\infty$. The result concerns certain integrability properties for the supercaloric function as well as its gradient, which are enough to guarantee that the Riesz measure also exists. 

\begin{lemma} \label{lem:barenblatt-hi}
Let $\tfrac{2n}{n+1}< p < \infty$ and $s_o = \max \big\{ p-2, \tfrac{n}{p}(2-p) \big\}$. Suppose that $u$ is a supercaloric function in $\Omega_T$. If $u \in L^{s_o}_{\loc}(\Omega_T)$, then $u \in L^s_{\loc}(\Omega_T)$ for every $0<s<p-1+\tfrac{p}{n}$ and $|Du| \in L^{q}_{\loc}(\Omega_T)$ for every $0<q<p-1+\tfrac{1}{n+1}$.
\end{lemma}

Then we show that in case $1<p \leq \tfrac{2n}{n+1}$, any locally integrable supercaloric function has a gradient in the sense of~\eqref{eq:gradient}, which is locally integrable to power $p-1$. Again, this guarantees that the Riesz measure exists. 

\begin{lemma} \label{lem:supercal-L1}
Let $1<p \leq \tfrac{2n}{n+1}$ and $u$ be a supercaloric function in $\Omega_T$. If $u \in L^1_{\loc}(\Omega_T)$, then $Du$ defined in the sense of~\eqref{eq:gradient} satisfies $|Du| \in L^{p-1}_{\loc}(\Omega_T)$.
\end{lemma}

\begin{proof}

Let $Q \Subset \Omega$ and denote $u_k = \min \{u,k\}$ for $k \in
\N$.
Since supercaloric functions are locally bounded from
  below by definition, we can assume without loss of generality that
$u \geq 1$ in $Q$. By Lemma~\ref{lem:supercal-prop} we observe that $u_k$ is a weak supersolution for every $k \in \N$. As in~\cite[Theorem 3.5]{GKM}, by Hölder's inequality we obtain 
\begin{align*}
\iint_{Q} |\nabla u_k|^q \, \d x\d t \leq \left( \tfrac{p}{p-1-\eps} \right)^q \left( \iint_{Q} \left|\nabla \left( u_k^\frac{p-1-\eps}{p} \right)\right|^p \, \d x\d t \right)^\frac{q}{p} \left( \iint_{Q} u^{q \frac{1+\eps}{p-q}} \, \d x\d t \right)^{1-\frac{q}{p}}
\end{align*}
for all $\eps \in (0,p-1)$. Observe that by Caccioppoli inequality for
weak supersolutions,~\eqref{eq:caccioppoli-sup} in
Lemma~\ref{lem:caccioppoli}, and $u \in L^1_{\loc}(\Omega_T)$ we have
that the first integral is finite for all $\eps \in (0,p-1)$ uniformly
in $k\in \N$. The second integral is finite whenever $q
\tfrac{1+\eps}{p-q} \leq 1$, which implies that it is finite for any
$q < \tfrac{p}{2}$ if we choose $\eps>0$ appropriately. Since $\tfrac{p}{2} > p-1$ when $p < 2$, the claim follows.

\end{proof}

Next we prove a result which states that the limit of any increasing sequence of nonnegative weak solutions is locally bounded, provided that the limit satisfies an appropriate integrability condition.

\begin{lemma} \label{lem:blowup}
Let $1<p<\infty$ and suppose that $(h_i)_{i \in \N}$ is a pointwise
increasing sequence of nonnegative continuous weak solutions in
$\Omega_T$, and denote $h := \lim_{k \to \infty}h_k$. If $h \in
L^s_{\loc}(\Omega_T)$ for some $s > \max\big \{ p-2, \tfrac{n}{p}(2-p)
\big\}$, then $h$ is a locally bounded weak solution in $\Omega_T$.
\end{lemma}

\begin{proof}

We fix $z_o \in \Omega_T$ and denote $Q_r^{\sigma}(z_o) = B_{\sigma r}(x_o) \times (t_o-\sigma r^p, t_o)$ for $\sigma > 0$ such that $Q_r(z_o) := Q_r^{1}(z_o) \Subset \Omega_T$. Furthermore, let $s_o = \max\big \{ p-2, \tfrac{n}{p}(2-p) \big\}$ and $s \in (s_o, s_o + 2]$. Then, by~\cite[Chapter V, Theorem 4.1 \& Theorem 5.1]{Di} and the fact that the sequence $(h_k)_{k\in \N}$ is pointwise increasing, we obtain that for every $\sigma \in (0,1)$
\begin{align*}
h_k(y,\tau) \leq \frac{c(n,p)}{(1-\sigma)^\frac{n+p}{s-s_o}} \max \left\{\left(\biint_{Q_r(z_o)} h^s \, \d x \d t\right)^\frac{1}{s-s_o} ,1\right\}
\end{align*}
holds for all $(y,\tau) \in Q_r^\sigma(z_o)$ and every $k \in \N$. Thus, by passing to the limit $k \to \infty$, local boundedness holds by a covering argument. By Caccioppoli's inequality for locally bounded weak solutions it is clear that $h$ is a weak solution in $\Omega_T$.

\end{proof}

The integrability condition in Lemma~\ref{lem:blowup} can be sharpened to the endpoint in case $p>2$ by~\cite[Proposition 4.3]{KL-dichotomy}. However, due to Lemma~\ref{lem:barenblatt-hi}, the assumptions in Lemma~\ref{lem:blowup} are sufficient for our purposes.

Then we encapsulate the assumptions we use in the proof of Theorem~\ref{thm:supercal-polarset}.

\begin{assumption} \label{as:barenblatt}
Denote $s_o = \max \big\{ p-2,\tfrac{n}{p}(2-p) \big\}$. Suppose 
\begin{itemize}

\item $u \in L^{s_o}_{\loc}(\Omega_T)$ if $\tfrac{2n}{n+1} < p < \infty$,

\item $u \in L^s_{\loc}(\Omega_T)$ for some $s > s_o$ if $1<p \leq \tfrac{2n}{n+1}$.

\end{itemize}

\end{assumption}

In case $p =2$ we have that $s_o = 0$ in
Assumption~\ref{as:barenblatt}. However, in this case every
supercaloric function is locally integrable, see
e.g.~\cite{watson-book}, such that the assumption is redundant.
For a more detailed analysis of polar sets in the context of the heat equation we refer to~\cite{watson}.

Let us denote the balayage of $u\chi_K$ by $\widehat R_K^u$. 

\begin{lemma} \label{lem:RuK-prop}
Let $1<p<\infty$ and $u$ be a nonnegative supercaloric function in $\Omega_T$ that satisfies Assumption~\ref{as:barenblatt}. Then, $\widehat R_K^u$ is a weak solution in $\Omega_T \setminus K$, $\widehat R_K^u (\cdot,t)$ has zero lateral boundary values in $W^{1,p}_0$-sense on $\partial \Omega$ for a.e. $t \in (0,T)$ and $\widehat R_K^u (\cdot,t) \equiv 0$ for all $t \in (0,\delta)$ for some $\delta > 0$. Furthermore, $\widehat{R}_K^u$ attains the zero boundary values continuously.
\end{lemma}

\begin{proof}

Let $Q \Subset \Omega_T\setminus K$ be a smooth enough space-time
cylinder and $\psi_i$ be a nonnegative Lipschitz function in $\Omega_T$ for every $i \in \N$ such that $\psi_i \nearrow \widehat{R}_K^{u}$ pointwise in $\Omega_T$ as $i \to \infty$. Let $h_i$ be a weak solution in $Q$ with boundary values $\psi_i$ on $\partial_p Q$. Observe that $h_i \in C(\overline{Q})$, $h_i \leq h_{i+1}$ and $h_i \leq \widehat{R}_K^{u}$ in $Q$ for every $i \in \N$.

By construction there holds $h:=\lim_{i \to \infty} h_i \leq
\widehat{R}_K^{u} \leq u$ in $Q$. Due to the integrability assumption
on $u$, Lemmas~\ref{lem:barenblatt-hi} and~\ref{lem:blowup} imply that
$h$ is a locally bounded weak solution in $Q$. The comparison
principle in~\cite[Theorem 2.4]{BBGP} guarantees that the function
\[ v = 
\begin{cases} 
      \widehat R_K^u & \text{ in } \Omega_T \setminus Q, \\
      h & \text{ in } Q,
   \end{cases}
\]
is supercaloric in $\Omega_T$, and by minimality of $\widehat{R}_K^u$ it follows that $\widehat{R}_K^{u} = h$ in $Q$, which implies that also $\widehat{R}_K^{u}$ is a locally bounded weak solution in $Q$. Since $Q \Subset \Omega_T\setminus K$ was arbitrary, the first claim follows. 

For the second claim, use a Poisson modification in a cylinder $Q' = D
\times (0,T)$ for which $D \Subset \R^n \setminus K$ and $\partial
\Omega \subset \partial D$. With a similar argument as in~\cite[Lemma~4.6]{MS-obstacle}, we can conclude that $\widehat{R}_K^{u}$ attains zero lateral boundary values in the Sobolev $W^{1,p}_0$-sense and continuously, since $\R^n \setminus \Omega$ is uniformly $p$-fat (see~\cite{GL1,GL2,GLL,kilpelainen-lindqvist}).

\end{proof}

We recall the following three results from~\cite{KKKP}.

\begin{lemma}[\!\!{\cite[Lemma 6.4]{KKKP}}] \label{lem:u1-ulambda-measure-estimates}
Suppose that $1<p<\infty$ and that $K \subset \Omega_\infty$ is a compact set, and denote by $u_1$ and $u_\lambda$ for $\lambda > 0$ the balayage of $\chi_K$ and $\lambda \chi_K$, respectively. Then
$$
\mu_{u_1} (\Omega_\infty) \leq c \left( \lambda^{1-p} + \lambda^{-\frac{1}{p-1}} \right) \mu_{u_\lambda}(\Omega_\infty)
$$
and 
$$
\mu_{u_\lambda} (\Omega_\infty) \leq c \left( \lambda^{p-1} + \lambda^{\frac{1}{p-1}} \right) \mu_{u_1}(\Omega_\infty),
$$
where $c = c(p) > 0$ and $\mu_{u_1}$ and $\mu_{u_\lambda}$ are the Riesz measures of $u_1$ and $u_\lambda$, respectively.
\end{lemma}

\begin{lemma}[\!\!{\cite[Lemma 6.3]{KKKP}}] \label{lem:Riesz-measure-order}
Suppose that $1<p<\infty$ and let $u$ and $v$ be weak supersolutions in $\Omega_\infty$ such that both are continuous weak solutions up to the boundary $\partial_p \Omega_\infty$ outside some compact subset of $\Omega_\infty$. If $u \geq v$ in $\Omega_\infty$ and $u = v  = 0$  on $\partial_p \Omega_\infty$, then 
$$
\mu_v(\Omega_\infty) \leq \mu_u (\Omega_\infty),
$$ 
where $\mu_v$ and $\mu_u$ are Riesz measures of $v$ and $u$, respectively.
\end{lemma}

\begin{lemma} \label{lem:truncation-Riesz-measure}
Suppose that $1<p<\infty$ and $\lambda > 0$. Let $v$ be a supercaloric function in $\Omega_\infty$ such that $v \in L^1_{\loc}(\Omega_\infty)$ and $|D v| \in L^{p-1}_{\loc}(\Omega_\infty)$ hold true, and suppose that $v$ is a continuous weak solution in $\overline{\Omega_\infty} \setminus K$ for some compact set $K\subset \Omega_\infty$ with $v < \lambda$ in $\Omega_\infty \setminus K$. Then
$$
\mu_v(\Omega_\infty) = \mu_{\min_{\{v,\lambda\}}} (\Omega_\infty).
$$
\end{lemma}

\begin{proof}

Suppose that $\{v > \lambda\} \neq \varnothing$ since otherwise the claim is clear. Observe that by the assumptions $\overline{\{v > \lambda\}}$ is a compact subset of $\Omega_\infty$ and both $v$ and $\min\{v,\lambda\}$ are weak solutions outside some compact set of $\Omega_\infty$. Also clearly $v$ satisfies~\eqref{eq:weak-super} by Lemma~\ref{lem:supercal-prop} with the gradient $Dv$ defined in the sense of~\eqref{eq:gradient}. Let $\varphi \in C^\infty_0(\Omega_\infty,[0,1])$ such that $\varphi =1 $ in $\overline{\{v > \lambda\}}$. It follows that
\begin{align*}
\int_{\Omega_\infty} \varphi \, \d \mu_v - \int_{\Omega_\infty} \varphi \, \d \mu_{\min\{v,\lambda\}} &= \iint_{\{v > \lambda\}} |D v|^{p-2} D v \cdot D \varphi - (v-\lambda) \partial_t \varphi \, \d x \d t \\
&= 0
\end{align*}
by the assumptions on $\varphi$. Since the support of $\varphi$ was arbitrary and by the assumptions the claim follows.
\end{proof}

\begin{lemma} \label{lem:supercal-cap-levelset}
Let $1 < p < \infty$ and $u$ be a nonnegative supercaloric function in $\Omega_T$ that satisfies Assumption~\ref{as:barenblatt}. Furthermore, suppose that $\lambda > 1$. Then, there exists a constant $c = c(n,p,\alpha)> 0$ such that
$$
\ca_{\mathrm{var}}(\{u > \lambda \} \cap \mathrm{int}( K), \Omega_T) \leq c \mu_{\widehat R_K^u}(\Omega_T) \left( \lambda^{1-p} + \lambda^{- \frac{1}{p-1}} \right) 
$$
for every compact set $K \subset \Omega_T$.
\end{lemma}

\begin{proof}

By Lemma~\ref{lem:RuK-prop} it follows that $\widehat{R}_K^u$ is a
weak solution in $\Omega_T \setminus K$. Thus, $\widehat{R}_K^u$ can
be extended as a weak solution to $\Omega \times [T,\infty)$ by using
$\widehat{R}_K^u$ as initial values at time $t = T$. We denote the
extended function in $\Omega_\infty$ still by $\widehat{R}_K^u$. Since $\widehat{R}_K^u$ is a bounded weak solution in $\Omega \times
[T,\infty)$ with vanishing boundary values, it tends to zero uniformly
in the limit $t\to\infty$ or even becomes extinct in finite time, see
\cite[Chapter~V, Thm.~4.3]{Di} and \cite[Chapter~VII, Prop.2.1]{Di}, respectively.
Therefore, there exists a compact set $K' \subset \Omega_\infty$ such that
$\widehat{R}_K^{u}$ is a continuous weak solution and $\widehat{R}_K^{u} < \lambda$ in $\Omega_\infty \setminus K'$.
Let
$\psi_i$ be an increasing sequence of continuous functions such that
$\psi_i = \widehat{R}_K^{u}$ in $\Omega_\infty \setminus K'$ and
$\psi_i \xrightarrow{i \to \infty} \widehat{R}_K^{u}$
pointwise. Denote the solution to the obstacle problem with obstacle
$\psi_i$ from Proposition~\ref{prop:obstacle} by $u_i$. Observe that $u_i$ is a weak solution in $\Omega_\infty \setminus K'$. This can be deduced as follows. Since $\psi_i \in C(\overline{\Omega_\infty})$, also $u_i \in C(\overline{\Omega_\infty})$ by Proposition~\ref{prop:obstacle}. Furthermore, $u_i$ agrees with its Poisson modification in any cylinder in $\Omega_\infty \setminus K'$, since $\widehat{R}_K^{u}$ is a weak solution in such a cylinder, and by using comparison principle and minimality of $u_i$. This further implies that the Riesz measure of $u_i$ has support in $K'$ for every $i \in \N$.

Thus, since the sequence $(u_i)_{i \in \N}$ is increasing, it follows that $\{u_i \geq \lambda\} \subset \{u_{i+1} \geq \lambda\}$ for every $i \in \N$, and furthermore, the sets are compact by continuity of $u_i$. Observe that
$$
\{u > \lambda\} \cap \mathrm{int}(K) \subset \{ \widehat{R}_K^{u} > \lambda \} \cap K \subset \bigcup_{i=1}^\infty  ( \{u_i \geq \lambda \} \cap K).
$$
Thus
\begin{align*}
\ca_{\mathrm{var}} (\{u > \lambda\} \cap \mathrm{int}(K), \Omega_T) &\leq \ca_{\mathrm{var}} \left( \bigcup_{i=1}^\infty ( \{u_i \geq \lambda \} \cap K), \Omega_T \right)\\
&\leq c \lim_{i\to \infty} \ca_{\mathrm{var}} (\{u_i \geq \lambda\} \cap K, \Omega_T) \\
&\leq c \lim_{i\to \infty} \ca_{\mathrm{var}} (\{u_i \geq \lambda\}, \Omega_\infty) \\
&\leq c \lim_{i\to \infty} \ca (\{u_i \geq \lambda\}, \Omega_\infty) 
\end{align*}
by Lemmas~\ref{lem:capvar-prop},~\ref{lem:capvar-upwards-conv} and
Theorem~\ref{thm:capvar-eq-cap}. Denote by $u_{1,i}$ and
$u_{\lambda,i}$ the balayage of $\chi_{\{u_i \geq \lambda\}}$ and
$\lambda\chi_{\{u_i \geq \lambda\}}$, respectively. By
Lemmas~\ref{lem:cap-compact-representation} and \ref{lem:u1-ulambda-measure-estimates} we obtain
$$
\ca (\{u_i \geq \lambda\}, \Omega_\infty) = \mu_{u_{1,i}} (\Omega_\infty) \leq c \left( \lambda^{1-p} + \lambda^{-\frac{1}{p-1}} \right) \mu_{u_{\lambda,i}}(\Omega_\infty).
$$
By observing that $u_{\lambda,i} \leq u_i$ and using Lemma~\ref{lem:Riesz-measure-order} we obtain
$$
\mu_{u_{\lambda,i}}(\Omega_\infty) \leq \mu_{\min \{u_{i},\lambda\}}(\Omega_\infty) = \mu_{\min \{u_{i},\lambda\}}(K').
$$
Since $\min\{u_i,\lambda\} \xrightarrow{i \to \infty} \min\{\widehat R_K^u,\lambda\}$ pointwise in $\Omega_\infty$, by Lemma~\ref{lem:supercal-prop} and weak convergence of the Riesz measures in Lemma~\ref{lem:bdd-weak-sol-convergence} together with Lemma~\ref{lem:truncation-Riesz-measure} we have
$$
\limsup_{i \to \infty} \mu_{\min\{u_i,\lambda\}}(K') \leq \mu_{\min\{\widehat R_K^u,\lambda\}}(K') = \mu_{\widehat R_K^u}(K'),
$$
since $\widehat R_K^u$ is a weak solution with $\widehat R_K^u < \lambda$ in $\Omega_\infty \setminus K'$. By combining all the estimates we arrive at 
$$
\ca_{\mathrm{var}} (\{u > \lambda\} \cap \mathrm{int}(K), \Omega_T) \leq c  \left( \lambda^{1-p} + \lambda^{-\frac{1}{p-1}} \right) \mu_{\widehat R_K^u}(\Omega_T),
$$
which concludes the result.

\end{proof}

If a supercaloric function $u$ takes also negative values, we use the
following construction. The main difficulty arises if $u$ is not bounded from below in a neighborhood of the lateral boundary of $\Omega_T$.

For a compact set $K \subset \Omega_T$ there exists a set
$\Omega'_{t_1,t_2} := \Omega' \times (t_1,t_2) \Subset \Omega_T$ with
$K \subset \Omega'_{t_1,t_2}$, where $\Omega'$ can be chosen in the
form $\Omega' = \{x\in \Omega : \mathrm{dist}
(x,\R^n \setminus \Omega) > \delta\}$ for small enough $\delta > 0$,
such that $\R^n \setminus \Omega'$ is uniformly $p$-fat by
Lemma~\ref{lem:parallel-fatness}, independently of $\delta>0$. Observe that since $u$ is locally bounded from below in $\Omega_T$, we have that $\tilde u := u + \big| \inf_{\Omega'_{t_1,t_2}} u \big|$ is a nonnegative supercaloric function in $\Omega'_{t_1,t_2}$, and 
\begin{align} \label{eq:signed-supercal-to-nonneg}
\begin{aligned}
v:=
\left\{
\begin{array}{c}
\tilde u \quad \text{ in } \Omega' \times (t_1,t_2), \\
0 \quad \text{ in } \Omega' \times (0,t_1],
\end{array}
\right.
\end{aligned}
\end{align}
is as well in $\Omega'_{0,t_2}$. Clearly, $v$ satisfies Assumption~\ref{as:barenblatt} in $\Omega'_{0,t_2}$ whenever $u$ does in $\Omega_T$. Thus, in that case, Lemma~\ref{lem:RuK-prop} guarantees that the balayage of $v \chi_K$ exists in $\Omega'_{0,t_2}$ with the desired properties. Furthermore, since the balayage is a continuous weak solution in $\Omega'_{0,t_2} \setminus K$ with zero boundary values on the parabolic boundary of $\Omega'_{0,t_2}$, we can extend it as a continuous weak solution to $\Omega' \times [t_2, \infty)$. With this construction, we obtain the following variant of Lemma~\ref{lem:supercal-cap-levelset}.

\begin{corollary} \label{cor:signed-supercal-cap-levelset}
Let $1 < p < \infty$ and $u$ be a supercaloric function in $\Omega_T$ that satisfies Assumption~\ref{as:barenblatt}. Furthermore, suppose that $\lambda > 1$, and let $K$ be a compact set in $\Omega_T$. Then, there exists an open set $\Omega' \Subset \Omega$ with $K \subset \Omega'$ and a constant $c = c(n,p,\alpha)> 0$ such that
$$
\ca_{\mathrm{var}}(\{u > \lambda \} \cap \mathrm{int}(K), \Omega_T) \leq c \mu_{\widehat R_K^{v}}(\Omega'_T) \left( \lambda^{1-p} + \lambda^{- \frac{1}{p-1}} \right),
$$
where $v$ is defined as in~\eqref{eq:signed-supercal-to-nonneg}, and $\widehat R_K^{v}$ is a balayage of $v \chi_K$ in $\Omega'_T$.
\end{corollary}
\begin{proof}

By Lemmas~\ref{lem:parallel-fatness} and~\ref{lem:supercal-cap-levelset}, and the construction above, we obtain 
$$
\ca_{\mathrm{var}}(\{v > \lambda \} \cap \mathrm{int}(K), \Omega'_{0,t_2}) \leq c \mu_{\widehat R_K^v}(\Omega'_T) \left( \lambda^{1-p} + \lambda^{- \frac{1}{p-1}} \right).
$$
Furthermore, since $v = u + \big| \inf_{\Omega'_{t_1,t_2}} u \big|$ in $K$, we have that $\{ u > \lambda \} \cap \mathrm{int}(K) \subset \{v  >\lambda\} \cap \mathrm{int}(K)$, and by monotonicity of the capacity we have 
\begin{align*}
\ca_{\mathrm{var}}(\{v > \lambda \} \cap \mathrm{int}(K), \Omega'_{0,t_2}) &\geq \ca_{\mathrm{var}}(\{u > \lambda \} \cap \mathrm{int}(K), \Omega'_{0,t_2}) \\
&\geq c \ca_{\mathrm{var}}(\{u > \lambda \} \cap \mathrm{int}(K), \Omega'_T) \\
&\geq c \ca_{\mathrm{var}}(\{u > \lambda \} \cap \mathrm{int}(K), \Omega_T),
\end{align*}
for a constant $c = c(n,p,\alpha) > 0$ by using also Lemmas~\ref{lem:capvar-Tinfty2} and~\ref{lem:capvar-omega-omega'}. Thus the claim follows.
\end{proof}

Now we are in a position to provide the\\[-1ex]

\noindent\textit{Proof of Theorem~\ref{thm:supercal-polarset}. }
Let $K_i \Subset \Omega_T$ be compact sets such that $\mathrm{int}(K_i) \subset \mathrm{int}(K_{i+1})$ for every $i \in \N$ and $\bigcup_{i=1}^\infty \mathrm{int}(K_i) = \Omega_T$. Observe that by Corollary~\ref{cor:signed-supercal-cap-levelset} we have
\begin{align} \label{eq:cap-polar-compact}
\ca_{\mathrm{var}} (\{u=\infty\} \cap \mathrm{int}(K_i),\Omega_T) &\leq \lim_{\lambda\to \infty} \ca_{\mathrm{var}} (\{u>\lambda\} \cap \mathrm{int}(K_i) , \Omega_T) = 0
\end{align}
for every $i \in \N$. Clearly, we have 
$$
\ca_{\mathrm{var}} (\{u=\infty\},\Omega_T) = \ca_{\mathrm{var}} \left( \bigcup_{i=1}^\infty \left( \{u=\infty\} \cap \mathrm{int}(K_i) \right),\Omega_T\right).
$$
By using Lemma~\ref{lem:capvar-upwards-conv} together with~\eqref{eq:cap-polar-compact} the claim follows.

\hfill \qed

The following lemma shows that the integrability condition in Theorem~\ref{thm:supercal-polarset} is sharp in the case $p > 2$.

\begin{lemma} \label{lem:monster-polarset}
Let $p > 2$. Then there exists a supercaloric function $u$ in $\Omega_T$ such that $\ca_{\mathrm{var}} (\{u=\infty\}, \Omega_T) > 0$, and $u \notin L^{p-2}_{\loc}(\Omega_T)$ with $u \in L^q_{\loc}(\Omega_T)$ for every $0<q<p-2$.
\end{lemma}

\begin{proof}

Let $\tau \in (0,T)$ and recall the following example of supercaloric function from~\cite[Section 3]{KuLiPa} defined by
\[u(x,t) = \begin{cases} 
      \mathfrak{D}(x,t) & (x,t) \in \Omega \times (0,\tau), \\
      (t-\tau)^{-\frac{1}{p-2}} \mathfrak{U}(x) & (x,t) \in \Omega \times [ \tau ,T),
   \end{cases}
\]
in which 
$$
\mathfrak{D}(x,t) = \left[ A \left( \frac{\tau}{\tau - t} \right)^\frac{n(p-2)}{\lambda(p-1)} + \left( \tfrac{p-2}{p} \right) \lambda^{-\frac{1}{p-1}} \left( \frac{|x|^p}{\tau-t} \right)^\frac{1}{p-1} \right]^\frac{p-1}{p-2},
$$
with $A > 0$ and $\lambda = n(p-2) +p$, is a weak solution in $\Omega
\times (0,\tau)$ (see also~\cite[Chapter~XI, Remark~7.1]{Di}) and $\mathfrak{U} \in C(\overline{\Omega}) \cap W^{1,p}(\Omega)$ is a weak solution to
$$
\Div (|\nabla \mathfrak{U}|^{p-2} \nabla \mathfrak{U}) + \tfrac{1}{p-2} \mathfrak{U} = 0 \quad \text{ in } \Omega
$$
with $\mathfrak{U} > 0$ in $\Omega$. Observe that $u$ is a supercaloric function in $\Omega_T$ with $\{u = \infty\} = \Omega \times \{\tau\}$. Lemma~\ref{lem:cap-slice-null} concludes the claim.

\end{proof}

Next, we show that every compact set of vanishing capacity is
contained in a polar set. 

\begin{lemma}
Let $1<p<\infty$ and $K \subset \Omega_T$ be a compact set with $\ca_{\mathrm{var}} (K, \Omega_T) = 0$. Then there exists a supercaloric function $u$ in $\Omega_T$ such that $K \subset \{(x,t) \in \Omega_T : u(x,t) = \infty\}$.
\end{lemma}

\begin{remark}
The result holds for any set of the form $\bigcup_{i=1}^\infty K_i$, in which $K_i$ is compact and $K_i \subset K_{i+1} \subset \Omega_T$ for each $i \in \N$.
\end{remark}

\begin{proof}
Since $\ca_{\mathrm{var}} (K, \Omega_T) = 0$, for every $i \in \N$ there exists $v_i \in C_0^\infty(\Omega \times \R)$ such that $v_i \geq \chi_K$ with 
$$
\|v_i\|_{\mathcal{W}(\Omega_T)} \leq 4^{-is},
$$
where $s = \max \{p,p'\}$.
By setting 
$$
\varphi_k = \sum_{i=1}^k v_i,
$$
for $k\in\N$, we have that
\begin{align*}
\|\varphi_k\|_{\mathcal{W}(\Omega_T)} \leq \sum_{i=1}^k 2^{is}  \|v_i\|_{\mathcal{W}(\Omega_T)} \leq \frac{1-2^{-sk}}{2^s -1} \leq \frac{1}{2^s - 1} =: c_1(p).
\end{align*}
By taking an open set $U \Subset \Omega$ and $0<t_1<t_2 <T$ such that $K \subset U \times (t_1,t_2)$ and $\zeta \in C_0^\infty(U, [0,1])$ with $\zeta =1$ in $K$, we have 
$$
\|\zeta \varphi_k\|_{\mathcal{W}(\Omega_T)} \leq c(p,\mathrm{diam} (\Omega),d) \|\varphi_k\|_{\mathcal{W}(\Omega_T)} \leq c_2(p,\mathrm{diam} (\Omega),d),
$$
where $d = \mathrm{dist}(\partial U \times (t_1,t_2), K)$. Denote $\psi_k = \zeta \varphi_k$. By Lemma~\ref{lem:en-leq-W} there exists a supercaloric function $u_k \in \mathcal{V}(\Omega_T)\cap C(\overline{\Omega_T})$ with $u_k \geq \psi_k$ such that 
$$
\|u_k\|_{\mathrm{en}, \Omega_T} \leq c(p) \|\psi_k\|_{\mathcal{W}(\Omega_T)} \leq c(p,\mathrm{diam} (\Omega),d).
$$
More precisely, the function $u_k$ from Lemma~\ref{lem:en-leq-W} is
constructed as the smallest supercaloric function above
$\psi_k$. Therefore, we have that $u_k$ is an increasing sequence
since $\psi_k$ is. By the inequality above, we have that $u = \lim_{k \to \infty} u_k$ is finite a.e. in $\Omega_T$ and thus a supercaloric function as an increasing limit of such functions. Since $u \geq \lim_{k \to \infty} \psi_k$ in $\Omega_T$, it follows that $u = \infty$ in $K$.

\end{proof}

We continue with an auxiliary result that provides a connection between the capacity of an
open set and the Riesz measure of the balayage, cf.
Lemma~\ref{lem:cap-compact-representation}.

\begin{lemma} \label{lem:capvar-openset-potential}
Let $1<p<\infty$ and $U \Subset \Omega_T$ be an open set. Then 
$$
\ca_{\mathrm{var}} (U, \Omega_T) \approx \mu_{\widehat{R}_U} (\Omega_T)
$$
up to a constant depending only on $n,p$ and $\alpha$.
\end{lemma}

\begin{proof}

Take an expanding sequence of compact sets $K_i$ such that $U =
\bigcup_{i=1}^\infty K_i$. By Lemma~\ref{lem:capvar-upwards-conv},
Theorem~\ref{thm:capvar-eq-cap} and Lemma~\ref{lem:cap-compact-representation} we obtain 
$$
\ca_{\mathrm{var}} (U, \Omega_T) \approx \lim_{i \to \infty} \ca_{\mathrm{var}} (K_i, \Omega_T) \approx \lim_{i \to \infty}\ca (K_i, \Omega_\infty) = \lim_{i \to \infty} \mu_{u_i} (\Omega_T),
$$
where we denoted $u_i = \widehat{R}_{K_i}$. Observe that $(u_i)_{i\in
  \N}$ is an increasing sequence of uniformly bounded supercaloric
functions, and thus lower semicontinuous weak supersolutions by
Lemma~\ref{lem:supercal-prop}. Hence, the pointwise limit $u$ exists
as a bounded supercaloric function, which again is a weak
supersolution by Lemma~\ref{lem:supercal-prop}. On one hand, $u_i \leq
\widehat{R}_U$ for each $i \in \N$, and on the other hand $u \geq
\chi_U$, which implies $u \geq \widehat{R}_U$ by minimality of
$\widehat{R}_U$. This implies that $u = \widehat{R}_U$.

Furthermore, we have
$$
\lim_{i \to \infty} \int_{\Omega_T} \varphi \, \d \mu_{u_i} = \int_{\Omega_T} \varphi \, \d \mu_{u}
$$
for every $\varphi \in C_0^\infty(\Omega_T)$ by Lemma~\ref{lem:bdd-weak-sol-convergence}. Thus, by choosing $\varphi = 1$ in $\overline{U}$ and using outer regularity of $\mu_u$ it follows that 
$$
\mu_u (\overline{U}) \approx \ca_{\mathrm{var}} (U, \Omega_T),
$$
which concludes the proof.
\end{proof}

Then we show that for a compactly contained set in $\Omega_T$, the
balayage vanishes if and only if the set has capacity zero. The proof
follows the strategy presented in~\cite[Theorem
4.1]{AS}.

\begin{lemma} \label{lem:balayage-zero-iff-cap-zero}
Let $1<p<\infty$ and $E \Subset \Omega_T$ be a set. Then we have $\widehat R_E \equiv 0$ if and only if $\ca_{\mathrm{var}}(E,\Omega_T) = 0$.
\end{lemma}

\begin{proof}
First we show the "only if" part. We extend $\widehat{R}_E$ by zero to $\Omega \times [T,\infty)$ and consider the reference set $\Omega_\infty$. By Choquet's topological lemma we have that there exists a decreasing sequence of supercaloric functions $v_j\geq \chi_E$ such that the pointwise limit $v = \lim_{j\to \infty} v_j$ satisfies $\widehat v = \widehat{R}_E$. Observe that $\widehat{R}_E = R_E$ a.e. by Lemma~\ref{lem:balayage-psi-prop} and $\widehat v = (v)_*$ everywhere in $\Omega_T$ since $v$ is a limit of uniformly bounded supercaloric functions, see~\cite[Lemma 2.7]{LP}. Since $v$ is a weak supersolution, also $v = (v)_* = \widehat{v}$ a.e. in $\Omega_\infty$ by Lemma~\ref{lem:supersol-lsc} and $v = \widehat{R}_E = 0$ a.e. in $\Omega_\infty$. 

Let $K \subset \Omega_T$ be a compact set such that
$\mathrm{int}(K)\supset \overline{E}$ and define $u_j = \min \{
\widehat{R}_K, v_j \}$. Observe that $v_j \geq u_j \geq \chi_E$. Since
also $u_j \geq R_E$ a.e. in $\Omega_\infty$, it follows that the
pointwise limit $\lim_{j\to \infty} u_j = u \geq R_E$ a.e. in
$\Omega_\infty$. Furthermore, $u$ is a weak supersolution as pointwise
limit of uniformly bounded weak supersolutions and thus $(u)_* = u
\geq R_E = \widehat{R}_E$ a.e. in $\Omega_\infty$. Since $u \leq v$ it
follows that $u = 0$ a.e. in $\Omega_\infty$. Furthermore, $\mu_{u_j}
\wto \mu_u$ weakly in $\Omega_\infty$ as $j \to \infty$ by Lemma~\ref{lem:bdd-weak-sol-convergence}.

Let $U_j = \left\{u_j > \tfrac12 \right\} \supset E$. Observe that
clearly $U_j \subset \{\widehat R_K > \tfrac12 \} \Subset
\Omega_\infty$. Exhaust $U_j$ by an expanding sequence of compact sets $\{K_j^i\}$. By Lemma~\ref{lem:capvar-upwards-conv} we have that for any $\eps > 0$ we can find a subsequence of such sets, denoted by $\{K_j'\}$ such that 
$$
\ca_{\mathrm{var}} (K_j',\Omega_\infty) \geq c(n,p,\alpha) \ca_{\mathrm{var}} (U_j,\Omega_\infty) - \eps.
$$
Observe that $\ca_{\mathrm{var}} (K_j',\Omega_\infty) \leq c \ca (K_j',\Omega_\infty) $ for $c= c(n,p,\alpha)>0$ by Theorem~\ref{thm:capvar-eq-cap}  and $\ca (K_j',\Omega_\infty) = \mu_{R_{K_j'}} (\Omega_\infty)$ by Lemma~\ref{lem:cap-compact-representation}. Denoting the r\'eduite of $\frac12 \chi_{K_j'}$ by $R_j$ we have 
\begin{align*}
c \ca_{\mathrm{var}} &(U_j,\Omega_\infty) - \eps\\
&\leq \ca_{\mathrm{var}} (K_j',\Omega_\infty)
\leq c \ca (K_j',\Omega_\infty) = c \mu_{R_{K_j'}} (\Omega_\infty) \leq c \mu_{R_j} (\Omega_\infty),
\end{align*}
where we used Lemma~\ref{lem:u1-ulambda-measure-estimates} as
well. Since $u_j \geq \widehat{R}_j$ for every $j \in \N$ by
minimality of $\widehat{R}_j$, we have
that $\widehat{R}_j \to 0$ a.e. in $\Omega_T$
and thus $\mu_{R_j} \wto 0$ weakly as $j\to\infty$. Since there is a
compact set in which the support of $\mu_{R_j}$ is contained for every $j\in \N$ we have
\begin{align*}
\ca_{\mathrm{var}}(E, \Omega_\infty) \leq \liminf_{j\to \infty} \ca_{\mathrm{var}} (U_j,\Omega_\infty) \leq c \left( \limsup_{j \to \infty} \mu_{R_j} (\Omega_\infty) + \eps \right) \leq 
c \eps.
\end{align*}
Since $\eps > 0$ was arbitrary, we conclude the proof by application of Lemma~\ref{lem:capvar-Tinfty2}.

Then we prove the "if" part.
By definition there exists a decreasing sequence of open sets $U_i$ satisfying $E \subset U_i \Subset \Omega_T$ such that 
$$
\lim_{i\to \infty} \ca_{\mathrm{var}} (U_i, \Omega_T) = 0,
$$
and Lemma~\ref{lem:capvar-openset-potential} implies that
$$
\mu_{\widehat{R}_{U_i}} (\Omega_T) \leq c(n,p,\alpha) \ca_{\mathrm{var}} (U_i, \Omega_T).
$$
The functions $\widehat{R}_{U_i}$ form a decreasing sequence of uniformly bounded weak supersolutions and they converge pointwise to a weak supersolution $u$. Since the Riesz measures converge weakly by Lemma~\ref{lem:bdd-weak-sol-convergence}, we have $\mu_u \equiv 0$ by the estimates above, and thus $\widehat{u} \equiv 0$. On the other hand, we have that $\widehat u \geq \widehat{R}_E$, which proves the claim.

\end{proof}

\section{Removability}
\label{sec:removability}

An important ingredient in the proof of the removability result for weak supersolutions, Lemma~\ref{lem:supersol-removability}, is the following approximation result. The proof is analogous to~\cite[Lemma 5.1]{AS}, which also applies ideas in~\cite{Petitta}. However, since our setting is slightly different, we include the proof for readers' convenience.

\begin{lemma} \label{lem:S-closure}
Let $1 < p < \infty$ and $E \subset \Omega_T$ be a relatively closed set in $\Omega_T$. If $\ca_{\mathrm{var}} (E,\Omega_T) = 0$, then
$$
\overline{C_0^\infty (O \setminus E)}^{\mathcal{S}(\Omega_T)} \simeq \overline{C_0^\infty(O)}^{\mathcal{S}(\Omega_{T})}\quad \text{ for all open } O \subset \Omega_T,
$$
where
$$
\mathcal{S}(\Omega_{T}) = \left\{ u \in \mathcal{V}(\Omega_T): \partial_t u \in \mathcal{V}'(\Omega_T) + L^1(\Omega_T) \right\}. 
$$
\end{lemma}

\begin{proof}

The direction $\overline{C_0^\infty (O \setminus
  E)}^{\mathcal{S}(\Omega_{T})} \hookrightarrow
\overline{C_0^\infty(O)}^{\mathcal{S}(\Omega_{T})}$ is clear. To prove
the reverse direction, let $\varphi \in C_0^\infty(O)$ and denote $S =
\spt (\varphi) \subset \Omega_T$ and $K = S \cap E$. Observe that $K$
is a compact subset of $\Omega_T$. Take a nested sequence of compact
sets $(K_i)_{i \in \N}$ in $\Omega_T$ (e.g., each a finite union of
compact space-time cylinders) such that $K\subset\mathrm{int}(K_i)$ for every
$i\in\N$ and $\cap_{i=1}^\infty K_i = K$. The fact $\ca_{\mathrm{var}} (K,\Omega_T) \leq \ca_{\mathrm{var}} (E,\Omega_T) = 0$ and Lemma~\ref{lem:capvar-limit-compact} imply that for every $i \in \N$ there exists $u_i \in C_0^\infty(\Omega \times \R)$ such that 
\begin{equation} \label{eq:approx-min-seq}
u_i \geq \chi_{K_i}\quad \text{ and } \quad \|u_i\|_{\mathcal{W}(\Omega_T)} \leq 2 \ca_{\mathrm{var}} (K_i,\Omega_T) \xrightarrow{i \to \infty} 0.
\end{equation}
Let $\bar{H}(s) \in C_0^\infty(\R, \R_{\geq 0})$ such that
$\spt\bar{H}\subset(-1,1)$ and $\int_0^1 \bar{H}(s)\, \d s = 1$. Furthermore, denote $H(s) = \int_0^s \bar{H}(t)\, \d t$.

Define $w_i = H(u_i) \in C_0^\infty(\Omega \times \R)$, which implies
$$
0\leq w_i \leq 1 \quad \text{ and } \quad w_i \geq \chi_{K_i}.
$$ 
Since $\partial_t u_i \in \mathcal{V}'(\Omega_T)$, we have the
representation $\partial_t u_i = - \Div F_i$ for some $F_i \in
L^{p'}(\Omega_T,\R^n)$ such that $\|F_i\|_{L^{p'}(\Omega_T)} \leq c(n)
\|\partial_t u_i\|_{\mathcal{V}'(\Omega_T)}$,
see~\cite[Prop.~9.20]{Brezis}. We decompose the time derivative of
$w_i$ into $\partial_tw_i=[ \partial_t w_i]_a+[ \partial_t
w_i]_b\in\mathcal{V}'(\Omega_T)+L^1(\Omega_T)$ by letting  
$$
\left[ \partial_t w_i \right]_a = \left[ \bar{H}(u_i) \partial_t u_i \right]_a = - \bar{H}(u_i) \Div F_i - \bar{H}' (u_i) \nabla u_i \cdot F_i \in \mathcal{V}'(\Omega_T) 
$$
and 
$$
\left[ \partial_t w_i \right]_b = \bar{H}' (u_i) \nabla u_i \cdot F_i \in L^1(\Omega_T).
$$
By the divergence theorem we obtain 
\begin{align*}
\left| \iint_{\Omega_T} v \left[ \partial_t w_i \right]_a \, \d x\d t \right| &= \left| \iint_{\Omega_T} \bar{H}(u_i) \nabla v \cdot F_i \, \d x\d t \right| \leq \|\bar H\|_{L^\infty} \|v \|_{\mathcal{V}(\Omega_T)} \|F_i \|_{L^{p'}(\Omega_T)} \\
&\leq c\|\bar H\|_{L^\infty} \|v \|_{\mathcal{V}(\Omega_T)} \|\partial_t u_i \|_{\mathcal{V}'(\Omega_T)}
\end{align*}
for every $v \in C_0^\infty(\Omega_T)$. By taking the supremum over all such $v$ with $\|v\|_{\mathcal{V}(\Omega_T)} \leq 1$ we obtain
$$
\|\left[ \partial_t w_i \right]_a\|_{\mathcal{V}'(\Omega_T)} \leq c\|\bar H\|_{L^\infty} \|\partial_t u_i \|_{\mathcal{V}'(\Omega_T)}.
$$
On the other hand 
\begin{align*}
\|\left[ \partial_t w_i \right]_b\|_{L^1(\Omega_T)} &\leq \|\bar{H}'\|_{L^\infty} \|u_i\|_{\mathcal{V}(\Omega_T)} \|F_i \|_{L^{p'}(\Omega_T)}\\
&\leq c \|\bar{H}'\|_{L^\infty} \|u_i\|_{\mathcal{V}(\Omega_T)} \|\partial_t u_i \|_{\mathcal{V}'(\Omega_T)}.
\end{align*}
Also clearly 
$$
\|w_i\|_{\mathcal{V}(\Omega_T)} \leq \|\bar H\|_{L^\infty}
\|u_i\|_{\mathcal{V}(\Omega_T)}.
$$
From~\eqref{eq:approx-min-seq}$_2$ and the above estimates it follows that 
$$
\|w_i\|_{\mathcal{S}(\Omega_T)} \xrightarrow{i \to \infty} 0.
$$
Moreover, decomposing
$$
\partial_t(w_i\varphi)=[\partial_tw_i]_a\varphi+\big([\partial_tw_i]_b\varphi+
w_i\partial_t\varphi\big)=:[\partial_t(w_i\varphi)]_a+[\partial_t(w_i\varphi)]_b,
$$
we have 
\begin{equation*}
  \|\left[\partial_t(w_i\varphi)\right]_a\|_{\mathcal{V}'(\Omega_T)}
  +
  \|\left[\partial_t(w_i\varphi)\right]_b\|_{L^1(\Omega_T)}
  \xrightarrow{i \to \infty} 0.
\end{equation*}
Thus $(1-w_i) \varphi \in C_0^\infty(O \setminus E)$ and 
$$
(1-w_i) \varphi \xrightarrow{i \to \infty} \varphi \quad \text{ in } \mathcal{S}(\Omega_T).
$$
\end{proof}

We use the preceding lemma to prove a removability result. We start with
a version for weak supersolutions. 

\begin{lemma} \label{lem:supersol-removability}
Let $1<p< \infty$, $E \subset \Omega_T$ be a relatively closed set and $v$ be a weak supersolution in $\Omega_T \setminus E$ such that each point $z \in E$ has a neighborhood $U \subset \Omega_T$ such that $v$ is essentially bounded in $U \setminus E$. Furthermore, suppose that $v$ is locally essentially bounded from below in $\Omega_T \setminus E$. If $\ca_{\mathrm{var}} (E, \Omega_T) = 0$, then there exists a lower semicontinuous extension $u$ which is a weak supersolution in $\Omega_T$ and $v=u$ a.e. in $\Omega_T \setminus E$.
\end{lemma}

\begin{proof}

By assumptions we can define
$$
u(x,t) = \essliminf_{\Omega_T \setminus E \ni (y,s) \to (x,t)} v(y,s) > - \infty
$$
for every $(x,t) \in \Omega_T$. Observe that $u=v$ a.e. in $\Omega_T \setminus E$ by Lemma~\ref{lem:supersol-lsc}. 

Fix $z \in E$ and take an open neighborhood $U \subset \Omega_T$ of
$z$ such that $|u| \leq M$ in $U \setminus E$. Let $O \Subset
U$. Cover $E \cap \overline{O} =:K$ by a finite union of space-time
cylinders such that the closures stay inside $U$, denote a cylinder
from this union by $Q'$ and consider $Q \Supset Q'$ such that $Q
\Subset U$. Let $\varphi \in C_0^\infty(Q, [0,1])$ s.t. $\varphi = 1$
in $Q'$ and take $\varphi_i \in C_0^\infty(Q \setminus K,[0,1])$ such
that $\varphi_i \xrightarrow{i \to \infty} \varphi$ in
$\mathcal{S}(\Omega_T)$ provided by Lemma~\ref{lem:S-closure}. 

From here, the proof in the case $p> 2$ follows with the same arguments presented in~\cite[Theorem 5.2]{AS}, thus we focus here on the case $1<p<2$.

Without loss of generality we may assume that $u \geq 1$ in $Q$ by adding a constant due to the assumptions. By Caccioppoli's inequality for nonnegative weak supersolutions,~\eqref{eq:caccioppoli-p'} in Lemma~\ref{lem:caccioppoli}, we have
\begin{align*}
\iint_Q &u^{-1-\eps} |\nabla u|^p  \varphi_i^{p'} \, \d x \d t \\
&\leq  c \iint_Q u^{p-1-\eps} \varphi_i^{p(p'-2)} |\nabla \varphi_i|^p\, \d x\d t + c \left| \iint_Q u^{1-\eps} \partial_t \varphi_i^{p'} \, \d x\d t \right| =: \mathrm{I} + \mathrm{II}.
\end{align*}
for every $\eps \in (0,1)$. Since $u \leq M$, it follows that 
$$
\iint_Q |\nabla u|^p \varphi_i^{p'} \, \d x \d t  \leq M^{1+\eps} \iint_Q |\nabla u|^p u^{-1-\eps} \varphi_i^{p'} \, \d x \d t.
$$
For any decomposition
$\partial_t \varphi_i
=
\left[ \partial_t \varphi_i \right]_a
+
\left[ \partial_t \varphi_i \right]_b
\in\mathcal{V}'(\Omega_T)+L^1(\Omega_T)$ we obtain
\begin{align*}
\mathrm{II} &= cp' \left| \iint_Q u^{1-\eps} \varphi_i^{p'-1} \left[ \partial_t \varphi_i \right]_a \, \d x\d t + \iint_Q u^{1-\eps} \varphi_i^{p'-1}\left[ \partial_t \varphi_i \right]_b \, \d x\d t \right| \\
&\leq cp' \| u^{1-\eps} \varphi_i^{p'-1} \|_{\mathcal{V}(\Omega_T)} \|\left[ \partial_t \varphi_i \right]_a\|_{\mathcal{V}'(\Omega_T)}  + c p' M^{1-\eps} \| \left[ \partial_t \varphi_i \right]_b\|_{L^1(\Omega_T)} \\
&\leq \delta \| u^{1-\eps} \varphi_i^{p'-1} \|_{\mathcal{V}(\Omega_T)}^p + c_\delta \|\left[ \partial_t \varphi_i \right]_a\|_{\mathcal{V}'(\Omega_T)}^{p'} + c M^{1-\eps} \| \left[ \partial_t \varphi_i \right]_b\|_{L^1(\Omega_T)},
\end{align*}
for every $\delta>0$. 
For the first term on the right hand side we estimate
\begin{align*}
\iint_Q |\nabla (u^{1-\eps} \varphi_i^{p'-1})|^p \, \d x \d t &\leq c(p) \iint_Q |\nabla u |^p \varphi_i^{p(p'-1)} \, \d x \d t \\
&\quad + c(p) M^{p(1-\eps)} \iint_Q |\nabla \varphi_i |^p \varphi_i^{p(p'-2)} \, \d x \d t \\
&\leq c(p) \iint_Q |\nabla u |^p \varphi_i^{p'} \, \d x \d t \\
&\quad + c(p) M^{p(1-\eps)} \iint_Q |\nabla \varphi_i |^p \, \d x \d t
\end{align*}
by using the facts $u\ge1$ in $Q$, $p' > 2$, $p(p'-1) = p'$ and $\varphi_i \leq 1$. Also,
$$
\mathrm{I} \leq c \max\{1, M^{p-1-\eps} \} \iint_Q |\nabla \varphi_i|^p \, \d x\d t.
$$ 
By fixing $\eps = \frac12$ and choosing $\delta>0$ so small that we
can absorb the integral of $|\nabla u|^p\varphi_i^{p'}$ into the
left hand side, it follows that 
$$
\iint_Q |\nabla u|^p \varphi_i^{p'} \, \d x \d t  \leq c(p,M) \max\left\{ \|\varphi_i \|_{\mathcal{S}}, \| \varphi_i \|_{\mathcal{S}}^{p'} \right\}.
$$
Thus, the integral on the left hand side is uniformly bounded with respect to $i \in \N$, and the fact that $|\nabla u|^p \in L^1(Q')$ follows by passing to a subsequence and using Fatou's lemma on the left hand side. Thus, $u \in L_{\loc}^p(0,T; W^{1,p}_{\loc}(\Omega))$.

To show that $u$ is a weak supersolution in $\Omega_T$, let $Q$ and
$Q'$ be as above. For any $\varphi \in C_0^\infty(Q,\R_{\ge0})$, we use
Lemma~\ref{lem:S-closure} to choose a
sequence $\varphi_i \in C_0^\infty(Q \setminus K,\R_{\ge0})$ such that
$\varphi_i \xrightarrow{i\to \infty} \varphi$ in
$\mathcal{S}(\Omega_T)$. The test functions $\varphi_i$ are admissible
in the weak formulation~\eqref{eq:weak-super} of the weak supersolution
$u$ in $\Omega_T\setminus E$, and we have 
\begin{align*}
\left| \iint_Q |\nabla u|^{p-2}\nabla u \cdot \nabla (\varphi_i - \varphi) \, \d x \d t \right| \leq \|u\|_{\mathcal{V}(Q)}^{p-1} \|\varphi_i - \varphi\|_{\mathcal{V}(\Omega_T)} \xrightarrow{i \to \infty} 0,
\end{align*}
and
\begin{align*}
  \left| \iint_Q u \partial_t (\varphi_i - \varphi) \, \d x \d t \right|
  &\leq \|u\|_{\mathcal{V}(Q)} \|\left[ \partial_t (\varphi_i - \varphi) \right]_a \|_{\mathcal{V}'(\Omega_T)} \\
  &\quad  + \|u\|_{L^\infty(Q)} \|\left[ \partial_t (\varphi_i - \varphi) \right]_b\|_{L^1(\Omega_T)} \xrightarrow{i \to \infty} 0.
\end{align*}
Consequently, inequality~\eqref{eq:weak-super} holds for any $\varphi
\in C_0^\infty(Q,\R_{\ge0})$. Thus, $u$ is a weak supersolution in $\Omega_T$. 
\end{proof}

By applying the preceding result to $u$ and $-u$, we immediately obtain the
following removability result for weak solutions. 

\begin{corollary} \label{cor:weak-sol-removability}
Let $1<p< \infty$, $E \subset \Omega_T$ be a relatively closed set and $v$ be a weak solution in $\Omega_T \setminus E$ such that each point $z \in E$ has a neighborhood $U \subset \Omega_T$ such that $v$ is essentially bounded in $U \setminus E$ and locally essentially bounded in $\Omega_T \setminus E$. If $\ca_{\mathrm{var}} (E, \Omega_T) = 0$, then there exists a continuous extension $u$ which is a weak solution in $\Omega_T$ and $v=u$ a.e. in $\Omega_T \setminus E$.
\end{corollary}

\begin{remark}
Observe that a weak (super)solution $v$ is locally essentially bounded (from below) in $\Omega_T \setminus E$ if $p > \frac{2n}{n+2}$ such that in this case the assumption is redundant, cf. Remark~\ref{rem:supersol-bddbelow}.
\end{remark}

\begin{remark}
If the set $E \subset \Omega_T$ is compact in Corollary~\ref{cor:weak-sol-removability}, then the removability result does not hold in general if $\ca_{\mathrm{var}} (E, \Omega_T) > 0$. This can be concluded as follows. By Lemma~\ref{lem:balayage-zero-iff-cap-zero} there exists $z_o = (x_o,t_o) \in \Omega_T$ such that $\widehat R_E(z_o) =: a > 0$, which further implies that $\widehat{R}_E > \tfrac12 a$ in $B_\rho (x_o) \times (t_o-\rho,t_o + \rho) \Subset \Omega_T$ for some $\rho > 0$.

In case $1<p<2$, Lemma~\ref{lem:supercal-slice-alt} implies that
$\widehat{R}_E (\cdot, t_o) > 0$ in the whole connected component of
$\Omega$ containing $x_o$. Since $E\subset\Omega_T$ is compact, this
implies that $\widehat{R}_E$ does not vanish identically in
$\Omega_T\setminus E$.  Since $\widehat{R}_E$ is a continuous weak solution in $\Omega_T \setminus E$ with zero boundary values on $\partial_p \Omega_T$, it follows that it cannot be extended as a continuous weak solution to $\Omega_T$.

In case $2 < p< \infty$, we construct a weak solution $h$ in
$Q:=B_\rho(x_o) \times (t_o,T)$ with initial values $\tfrac12 a \eta$
and zero lateral boundary values, where $\eta \in
C_0^\infty(B_\rho(x_o), [0,1])$ is a cutoff function with $\eta \equiv 1$ in
$B_{\frac{\rho}{2}}(x_o)$. Then, by application of
Harnack's inequality~\cite[Chapter 5, Theorem 1.1]{DGV} we have
 $h(x_o,t) > 0$ for all $t \in (t_o,T)$, and the comparison
principle implies that $\widehat{R}_E \geq h$ in $Q$. This shows that
$ \widehat{R}_E$, which is a weak solution in $\Omega_T\setminus E$
with zero boundary values on $\partial_p \Omega_T$, cannot be extended
as a continuous weak solution to $\Omega_T$.
\end{remark}

Finally, we extend the removability result to the case of supercaloric
functions, i.e. we give the \\[-1ex]

\noindent\textit{Proof of Theorem~\ref{thm:supercal-removability}. }
By assumption and definition of a supercaloric function we know that
$v$ is locally bounded from below in $\Omega_T \setminus E$ and that
for each $z \in E$ there exists a neighborhood $U \subset \Omega_T$
such that $v$ is locally essentially bounded from below in $U$. The
truncation $v_k = \min\{v,k\}$ for each $k = 1,2,...$ is a weak
supersolution by Lemma~\ref{lem:supercal-prop} and satisfies the assumptions in Lemma~\ref{lem:supersol-removability}. It also follows by Lemma~\ref{lem:meas-leq-cap} that $|E|= 0$. Thus 
$$
u_k(x,t) := \essliminf_{\Omega_T \setminus E \ni (y,s) \to (x,t)} v_k(y,s) = \essliminf_{\Omega_T \ni (y,s) \to (x,t)} v_k(y,s) 
$$
is a weak supersolution in $\Omega_T$. Observe that $u_k(x,t) = v_k(x,t)$ for every $(x,t) \in \Omega_T \setminus E$ by Lemma~\ref{lem:supercal-prop}, which implies
$$
u_k(x,t) = \essliminf_{\Omega_T \ni (y,s) \to (x,t)} u_k(y,s)
$$ 
for every $(x,t) \in \Omega_T$. It follows from
  Lemma~\ref{lem:supersol-supercal} that $u_k$ is a supercaloric function in $\Omega_T$. The function $u
:= \lim_{k \to \infty} u_k$ is an increasing limit of supercaloric
functions, and it is finite in a dense subset of $\Omega_T$ since $\ca_{\mathrm{var}} (E,\Omega_T) = 0$. Thus $u$ is a supercaloric function in $\Omega_T$ by~\cite[Proposition 5.1]{KKP}.
\hfill \qed

\appendix

\section{Different notions of variational capacity} \label{appendix}

In this section we analyze two variants of the variational capacity
introduced in Definition~\ref{def:varcap}.

\begin{definition} \label{def:varcap'}
For a compact set $K \subset \Omega_T$, $0<T\le\infty$, we define
\begin{align*}
&\ca_{\text{var}}' (K,\Omega_T) \\
&\phantom{+}= \inf \left\{ \|v\|_{\mathcal W(\Omega_T)}: v \in \mathcal W(\Omega_T),\, v \geq \chi_U \text{ a.e. in $\Omega_T$ for some open } U \supset K \right\}.
\end{align*}
The capacity $\ca_{\text{var}}'$ for open and arbitrary sets is defined as in Definition~\ref{def:varcap}.
\end{definition}

\begin{definition} \label{def:varcap*}
For an open set $U \subset \Omega_T$, $0<T\le\infty$, we define 
\begin{align*}
\ca_{\text{var}}^* (U,\Omega_T)= \inf \left\{ \|v\|_{\mathcal W(\Omega_T)}: v \in \mathcal W(\Omega_T),\, v \geq \chi_U \text{ a.e. in } \Omega_T \right\}.
\end{align*}
For an arbitrary set $E \subset \Omega_T$ we let
$$
\ca_{\text{var}}^* (E,\Omega_T) = \inf \left\{ \ca_{\text{var}}^* (U,\Omega_T),  U \text{ open subset of } \Omega_T, E \subset U \right\}.
$$
\end{definition}

First, we observe that these two notions of capacity are equivalent.
\begin{lemma} \label{lem:cap*-cap'-eq}
Let $U \subset \Omega_T$ be an open set. Then,
$$
\ca_{\mathrm{var}}^* (U,\Omega_T) = \ca_{\mathrm{var}}' (U,\Omega_T).
$$
\end{lemma}

\begin{remark}
The result for open sets implies that it actually holds for all sets.
\end{remark}

\begin{proof}

We first show "$\geq$". Without loss of generality we may suppose that the left hand side is finite, which clearly implies that the right hand side is finite as well.
For every $\eps > 0$ there exists $v \in \mathcal W(\Omega_T)$ with $v \geq 1 \text{ a.e. in } U $ such that 
$$
\|v\|_{\mathcal W(\Omega_T)} \leq \ca_{\text{var}}^* (U,\Omega_T) + \eps.
$$
Let $K \subset U$ be a compact set. Since $U$ is a neighborhood of $K$, it follows that 
$$
\ca_{\mathrm{var}}' (K,\Omega_T) \leq \|v\|_{\mathcal W(\Omega_T)}.
$$
By combining the estimates, letting $\eps \to 0$ and taking the supremum over $K \subset U$ the claim follows.

Then we consider the inequality "$\leq$". Without loss of generality
we assume that the right hand side is finite. Let $K_i  = \left \{ z \in U :  \mathrm{dist  } (z, \partial U) \geq \frac{1}{i}\right\}$. Then, there exists a sequence of functions $v_i \in \mathcal{W}(\Omega_T)$ such that $v_i \geq 1$ a.e. in a neighborhood of $K_i$ and
$$
\|v_i\|_{\mathcal{W}(\Omega_T)} \leq \ca_{\mathrm{var}}' (K_i,\Omega_T) + \tfrac{1}{i} \leq \ca_{\mathrm{var}}' (U,\Omega_T) + \tfrac{1}{i} < \infty.
$$
Thus, there exist $v \in \mathcal{W}(\Omega_T)$ and a (nonrelabeled)
subsequence $(v_i)_{i \in \N}$ such that
\begin{align} \label{eq:W-convergences}
\begin{aligned}
\left\{
\begin{array}{cl}
v_i \wto v \quad &\text{ weakly in } \mathcal{V}(\Omega_T), \\[0.8ex]
\partial_t v_i \wto \partial_t v \quad &\text{ weakly in } \mathcal{V}'(\Omega_T), \\[0.5ex]
v_i \wsto v \quad &\text{ weakly-* in } L^\infty(0,T;L^2(\Omega)),
\end{array}
\right.
\end{aligned}
\end{align}
in the limit $i\to\infty$. 
Furthermore, the norms $\|\cdot\|_{\mathcal{V}(\Omega_T)}, \|\cdot\|_{\mathcal{V}'(\Omega_T)}$ and $\|\cdot\|_{L^\infty(0,T;L^2(\Omega))}$ are lower semicontinuous with respect to the convergences in~\eqref{eq:W-convergences}.
Observe that $v \geq 1$ a.e. in $U$, and thus
$$
\ca_{\mathrm{var}}^* (U,\Omega_T) \leq \|v\|_{\mathcal{W}(\Omega_T)} \leq \ca_{\mathrm{var}}' (U,\Omega_T),
$$
which concludes the proof.

\end{proof}

Then we state that $\ca_{\mathrm{var}}'$ is continuous with respect to decreasing limits of compact sets. The proof is analogous to the proof of Lemma~\ref{lem:capvar-limit-compact}.

\begin{lemma} \label{lem:cap'-limit-compact}
Let $1<p<\infty$, $0< T\leq \infty$ and $(K_i)_{i\in \N}$ be a sequence of compact sets contained in $\Omega_T$ satisfying $K_i \supset K_{i+1}$ for every $i \in \N$. Then
$$
\lim_{i \to \infty} \ca_{\mathrm{var}}' (K_i, \Omega_T) =  \ca_{\mathrm{var}}' (\cap_{i=1}^\infty K_i, \Omega_T).
$$
\end{lemma}

Next, we prove that $\ca_{\mathrm{var}}'$ is equivalent to the
variational capacity from Definition~\ref{def:varcap}.

\begin{lemma} \label{lem:cap-cap'-equiv}
Let $K \subset \Omega_T$, $0<T\le\infty$, be a compact set. Then
$$
 \ca_{\mathrm{var}}' (K,\Omega_T) \leq \ca_{\mathrm{var}} (K,\Omega_T) \leq c  \ca_{\mathrm{var}}' (K,\Omega_T),
$$
where $c = c(n,p,\alpha) > 0$.
\end{lemma}

\begin{proof}

We first prove the first inequality. Suppose that the right hand side is finite. By definition for every $\eps > 0$ there exists $v \in C_0^\infty(\Omega \times \R)$ with $v \geq \chi_K$ such that
$$
\|v\|_{\mathcal{W} (\Omega_T)} \leq \ca_{\mathrm{var}} (K,\Omega_T) + \eps.
$$
Let $\gamma > 1$. It follows that $\gamma v \geq 1$ in a neighborhood of $K$ such that
$$
\ca_{\mathrm{var}}' (K,\Omega_T) \leq \|\gamma v\|_{\mathcal{W} (\Omega_T)} \leq \gamma^s \|v\|_{\mathcal{W} (\Omega_T)},
$$
where $s = \max\{p,p'\}$. By combining the estimates and letting $\gamma \to 1$ and $\eps \to 0$, respectively, the claim follows.

Then we turn our attention to the second inequality and suppose that the right hand side is finite. By definition, for every $\eps  >0$ there exists $v \in  \mathcal{W} (\Omega_T)$ such that $v\geq 1$ in a neighborhood of $K$ and
$$
\|v\|_{\mathcal{W} (\Omega_T)} \leq \ca_{\mathrm{var}}' (K,\Omega_T) + \eps.
$$
In case $T < \infty$, suppose that $K$ is a finite union of space-time
cylinders, whose bases are balls, as in the proof of
Lemma~\ref{lem:W-approx}. As in the aforementioned lemma, we use a reflection argument to extend the function $v$ to $\mathcal{W}(\Omega \times (-T,2T) )$, and define $v = 0$ outside $\Omega\times [-T,2T]$.
We take a cutoff function $\zeta_\eps \in C_0^\infty(\Omega, \R_{\geq 0})$ in space and mollify $\zeta_\eps v$ such that $(\zeta_\eps v)_\delta = \zeta_\eps v * \eta_\delta \in C_0^\infty(\Omega \times \R)$ with $5 (\zeta_\eps v)_\delta \geq \chi_K$ for small enough $\delta > 0$. As in Lemma~\ref{lem:W-approx}, we have that 
$$
\ca_{\mathrm{var}} (K,\Omega_T)\leq \|5(\zeta_\eps v)_\delta\|_{\mathcal{W} (\Omega_T)} \leq c (n,p,\alpha) \|v\|_{\mathcal{W} (\Omega_T)}.
$$
By combining the estimates and letting $\eps \to 0$ the claim follows in case $T < \infty$ and when $K$ is a finite union of space-time cylinders. By using Lemmas~\ref{lem:capvar-limit-compact} and~\ref{lem:cap'-limit-compact} the claim holds for arbitrary compact set $K \subset \Omega_T$ for $T < \infty$. On the other hand, we have
$$
\ca_{\mathrm{var}} (K,\Omega_\infty) \leq c \ca_{\mathrm{var}} (K,\Omega_T) \leq c \ca_{\mathrm{var}}' (K,\Omega_T) \leq c \ca_{\mathrm{var}}' (K,\Omega_\infty),
$$
by using also Lemma~\ref{lem:capvar-Tinfty2} and the argument above, which concludes the proof.

\end{proof}

For $\ca_{\mathrm{var}}^*$ we have the following continuity property
with respect to increasing limits of sets. 

\begin{lemma} \label{lem:cap*-increasinglim}
Let $1<p<\infty$, $0<T\leq \infty$ and $E_i \subset \Omega_T$ be sets satisfying $E_i \subset E_{i+1}$ for every $i \in \N$. Then
$$
\ca_{\mathrm{var}}^* \left( \bigcup_{i=1}^\infty E_i, \Omega_T \right) =\lim_{i \to \infty} \ca_{\mathrm{var}}^* \left(E_i, \Omega_T \right).
$$
\end{lemma}

\begin{remark} \label{rem:cap*-increasinglim}
Lemma~\ref{lem:cap*-cap'-eq} implies that the result holds also for $\ca_{\mathrm{var}}'$ and Lemma~\ref{lem:cap-cap'-equiv} that it holds for $\ca_{\mathrm{var}}$ up to a constant depending on $n,p$ and $\alpha$.
\end{remark}

\begin{proof}
The inequality "$\geq$" is clear by monotonicity. We prove the reverse direction and w.l.o.g. assume that the right hand side is finite. By definition there exists an open set $U_i$ such that $E_i \subset U_i \subset \Omega_T$ and 
$$
\ca_{\mathrm{var}}^* \left(U_i, \Omega_T \right) \leq \ca_{\mathrm{var}}^* \left(E_i, \Omega_T \right) + \tfrac{1}{2i}
$$
for every $i \in \N$.
Furthermore, for each $U_i$ there exists $v_i \in \mathcal{W}(\Omega_T)$ with $v_i \geq 1$ a.e. in $U_i$ such that
$$
\|v_i\|_{ \mathcal{W}(\Omega_T)} \leq \ca_{\mathrm{var}}^* \left(U_i, \Omega_T \right) + \tfrac{1}{2i}.
$$
Thus, as in the proof of Lemma~\ref{lem:cap*-cap'-eq} there exist $v \in \mathcal{W}(\Omega_T)$ and a (nonrelabeled) subsequence $(v_i)_{i\in \N}$ such that the convergences in~\eqref{eq:W-convergences} hold, and furthermore, $v \geq 1$ a.e. in $\bigcup_{i=1}^\infty U_i$. Observe that $\bigcup_{i=1}^\infty U_i \supset \bigcup_{i=1}^\infty
E_i$. By definition of capacity, weak lower semicontinuity of the norm
and the estimates above we obtain
\begin{align*}
\ca_{\mathrm{var}}^* \left( \bigcup_{i=1}^\infty E_i, \Omega_T \right) \leq \|v\|_{\mathcal{W}(\Omega_T)} \leq \liminf_{i\to \infty} \|v_i\|_{\mathcal{W}(\Omega_T)} \leq \lim_{i\to \infty} \ca_{\mathrm{var}}^* \left(E_i, \Omega_T \right),
\end{align*}
which concludes the proof.

\end{proof}

Concerning subadditivity, we have the following result for a certain
power of the variational capacity. 

\begin{lemma}
Let $1<p<\infty$ and $0<T\leq \infty$. Then, $(\ca_{\mathrm{var}}')^s$
satisfies the countable subadditivity property, where $s = (\max\{p,p'\})^{-1}$. That is, if $E_i \subset \Omega_T$ is a set for every $i \in \N$, then
$$
(\ca_{\mathrm{var}}')^s \left(\bigcup_{i=1}^\infty E_i , \Omega_T\right) \leq \sum_{i =1}^\infty (\ca_{\mathrm{var}}')^s ( E_i , \Omega_T).
$$
\end{lemma}

\begin{remark}
By Lemma~\ref{lem:cap*-cap'-eq} the result holds also for $\ca_{\mathrm{var}}^*$.
\end{remark}

\begin{proof}
Analogous to the proof of Lemma~\ref{lem:subadditivity}.
\end{proof}

Monotonicity and Lemmas~\ref{lem:cap'-limit-compact} and~\ref{lem:cap*-increasinglim} (see also Remark~\ref{rem:cap*-increasinglim}) imply that $\ca_{\mathrm{var}}'$ and $\ca_{\mathrm{var}}^*$ are Choquet capacities. This implies that all Borel sets are capacitable, see e.g.~\cite[Chapter 2, Theorem 2.5]{HKM} and references therein.

\begin{proposition} \label{prop:capvar-choquet}
Let $1<p<\infty$, $0<T\leq \infty$ and $B \subset \Omega_T$ be a Borel set. Then
$$
\ca_{\mathrm{var}}^i (B, \Omega_T) = \sup \left\{ \ca_{\mathrm{var}}^i (K, \Omega_T), K \text{ compact}, K \subset B \right\}.
$$
for $i \in \{',*\}$. For $\ca_{\mathrm{var}}$, the above property holds up to a constant depending on $n,p$ and $\alpha$.
\end{proposition}

\end{document}